\documentclass[11pt,article,leqno]{amsart}
\usepackage{enumerate}
\usepackage{amsmath,amssymb,amsthm,amsfonts}
\usepackage[T1]{fontenc}
\usepackage{tikz}
\usepackage[ruled,vlined,commentsnumbered]{algorithm2e}
\usepackage{xspace}
\usepackage{url}
\usepackage{enumerate}
\usepackage{hyperref}
\usepackage{fullpage}
\usepackage{subcaption}
\usepackage{graphicx}
\usepackage[ruled]{algorithm2e}
\usepackage{graphicx, color}
\usepackage{cite}
\theoremstyle{plain}

\usepackage{graphicx, color}
\usepackage{enumitem}
\usepackage{tikz}
\usepackage{tkz-euclide}
\usetikzlibrary{positioning}
\usetikzlibrary{arrows.meta}
\usetikzlibrary {decorations.pathmorphing, decorations.text}

\usepackage[pagewise]{lineno}

\hypersetup{
  colorlinks,
  linkcolor={red!70!black},
  citecolor={blue!50!black},
  urlcolor={blue!80!black}
}

\newtheorem{question}{Question}

\newtheorem{theorem}{Theorem}
\newtheorem{lemma}{Lemma}
\newtheorem{proposition}{Proposition}
\newtheorem{corollary}{Corollary}

\theoremstyle{definition}
\newtheorem{example}{Example}
\newtheorem{remark}{Remark}

\newtheorem{definition}{Definition}

\newtheorem{claim}{Claim}

\numberwithin{equation}{section}

\newcommand{\cg}{\ensuremath{\mathfrak{g}}\xspace}
\newcommand{\ccG}{\ensuremath{\mathfrak{G}}\xspace}
\newcommand{\cm}{\ensuremath{\mathfrak{m}}\xspace}
\newcommand{\cM}{\ensuremath{\mathfrak{M}}\xspace}

\newcommand{\SK}{\ensuremath{/\hspace*{-0.1cm}/}\xspace}

\newcommand{\TC}{\mathrm{TC}}

\newcommand{\poly}{\ensuremath{{\rm poly}}}

\newcommand{\Max}{{\rm Max}}

\DeclareMathOperator{\vcd}{VC-dim}

\definecolor{color1}{RGB}{255,128,0}

\usepackage[left=1in,right=1in,top=1in,bottom=1in,foot=0.5in]{geometry}

\newcommand{\cH}{\ensuremath{\mathcal{H}}\xspace}
\newcommand{\cK}{\ensuremath{\mathcal{K}}\xspace}

\newcommand{\cS}{\ensuremath{\mathcal{S}}\xspace}

\DeclareMathOperator{\Imp}{\iota}

\DeclareMathOperator{\Proj}{\pi}
\DeclareMathOperator{\conv}{\frak c}
\DeclareMathOperator{\C}{\frak C}
\DeclareMathOperator{\Lk}{Lk}

\newtheorem{case-int}{Case}
\newtheorem{case-pc}{Case}
\newtheorem{case-x|y}{Case}
\newtheorem{case-x|k}{Case}
\newtheorem{case-k|x}{Case}
\newtheorem{case-ext}{Case}
\newtheorem{case-Wa}{Case}
\newtheorem{case-CG}{Case}
\newtheorem{case-A5}{Case}
\newtheorem{subcase-int}{Subcase}
\newtheorem{subcase-pc}{Subase}
\newtheorem{subcase-x|y}{Subcase}
\newtheorem{subcase-x|k}{Subcase}
\newtheorem{subcase-ext}{Subcase}
\newtheorem{subsubcase-x|k}{Subsubcase}
\newtheorem{subsubcase-k|x}{Subsubcase}

\SetEndCharOfAlgoLine{}
\SetArgSty{textnormal}

\tikzset{vertex/.style={
  fill=black,
  draw=black,
  circle,
  inner sep=1pt,
  minimum size = 3pt
  }
}

\begin{document}
	
	\centerline{\bf\Large Separation axiom $S_3$ for geodesic convexity in graphs}

	\bigskip
	\centerline{\sc \large Victor Chepoi}
	
	\bigskip
	\centerline{LIS, Aix-Marseille Universit\'e, Facult\'e des Sciences de Luminy,} 
	\centerline{F-13288 Marseille Cedex 9,
	France}
	\centerline{ {\sf victor.chepoi@lis-lab.fr} }

%

\medskip\noindent
{\bf Abstract.}  {\footnotesize Semispaces of a convexity space $(X,\C)$ are maximal convex sets  missing a point. 
The separation axiom $S_3$ asserts that any point $x_0\in X$ and any convex set $A$ not containing $x_0$
can be separated by complementary halfspaces (convex sets with convex complements) or, equivalently, that all semispaces are halfspaces. 
The characterization of semispaces and the  $S_3$-axiom in linear spaces are classical results in convexity.
In this paper, we study $S_3$ for geodesic convexity in graphs and the structure of semispaces in $S_3$-graphs.
We characterize $S_3$-graphs and their semispaces in terms of separation  by halfspaces of vertices $x_0$ and special sets, called maximal $x_0$-proximal sets and in terms of convexity of their
mutual shadows $x_0/K$ and $K/x_0$.
In $S_3$-graphs $G$ satisfying the triangle condition (TC),
 maximal proximal sets are the pre-maximal cliques of $G$ (i.e., cliques $K$ such that $K\cup\{ x_0\}$ are maximal cliques). This allows to characterize  the $S_3$-graphs satisfying (TC) in a structural way and to enumerate their semispaces efficiently. 
In case of meshed graphs (an important subclass of graphs satisfying (TC)), the $S_3$-graphs have been characterized by excluding five forbidden subgraphs.
On the way of proving this result, we also establish some properties of  meshed graphs, which maybe of independent interest. In particular,
we show that any connected, locally-convex set of a meshed graph is convex. We also provide several examples of $S_3$-graphs, including  the basis graphs of matroids.
Finally, we consider the (NP-complete) halfspace separation problem, describe two methods of its solution, and apply them to particular classes of graphs and graph-convexities.
}

\tableofcontents

\bigskip\bigskip
\section{Introduction}\label{s:introduction}

\subsection*{History of the separation axioms} Halfspaces  and semispaces  are two relevant types of convex sets in linear spaces. Later, they have been
generalized and  studied in the theory of abstract convexity. A \emph{halfspace} 
is a convex set with a convex complement and a \emph{semispace} (called also a copoint  and a hypercone)
is a maximal by inclusion convex set missing a point.  Halfspaces 
in linear spaces
arise in separation theorems, which are fundamental mathematical results with numerous applications, in particular, in optimization and machine learning.
One of the first separation results in ${\mathbb R}^d$ is Farkas’s lemma \cite{Fa}, which is
motivated by separating a point from a convex polyhedral cone and is at the heart of linear optimization. Extensions are the well-known hyperplane
separation theorem by Minkowski \cite{Mi}, which states
that disjoint convex sets are always separable by hyperplanes, and the Hahn-Banach theorem \cite{Ba,Hahn} in topology.
Kakutani \cite{Ka} and Tukey \cite{Tu} proved the following separation theorem in linear spaces $L$: if $A$ and $B$ are disjoint convex sets, then there exists a halfspace
$H$ such that $A\subseteq H$ and $B\subseteq L\setminus H$ (Stone \cite{St} proved a similar result for ideals in distributive lattices). For a presentation of various separation theorems
in linear spaces, see the books by K\"othe  \cite{Ko} and Valentine \cite{Va}. 
Semispaces have been independently introduced by Hammer \cite{Ha}, K\"othe \cite{Ko}, and Motzkin \cite{Mo} (for 3-dimensional Euclidean space) and their initial motivations were again the separation theorems.
On the other hand, the semispaces form the minimal intersection base for convex sets, i.e., any convex set
is the intersection of semispaces and any family of convex sets satisfying this intersection property contains all semispaces. Semispaces also correspond to meet-irreducible elements of the lattice of
convex sets.
As intersection bases of a convexity/closure space or meet-irreducibles of a lattice,  semispaces have found applications in artificial intelligence \cite{Kh} and lattice theory \cite{DaPr}.  
Hammer \cite{Ha}, Klee \cite{Klee}, and  K\"othe \cite{Ko} presented several characterizations of semispaces in linear spaces. By these results, the semispaces at $x_0$ are the maximal
convex cones which exclude their vertex $x_0$ (which justifies the term ``hypercone'' used in \cite{Ko}). 
More recently, Edelman and Jamison \cite{EdJa} and Jamison \cite{Ja,Ja2,Ja_hyper,Ja3} used the term ``copoint'' (originating from matroid theory) to designate semispaces; this is the accepted term now.
Since  we also deal with halfspaces,  we prefer to use the original term ``semispace'', used in  \cite{Ha} and  \cite{Klee}.
The period when most of previous papers on halfspaces and semispaces 
appeared is essentially the same as the period when the axiomatic approach to convexity was introduced by Levi \cite{Le}.
Ellis \cite{El} was the first to consider the Kakutani separation axiom in convexity spaces. He formulated the axiom of Join-Hull Commutativity (JHC) and proved that under JHC,
the Kakutani separation property is equivalent to the Pasch axiom from geometry. The Pasch axiom and the Peano axiom
(which is equivalent to JHC) are the principal axioms 
of join geometries, investigated in the book by Prenowitz and Jantosiak \cite{PrJa}.
A systematic study of various aspects of convexity spaces
started in the 70ths of the last century (with the papers by Calder \cite{Ca}, Ekhoff \cite{Ek}, Kay and Womble \cite{KaWo}, Hammer \cite{Ha,Ha2}, and the dissertation of Jamison \cite{Ja}).
Convexity spaces endowed with topology have been investigated by van de Vel \cite{VdVtopo}. Starting from the 80th also various models of discrete convexity, in particular, convex geometries \cite{EdJa} (alias greedoids \cite{KoLoSch}) and geodesic
convexity in graphs \cite{Ch_thesis}, have been investigated (for an earlier survey, see \cite{Du_discrete}).  For an exposition of the theory of convexity spaces see the books by  Soltan \cite{So} and van de Vel \cite{VdV}.
In analogy to well-known separation axioms $T_1-T_4$ in set-theoretical topology, Jamison \cite{Ja} introduced the separation axioms $S_2,S_3,$ and $S_4$ for convexity spaces; they
concern the separation of pairs of points, points and convex sets, and pairs of disjoint convex sets by complementary halfspaces, respectively. The $S_4$-axiom is also called the ``Kakutani separation axiom''.
The $S_3$-axiom is equivalent to the fact that each semispace is a halfspace, and thus to the fact that each convex set is an intersection of halfspaces.  Jamison \cite{Ja} proved that for
domain-finite convexity spaces, $S_4$ is equivalent to the separation by complementary halfspaces
of any two disjoint polytopes (convex hulls of finite points) and
Chepoi \cite{Ch_S3,Ch_thesis,Ch_separa} proved that for convexity spaces of arity $n$, $S_4$ is equivalent to  the separation by complementary halfspaces
of any two disjoint $n$-polytopes (convex hulls of  $n$ points). In the important case of arity $n=2$ (covering the geodesic convexity in metric spaces),
by this result, $S_4$ is equivalent to the Pasch axiom. 
Since JHC-convexity spaces have arity 2, this is a far-reaching generalization of Ellis's result. 
Similar kind of characterizations of $S_3$-convexity spaces are missing.

Separations axioms have been also investigated for particular types of convexities, in particular, for several other types of convexities in linear spaces and for geodesic and induced path (monophonic) convexities
in graphs. Soltan \cite{So_normed,So}
characterized the finite-dimensional normed vector spaces with $S_4$ geodesic convexity (induced by the norm) as the normed spaces which are direct sums of strongly normed or 2-dimensional subspaces.
He also presented an example of a 3-dimensional normed space with $S_3$ but not $S_4$ convexity. The convex sets in $H$-convexity (an
important notion of convexity  in ${\mathbb R}^d$), are intersections of certain halfspaces (see for example the book \cite{So}), thus $H$-convexity is an $S_3$-convexity.
Ordered incidence geometries \cite{BTBI} and max-plus convexity \cite{BrHo1,BrHo2} are other examples of $S_4$-convexities. Semispaces of max-plus convexity have been investigated by Nitica and Singer \cite{NiSi1,NiSi2}.
Using the equivalence between $S_4$ and the Pasch axiom in case of geodesic convexity in graphs,
$S_4$-weakly modular graphs have been characterized in terms of forbidden subgraphs \cite{Ch_local,Ch_separa} 
and $S_4$-bipartite graphs have been characterized in terms of forbidden partial cube minors
\cite{Ch_bipartite, Ch_separa}; this equivalence was also used in \cite{SeHoWr} to prove that outerplanar graphs are $S_4$. Chordal graphs with $S_3$ geodesic convexity have been characterized in \cite{SoCh_chordal} in terms of
forbidden graphs. Finally, it was shown in \cite{AlKn,Ba_intrinsic,Ch_thesis,Ch_bipartite}
that $S_3$-bipartite graphs are exactly the partial cubes, i.e., the graphs that can be isometrically embedded into hypercubes. Bandelt \cite{Ba_intrinsic} characterized graphs in which the induced path convexity is $S_2$, $S_3$, or $S_4$ in terms of forbidden
subgraphs (for $S_3$, three of the six forbidden subgraphs for induced path convexity are also forbidden for $S_3$ for geodesic convexity, see the first three graphs from Figure \ref{non-S3} below).

More recently, semispaces and  halfspaces have found applications in Computer Science and Discrete Mathematics. As we already noted, the separation theorems in linear spaces
have important applications in optimization. Such separation results are also heavily used in the
machine learning algorithms, such as, the Perceptron algorithm \cite{MiPa} and the Support Vector Machines (SVMs) \cite{BoGuVa}. To learn concepts or classify data in metric spaces, Hein et al. \cite{HeBouSch} generalize SVMs to arbitrary metric spaces by defining  separation by means of (isometric or low-distortion) embeddings of metric spaces into Banach and Hilbert spaces.
They prove that the SVM algorithm works for all metric spaces that can be isometrically embedded into a Hilbert space. Since such embeddings do not always exist, some authors develop learning algorithms,
based on separation, directly in the original metric space. In particular, Seiffarth, Horváth, and Wrobel \cite{SeHoWr} introduced the halfspace separation problem in convexity spaces, asking to decide if two disjoint
convex sets $A,B$ are separable by complementary halfspaces and to compute these halfspaces if they exist. They proved that this problem is NP-complete already for geodesic convexity in graphs (see also the thesis \cite{Se}
for an overview of this research area and \cite{ThGa} for the use of halspaces in graphs in machine learning).  Under the separation axiom $S_3$ in convexity spaces, Moran and  Yehudayoff \cite{MoYe} prove  the existence of weak $\epsilon$-nets (a notion which is used
in PAC-learning, combinatorics,  and discrete geometry). Under $S_3$ axiom for geodesic convexity in graphs, Small \cite{Sm}  characterizes and investigates  the Tukey depth medians in $S_3$-graphs
(this notion also occur in  \cite{Se}). Halfspaces and hyperplanes are also important in the theory of CAT(0) cube complexes \cite{Sa} and Coxeter groups \cite{AbBr,Da}
in geometric group theory and of median graphs in  metric graph theory \cite{BaCh_survey}. Geodesic convexity in graphs found also applications in metric graph theory \cite{BaCh_survey} and in incidence geometry
\cite{Shu}.  Finite convexity spaces  found numerous applications in formal concept analysis, database theory, and
propositional logic. Since semispaces form the minimal intersection base, they are used for  compact representation of  finite convexity spaces and their efficient enumeration is an important algorithmic question \cite{Kh}.

\subsection*{Our results and motivation} In this paper, we investigate the separation property $S_3$ for geodesic convexity in graphs and the structure of semispaces in $S_3$-graphs. We characterize the $S_3$-graphs
in terms of separation by halfspaces of vertices and convex sets of special form, called maximal proximal sets. We also characterize semispaces of $S_3$-graphs in terms of convexity of point-shadows of maximal proximal sets. 
While the structure of  maximal proximal sets is quite general, we show that in $S_3$-graphs $G$ satisfying the triangle condition (TC) the maximal proximal sets are precisely
the maximal cliques of $G$ minus a vertex. This allows us to design a polynomial algorithm (in the number of vertices of $G$ and the number of semispaces)  for enumerating semispaces of $S_3$-graphs satisfying (TC) (notice that their
number can range from logarithmic to exponential). In case of meshed graphs (an important and general  subclass of graphs satisfying (TC)), we further specify the characterization of $S_3$-graphs and  show that they are exactly the meshed graphs not containing five
forbidden subgraphs. This immediately leads us to a polynomial time recognition of meshed $S_3$-graphs. To prove this result, we establish some properties of meshed graphs, which maybe of independent interest. In particular,
we show that any connected, locally-convex set of a meshed graph is convex, an analog of Nakajima-Tietze theorem for Euclidean convexity. Finally, we provide several examples of $S_3$-graphs and meshed $S_3$-graphs, from which one can build
larger examples of $S_3$-graphs by taking Cartesian products and gated amalgamations. In particular, we show that the important class of basis graphs of matroids is a subclass of meshed $S_3$-graphs. We also show that three classes
of planar graphs of combinatorial nonpositive curvature ((3,6)-,(4,4), and (6,3)-graphs) are $S_3$-graphs. Finally, we consider the halfspace separation problem, describe two methods of its solution, and apply them to particular classes of graphs and graph-convexities.

The triangle condition (TC) is a metric condition on a graph $G$, which requires that for any vertex $u$ and any edge $vw$ such that $u$ has the same distance $k$ to
$v$ and $w$, there exists a common neighbor $x$ of $v,w$ at distance $k-1$ from $u$. Solely (TC) is too weak to have any impact on the structure of $G$ (all bipartite
graphs satisfy (TC)), however combined with other similar conditions 
it implies that the triangle-square complex of the graph $G$ is simply connected and the cycles of $G$ can be paved with triangles and squares in a  structured way. The quadrangle
condition (QC) requires  that for any vertex $u$ and pair of vertices $v,w$ at distance 2 such that $u$ has the same distance $k$ to
$v$ and $w$ and has distance $k+1$ to a common neighbor $z$ of $v$ and $w$, then $v$ and $w$ have a common neighbor $x$ having distance $k-1$ to $u$. Finally, the weak quadrangle
condition (QC$^-$) requires that  for any vertex $u$ and pair of vertices $v,w$ at distance 2, $v$ and $w$ have a common neighbor $x$ such that $2d(u,x)\le d(u,v)+d(u,w)$.
The graphs satisfying (TC) and (QC) have been called  weakly modular graphs \cite{BaCh_helly,Ch_metric} and  the graphs satisfying (QC$^-$) have been called meshed graphs \cite{BaMuSo}.
Weakly modular graphs are meshed and meshed graphs satisfy (TC). Meshedness can be viewed as a kind of convexity condition of the distance function (essential in CAT(0) spaces \cite{BrHa}).

Several important classes of graphs studied in metric graph theory and occurring in geometric group theory are weakly modular: those are median  graphs, bridged graphs, and Helly graphs.
These classes have a rich combinatorial, metric, and geometric structure and have been characterized in a multitude of ways.
Median graphs are the bipartite weakly modular graphs in which the vertex $x$ in (QC) is unique. Topologically,
median graphs are exactly the 1-skeleta of CAT(0) cube complexes \cite{Ch_CAT,Ro}, which are central objects in geometric group theory
\cite{Bo,Sa,Sa_survey}.  They have been characterized by Gromov \cite{Gr} in a local-to-global way  as simply connected cube
complexes in which the links of vertices are simplicial flag complexes. Hyperplanes and halfspaces play an important role in the theory of CAT(0) cube complexes and median graphs \cite{Sa}.
Bridged graphs are the graphs in which all isometric cycles have length 3 and they are exactly the graphs in which
all balls around convex sets are convex \cite{FaJa,SoCh}. It was shown in \cite{Ch_CAT} that bridged graphs are exactly the 1-skeleta of simply connected simplicial flag complexes
in which the links of vertices are 6-large (do not contain induced 4- and 5-cycles). Later this class of simplicial complexes was rediscovered and investigated in \cite{JaSw} and dubbed
systolic complexes. Systolic complexes satisfy many global properties of CAT(0) spaces (contractibility, fixed point property)
and were suggested in \cite{JaSw} as a variant of simplicial complexes of combinatorial nonpositive curvature (for their generalization, see \cite{ChOs}).  Finally, Helly graphs are the graphs in which the
family of all balls satisfy the Helly property (i.e., any collection of pairwise intersecting balls has a non-empty intersection). They are discrete analogs of injective metric spaces \cite{ArPa} and
have been metrically characterized in \cite{BaPr}.  In \cite{CCHO}, Helly graphs have been characterized in a local-to-global way as the graphs whose clique complexes are simply connected and the maximal
cliques satisfy the Helly property. One important distinguishing property of Helly graphs (similar to injective spaces) is that any graph isometrically embeds into a unique smallest Helly graph (in case of
injective spaces, this is called the injective hull or the tight span of the metric space and its existence has been discovered in \cite{Dr,Is_inj}). For other classes of weakly modular graphs, see the survey
\cite{BaCh_survey} and the paper \cite{CCHO}. 
Finally, in \cite{CCHO} a Cartan-Hadamard-like local-to-global characterization of all weakly modular graphs was given: they are the graphs with simply connected triangle-square complexes
and in which (TC) and (QC) are satisfied locally (when the vertices $u,v,w$ belong to balls of radius 2 or 3).

It was shown in \cite{CCHO}  that a similar local-to-global characterization is  impossible for meshed graphs. On the other hand, meshed graphs comprise several important classes of graphs which
are not weakly modular: it was shown in \cite{Ch_delta} that basis graphs of matroids and of even
$\Delta$-matroids are meshed. Matroids constitute a fundamental combinatorial structure and remain an active research domain since their definition by Whitney in 1935. Matroids can be equivalently defined in a multitude of ways;
in terms of bases they are defined as set-systems satisfying the basis exchange axiom: for any bases $A,B$ and any $a\in A\setminus B$ there exists $b\in B\setminus A$ such that $A\setminus \{ a\}\cup \{ b\}$ is a base.
Even $\Delta$-matroids \cite{Bou,ChKa,DrHa} are defined by replacing the set difference by the symmetric set difference and requiring that all bases have the same parity. Even $\Delta$-matroids generalize classical matroids
and are fundamental examples of Coxeter matroids \cite{BoGeWh} and jump systems \cite{BouCu,Lo}.  The vertices of a basis graph are the basis of a matroid or an even $\Delta$-matroid and two such bases are adjacent
if they differ in two elements. Basis graphs are exactly the 1-skeleta of the basis polytope, which is the usual convex hull of (0,1)-vectors corresponding to bases \cite{BoGeWh}. Basis graphs of matroids have been
nicely characterized in a metric way by Maurer \cite{Mau}.  One of his principal conditions is the Positioning Condition (PC), which can be viewed as a stronger version of (QC$^-$).
Maurer's theorem was extended in \cite{Ch_delta} to all even $\Delta$-matroids. Answering a conjecture by Maurer, a local-to-global characterization of basis graphs of matroids and of even $\Delta$-matroids
was given in \cite{ChChOs_bgm}: basis graphs are exactly the graphs with simply connected triangle-square complexes and which satisfy the metric conditions of \cite{Mau} or \cite{Ch_delta} locally, in balls of radius 3.
Finally, note that basis graphs of matroids and of even $\Delta$-matroids are isometric subgraphs of Johnson graphs and of half-cubes, respectively. All this shows the importance of weakly modular and meshed
graphs in metric graph theory, geometric group theory,  CAT(0) geometry, and combinatorics and can be considered  as the ``positive'' reason for studying geodesic convexity in meshed and weakly modular graphs.

There is also a ``negative'' reason for investigating geodesic convexity in classes of graphs. First, by the results of Burris \cite{Bu} and Duchet \cite{Du_retracts}, each convexity space $(X,\C)$
with convex points can be embedded in a graph $G=(V,E)$ ($V$ is finite if $X$ is finite) such that any convex set of $\C$ is a subspace of $G$, i.e., the intersection of a geodesically convex set of $G$ with $X$.
This embedding preserve the convexity parameters (Helly, Radon, Caratheodory, 	and Tverberg numbers) but not the separation properties. Earlier and in the same vein,  de Groot  \cite{dGr} proved that for any closed set $A$ of a metric space
$(X,d)$ there exists a metric $d'$ on $X$ topologically equivalent to $d$ and which coincides with $d$ on $A$, such that $A$ is a geodesically convex set of $(X,d')$. Last but not least, by Proposition 8.45 of \cite{BrHa},
any geodesic metric space $(X,d)$ is $(3,1)$-quasi-isometric to a graph $G = (V,E)$. This graph $G$ is constructed in the following way: let $V$ be a maximal $\frac{1}{3}$-net ($V$ is a subset of $X$ such that $d(x,y)> \frac{1}{3}$ for any $x,y\in V$)  and
two points $x,y\in V$ are adjacent in $G$ if and only if $d(x,y)\le 1$. These universality results explain why graphs endowed with standard graph-metric are very general and why most computational problems for geodesic convexity in graphs are NP-complete  \cite{CoDoSa,Douetal,Douetalbis}.  On the other hand, weakly modular and meshed graphs
are still universal since they contain Helly graphs and any graph $G$ embeds in a Helly graph as a subgraph (add one or several  vertices adjacent to all vertices of $G$) or as an isometric subgraph  (via the injective hull construction).

\subsection*{Organisation} The rest of the paper is organized as follows. In Section \ref{s:preliminaries} we present classical notions that we use: convexity spaces, separation axioms, and notions from metric graph theory.
We also present some illustrative examples of $S_3$-graphs. In Section \ref{s:s3graphs} we characterize $S_3$-graphs and their semispaces in terms of convexity of shadows of maximal proximal sets. These results are specified
in Section \ref{s:s3graphstc} to the case of graphs satisfying the triangle condition (TC). Alltogether, these results of Sections \ref{s:s3graphs} and \ref{s:s3graphstc} represent the first main result of the paper.
In Section \ref{s:meshed} we present known and new properties of meshed graphs, which are used in the next section (but which also maybe useful in other contexts). In Section \ref{s:mesheds3graphs}
we characterize meshed $S_3$-graphs in terms of forbidden graphs. This is the second main result of the paper. In Section \ref{s:examples} we prove that several classes of graphs are $S_3$-graphs
or meshed $S_3$-graphs. In the final Section \ref{s:halfspacesep} we consider the halfspace separation problem. Together with the formulation of the results of the paper and their proofs,
we also formulate several open questions. The results and the questions are discussed in numerous remarks.
In some of them, we use notions which are not defined in the paper but can be easily found in the references cited in respective remarks.

\subsection*{Dedication} I would like to dedicate this paper to the memory of Andreas Dress, Maurice Pouzet, and Petru Soltan. 

\section{Preliminaries} \label{s:preliminaries} In this section, we recall the basic notions and results about convexity spaces and graphs.

\subsection{Convexity spaces and separation axioms}
In this subsection, we follow the books by Soltan \cite{So} and van de Vel \cite{VdV}.
A \emph{convexity space} (or a \emph{closure space}) is a pair $(X,\C)$ where $X$ is a set  and $\C$ is a family of subsets of $X$ such that $\varnothing,X\in \C$ and $\C$ is closed by taking intersections:
for $C_i\in \C, i\in I$,  $\bigcap_{i\in I} C_i\in \C$. The elements of $X$ are called  \emph{points} and the elements of $\C$ are called \emph{convex sets}.
Let $\conv(A)$ denotes the \emph{convex hull} of $A\subseteq X$: $\conv(A)$ is the intersection of all convex sets containing $A$.
A \emph{polytope} is the convex hull of a finite set of points and a $k$-\emph{polytope}
is the convex hull of at most $k$ points. A convexity space $(X,\C)$ is called \emph{domain-finite} (an \emph{alignement}  or \emph{algebraic}) if for any set $A$, $\conv(A)$ is the union of $\conv(A')$
such that $A'\subseteq A$ and $|A'|<\infty$. A convexity space $(X,\C)$ has \emph{arity $n$} if $A\in \C$ if and only if $\conv(A')\subset A$ for any $A'\subset A$ with $|A'|\le n$.
We will consider  only domain-finite convexity spaces.  Furthermore, we will additionally suppose that all points of $X$ are convex sets of $\C$. A subset $Y\subset X$ of a convexity space $(X,\C)$ induces
in a natural way a convexity space $(Y,\C|_Y)$ on $Y$, where $\C|_Y=\{ A\cap Y: A\in \C\}$. Then $(Y, \C|_Y)$ is called a \emph{subspace} of $(X,\C)$.  
%
%
%
%
%
For a convexity space $(X,\C)$, the \emph{join} of two sets $A$ and $B$ is the set $A\ast B=\bigcup_{a\in A,b\in B} \conv(a,b)$.

A convexity space $(X,\C)$
is called \emph{Join-Hull Commutative} (JHC for short) if for any point $x$ and any convex set $C$, $\conv(x\cup C)=x\ast C$.
Join-Hull Commutativity is  equivalent to the more general property
that $\conv(A\cup B)=A\ast B$ for any two convex sets $A,B$~\cite{KaWo}. JHC-convexity spaces have arity 2 and can be characterized in the following way via the Peano axiom:

\begin{theorem} \cite{Ca} \label{JHC}  A convexity space $(X,\C)$ satisfies JHC if and only if it satisfies the Peano axiom:  for any $u,v,w\in X$,  $x\in \conv(w,v)$,
and $y\in \conv(u,x)$ there exists $z\in \conv(u,v)$ such that $y\in \conv(w,z)$.
\end{theorem}

\begin{definition} [Halfspaces and separation axioms] A \emph{halfspace} (also called a \emph{hemispace}) is a convex set $H$ with convex complement $X\setminus H$; clearly, $X\setminus H$ is also a halfspace.
Two disjoint sets $A,B$ are \emph{separable by halfspaces} (or simply, \emph{separable}) if there exists a halfspace $H$ such that $A\subseteq H$ and $B\subseteq X\setminus H$.
In convexity theory, the following
separation axioms have been  considered \cite{Ja,VdV}:
\begin{enumerate}
\item[$S_2$]: any two distinct points $p,q$ of $X$ are separable by halfspaces;
\item[$S_3$]: any convex set $A$ and any point $p\notin A$ are separable by halfspaces;
\item[$S_4$]: any two disjoint convex sets $A,B$ are separable by halfspaces.
\end{enumerate}
\end{definition}

Note that $S_4$ implies $S_3$ and $S_3$ implies $S_2$. (In \cite[p.477]{KaWo} it is wrongly asserted that
for domain-finite convexities, $S_3$ and $S_4$ are equivalent.) Notice also that any subspace of an $S_3$-space is $S_3$ (this is not true for $S_4$), furthermore, the $S_3$-spaces define a variety sensu \cite{Ja4}.

\begin{definition} [Semispaces] For $x_0\in X$, a \emph{semispace} (a \emph{copoint} or a \emph{hypercone}) at  $x_0$  is a maximal
by inclusion convex set $S\in \C$ not containing $x_0$. Then $x_0$ is called the \emph{attaching point} of $S$.
\end{definition}

The family of all semispaces is an \emph{intersection base} in the sense that each convex set $A\in \C$ is the intersection of semispaces and no other minimal subfamily of $\C$ has this property.

A \emph{simplicial complex} on a set $X$ is a family of sets $\mathfrak X$ (called \emph{simplices}) such that  $\sigma\in {\mathfrak X}$ and  $\sigma'\subseteq \sigma$ imply
$\sigma'\in {\mathfrak X}$. Since the intersection of simplices is a simplex, the pair $(X,{\mathfrak X}\cup \{ X\})$ is a convexity space, called  \emph{simplicial convexity}. \emph{Facets} of
$\mathfrak X$ are the maximal by inclusion simplices of $\mathfrak X$. A simplex $\sigma\in {\mathfrak X}$ is called a \emph{free face} if  $\sigma\subsetneq \sigma'$
and $\sigma\varsubsetneq \sigma''$ for $\sigma',\sigma''\in {\mathfrak X}$ imply $\sigma'=\sigma''$. If $\sigma$ is a free face and $\sigma'$ is the unique simplex such that $\sigma\subsetneq \sigma'$, then
$|\sigma|=|\sigma'|-1$ and $\sigma'$ is a facet.

\begin{example} \label{exemple-simplicial-convexity}
The semispaces of the simplicial convexity $(X,{\mathfrak X}\cup \{ X\})$  are the facets and the free faces of the simplicial complex $\mathfrak X$. Indeed, given $x_0\in X$, pick any semispace $S$
with attaching point $x_0$. Then $S$ is a simplex $\sigma$ of $\mathfrak X$. By definitions of semispaces and simplicial complexes, either $\sigma$ is a maximal simplex of $\mathfrak X$ or any simplex $\sigma'$
properly including $\sigma$ has the  form $\sigma'=\sigma\cup \{ x_0\}$. In the second case, since $\sigma$ is a semispace at $x_0$,  $\sigma'$ is the unique simplex properly containing $\sigma$, thus $\sigma$ is a free face.
If $X$ is finite, say $|X|=n$ and $\mathfrak X$ contains $m$ facets, then one can easily see that $\mathfrak X$ contains at most $nm$ free faces. Thus the simplicial convexity contains at most $m(n+1)$
semispaces.
\end{example}

The following characterization of $S_3$ is well-known and easy to prove, see \cite{So,VdV}:

\begin{theorem} \label{copoints-S3} A convexity space $(X,\C)$ (with convex points) satisfies the separation axiom $S_3$ if and only if any semispace of $(X,\C)$ is a halfspace.
Consequently, the $S_3$-convexity spaces are exactly the convexity spaces in which all convex sets are intersections of halfspaces.
\end{theorem}

To characterize the  separation properties $S_3$ and $S_4$, the following notion of shadow is important:

\begin{definition} [Shadows] \cite{Ch_separa,Ch_S3}
Given two  sets $A$ and $B$ of a convexity space $(X,\C)$, the \emph{shadow of $A$ with respect to $B$}
is the set $$A/B=\{ x\in X: \conv(B\cup \{ x\})\cap A\ne \varnothing\}.$$
If $B=\{ x_0\}$, then we will write $x_0/A$ and $A/x_0$ instead of
$\{ x_0\}/A$ and $A/\{ x_0\}$. 
\end{definition}

\begin{remark}  In the paper \cite{Ch_S3}, which introduced shadows, they were called ```penumbras'' (or ``twilights'').
In \cite{VdV} the term ``extension'' was used. The term ``extension''  originates from
join geometries   \cite{PrJa}, where this term and  the notation $A/B$ designate the union $\bigcup_{a\in A, b\in B} a/b$.  The notion of shadow is more general
that the notion of extension, therefore we prefer to use the term ``shadow''.  Since join geometries are JHC convexity spaces, the two notions
coincide for convex sets $A$ and $B$ in join geometries.
Note also that the shadow $A/x_0$ can be viewed as an extension:
$$A/x_0=\bigcup_{a\in A} a/x_0=\{ x\in X: \conv(x_0,x)\cap A\ne\varnothing\}.$$
\end{remark}

Now, we recall the characterization of convexity spaces satisfying the separation axiom $S_4$. This separation axiom is often called
the \emph{Kakutani separation property}~\cite{VdV}, due to Kakutani \cite{Ka}
who first considered this property (see also the paper by Tukey \cite{Tu}). Jamison \cite{Ja} proved an equivalence
between $S_4$ and the separation of disjoint polytopes by halfspaces and van de Vel \cite{VdVtopo} characterized $S_4$ in terms of screening with half-spaces. $S_4$ can be also
characterized in terms of convexity of shadows:


Finally, in case of convexities of arity $n$, $S_4$ has been characterized in the following way by the author of this paper:

\begin{theorem} \cite{Ch_S3,Ch_thesis,Ch_separa} \label{S4arity} Let $(X,\C)$ be a convexity space.
\begin{itemize}
\item[(1)]  $(X,\C)$ satisfies the separation axiom $S_4$ if and only if the shadows $A/B$ and $B/A$ are convex for any two convex sets $A,B$.
\item[(2)]  If $(X,\C)$ has arity $n$, then $(X,\C)$  satisfies the separation axiom $S_4$ if and only if for any $n$-polytope $A$ and any $(n-1)$-polytope $B$, the
shadow $A/B$ is convex and if and only if any two disjoint $n$-polytopes $A$ and  $B$ are separable.
\item[(3)]  If $(X,\C)$ has arity 2, then $(X,\C)$ satisfies  $S_4$ if and only if it satisfies the Pasch axiom:  for any $u,v,w\in X$, $x\in \conv(w,u), y\in \conv(w,v)$, there
exists $z\in \conv(u,y)\cap \conv(v,x)$.
\end{itemize}
\end{theorem}

\begin{remark}
The last assertion of Theorem \ref{S4arity} generalizes the result of Ellis \cite{El}, which established the same equivalence under JHC. Ellis \cite{El} presented an example of a convexity space of arity 2 showing that the Peano axiom
is independent from the Pasch axiom. (See also the paper \cite{vdVS4} by van de Vel for the equivalence between $S_4$ and Pasch axiom under JHC.)  Ellis \cite{El}  proved his result for pairs of
convexity spaces on the same ground set with the motivation to generalize Stone's theorem \cite{St} for ideals and filters in distributive lattices. In \cite{Ch_S3}, we also proved Theorem \ref{S4arity} for pairs of convexity spaces; in this
form, this result was later rediscovered by Kubi\'s \cite{Ku}. Note that in several cases, for example for ordered incidence geometries \cite{BTBI} or for $\mathbb{B}$-convexity
or max-plus convexity \cite{BrHo1,BrHo2}, to establish $S_4$, the authors establish the Peano axiom and then prove the Pasch axiom. Theorem \ref{S4arity} shows that it is sufficient to prove only the Pasch axiom.
\end{remark}

Convexity spaces satisfying the separation axiom $S_3$  can be characterized in the following way:

\begin{proposition} \cite{Ch_S3,Ch_thesis} \label{S3shadows1} A convexity space $(X,\C)$ satisfies the separation axiom $S_3$ if and only if for any convex set $A$ and any point $x_0\notin A$, the shadow $x_0/A$ is convex.
\end{proposition}

A convexity space $(X,\C)$  satisfies the \emph{sandglass axiom} \cite{Ch_S3} if for any six vertices $u,u',v,v',x,y$ such that $y\in \conv(u,u')\cap \conv(v,v')$ and $x\in \conv(u,v)$ there exists
a point $x'\in \conv(u',v')$ such that $y\in \conv(x,x')$.  For JHC-convexities, the following characterization of $S_3$ holds:

\begin{proposition} \label{sandglass1}  \cite{Ch_S3,Ch_thesis} A JHC-convexity space $(X,\C)$ satisfies the separation axiom $S_3$ if and only if $\C$ satisfies the sandglass property.
\end{proposition}


\begin{remark} \label{complexity}
From the computational point of view,  the characterizations of JHC-convexity spaces via Peano axiom (Theorems~\ref{JHC}),  of $S_4$-convexity spaces of arity $n$ (Theorem~\ref{S4arity}), and of JHC-convexity spaces satisfying $S_3$ (Proposition~\ref{sandglass1})
can be considered as efficient, since they can be tested in time polynomial in the size of $X$ (for more details, see the paper  \cite{SeHoWr} and the thesis \cite{Se}).
Indeed, Peano, Pasch, and sandglass axioms are 6- or 7-point conditions and the convexity of shadows $A/B$ for $n$-polytopes $A$ and $(n-1)$-polytopes $B$ (characterizing the $S_4$ axiom for convexity spaces of arity $n$)
can be viewed as a condition on $3n$ points.
On the other hand, Proposition \ref{S3shadows1} does not provide any insight on the complexity of testing the property $S_3$.

The importance of characterizing separation properties by convexity of shadows stems from the fact that in convexity spaces where convex hulls can be computed efficiently, the construction of shadows  and testing their
convexity can be done efficiently.
\end{remark}

The following questions are to our knowledge open (already for geodesic convexity in  graphs): 

\begin{question} \label{S3problem1} Is it possible to characterize convexity spaces of arity $n$ (arity $2$ or geodesic convexity in graphs) satisfying the separation axiom $S_3$ via a condition
(a) on specific subsets or (b) on subsets with a fixed number of points?
\end{question}

\begin{question} \label{S3problem-decision} What is the complexity of deciding if a finite convexity space (a finite convexity space of arity 2 or geodesic convexity in graphs) is $S_3$?
\end{question}

\begin{question} \label{S3problem2} Characterize the semispaces in convexity spaces of arity $n$ (arity $2$ or geodesic convexity in graphs).
\end{question}

\begin{remark}
We believe that the answer to Question \ref{S3problem1} is negative even for geodesic convexity in graphs.  In \cite{Ch_S3} we asserted that the answer
to Question \ref{S3problem1} is negative in case of arity $n=2$ but the provided example is not correct. Furthermore, we believe that deciding
if a finite convexity space is $S_3$ is NP-complete (and again, already for geodesic convexity in graphs). Proposition \ref{S3shadows1} shows that this decision problem is in NP.
A positive answer to any version of Question \ref{S3problem1}(b) would imply a positive answer to Question \ref{S3problem-decision}. On the other hand, under the assumption that P$\ne$NP, the NP-completness
of deciding $S_3$ would imply that Question \ref{S3problem1}(b) has a negative answer.

Most of the remaining open questions are specifications of Questions \ref{S3problem1}, \ref{S3problem-decision}, and \ref{S3problem2}. The results of this paper can be viewed as contributions to these three fundamental
questions.
\end{remark}


\subsection{Interval spaces}
Let $V$ be any set, whose elements are called \emph{points}. For each pair $u,v\in V$, let $u\circ v$ be a subset of $V$, called the \emph{interval} between $u$ and $v$.
Then $(V,\circ)$ is called an \emph{interval space}~\cite{VdV} if $u\in u\circ v$ and $u\circ v=v\circ u$.
The interval space $(V,\circ)$ is  \emph{geometric}
if $u\circ u=\{ u\}$, $w\in u\circ v$ implies $u\circ w\subseteq u\circ v$, and $v,w\in u\circ x$ and $v\in u\circ w$ implies $w\in v\circ x$ for all $u,v,w,x\in V$~\cite{BaVdVVe,Ve}.
A particular instance of geometric interval space is any metric space $(V,d)$ where the intervals are the metric intervals $[u,v]=\{ x\in V: d(u,x)+d(x,v)=d(u,v)\}$.
Notice that  the convexity in any interval space has arity 2.
With any interval space on $V$ one can associate a graph $G=(V,E)$ such that $u\sim v$ if and only if $u\circ v=\{ u,v\}$. An interval space $(V,\circ)$ is \emph{graphic}
if $u\circ v=[u,v]$ for any pair $u,v\in V$, where $[u,v]$ is the interval between $u$ and $v$ in $G$. Clearly, the convex sets of a graphic interval space coincide with the
geodesically convex sets of its graph.  A simple sufficient condition for a discrete interval space to be graphic was given in~\cite{BaCh_helly}.

\begin{definition} [Triangle Condition] An interval space satisfies the \emph{triangle condition} if
\begin{enumerate}
\item[(TC)] for any  $u,v,w\in V$ such that $u\circ v\cap u\circ w=\{ u\}, u\circ v\cap v\circ w=\{ v\},$ and $u\circ w\cap v\circ w=\{ w\}$, the intervals $u\circ v,u\circ w,v\circ w$ are edges whenever at least one of them is an edge.
\end{enumerate}
For graphs, the triangle condition can be reformulated as follows:
\begin{enumerate}
\item[(TC)] for any  $u,v,w\in V$ with $1=d(v,w)<d(u,v)=d(u,w)$ there exists a common neighbor
$x$ of $v$ and $w$ such that $d(u,x)=d(u,v)-1.$
\end{enumerate}
\end{definition}


\begin{theorem}\cite{BaCh_helly}\label{graphic-triangle-condition}  A discrete geometric interval space satisfying the triangle condition is graphic.
\end{theorem}

Burris \cite{Bu} proved that any convexity
space $(X,\C')$ is a subspace of some interval space $(V,\circ)$  ($(V,\circ)$ is finite if $(X,\C')$ is finite). Based on this result,
Duchet \cite{Du_retracts} proved that any convexity space $(X,\C)$ with convex points is a subspace  of the geodesic convexity of some graph $G=(V,E)$ (if  $(X,\C)$  is finite, then $V$ is also finite).
Note that however the graph $G$ is not necessarily $S_3$ if $(X,\C)$ is. Thus one can ask if \emph{any $S_3$-space is a subspace of an $S_3$-graph.}

In case of geodesic convexity in graphs, the separation axiom $S_4$ implies that the intervals are convex sets (i.e., $\conv(u,v)=[u,v]$), thus the Pasch axiom from Theorem \ref{S4arity}
can be rewritten in the following form:   for any $u,v,w\in V$, for any $x\in [w,u], y\in [w,v]$, there
exists $z\in [u,y]\cap [v,x]$.  In the case of graphs (and of interval spaces) one can also consider a
stronger version of the join operation: for the rest of the paper, the \emph{join} of two sets $A$ and $B$ is the set $A\ast B=\bigcup_{a\in A, b\in B} [a,b]$ and we call  the convexity
\emph{join-hull commutative} if for any convex set $A$ and any vertex $x$, $\conv(A\cup\{x\})=x\ast A$. Then JHC is equivalent to the fact that the join $A\ast B$ of two convex sets $A$ and $B$ is convex and is equivalent
to the following version of the Peano axiom: for any $u,v,w\in X$, for any $x\in [w,v]$, and for any $y\in [u,x]$ there exists $z\in [u,v]$ such that $y\in [w,z]$.

\begin{definition} [$S_3$,  Pasch, and Peano graphs]  We call a graph $G$ an \emph{$S_3$-graph}  if its geodesic convexity satisfies the $S_3$-property.
We call a graph $G$ a \emph{Pasch graph} (or an \emph{$S_4$-graph}) if its geodesic convexity satisfies the separation
property  $S_4$. Finally,  a graph $G$ is a \emph{Peano graph} if the geodesic convexity satisfies  the JHC property and a \emph{Pasch-Peano graph}  if $G$ is a Pasch and Peano graph.
\end{definition}

For a study of Pasch-Peano graphs, see \cite{BaChvdV} (its results have been presented in the book \cite{VdV}).  Peano and Pasch axioms can be efficiently verified,  
however,  $S_3$ is much more elusive. 
In case of bipartite graphs, $S_3$ can be characterized in a pretty way: 

\begin{proposition} \cite{AlKn,Ba_intrinsic,Ch_thesis,Ch_bipartite,Dj}\label{copoints-bipartite} For the geodesic convexity of a bipartite graph $G=(V,E)$, the following conditions are equivalent:
\begin{enumerate}
\item[(i)] the geodesic convexity of $G$ satisfies the separation axiom $S_2$;
\item[(ii)] the geodesic convexity of $G$ satisfies the separation axiom $S_3$;
\item[(iii)] $G$ is a partial cube, i.e., $G$ has an isometric embedding in a hypercube;
\item[(iv)] for any edge $uv$ of $G$, the sets $W(u,v)=\{ x\in V: d(u,x)<d(v,x)\}$ and $W(v,x)=\{ x\in V: d(v,x)<d(v,y)\}$ are complementary halfspaces of $G$.
\end{enumerate}
\end{proposition}
The equivalence (iii)$\Longleftrightarrow$(iv) is the content of Djokovi\'{c}'s theorem \cite{Dj}. The equivalence between the conditions (i),(ii),(iii), and (iv) has been independently established in \cite{AlKn,Ba_intrinsic,Ch_thesis,Ch_bipartite}.
$S_4$-bipartite graphs have been characterized in  \cite{Ch_bipartite,Ch_separa}  as partial cubes with  forbidden pc-minors.

\subsection{Graphs}
All graphs $G=(V,E)$ in this paper are
undirected, connected, and contain no multiple edges, neither loops, and are not necessarily finite or locally finite.
We write $u \sim v$ if $u, v \in V$ are adjacent and $u\nsim v$ if $u,v$ are not adjacent. Furthermore, we say that a set $A\subset V$ and a vertex $x_0\notin A$ are
\emph{adjacent} (notation $x_0\sim A$) if $x_0$ is adjacent to a vertex of $A$. For a subset
$A\subseteq V,$ the subgraph of $G=(V,E)$  {\it induced by} $A$
is the graph $G(A)=(A,E')$ such that $uv\in E'$ if and only if $uv\in E$. A set $P$ is an \emph{induced path} if $G(P)$ is a path of $G$.
The \emph{length} of a $(u,v)$-path $P$ is the number  of
edges in $P$. A \emph{shortest $(u,v)$-path} (or a
$(u,v)$-\emph{geodesic})
is a $(u,v)$-path with a minimum number of edges.  The \emph{distance} $d_G(u,v)$ between two vertices $u$ and $v$ of  $G$ is the length of a $(u,v)$-geodesic.
If there is no ambiguity, we will write $d(u,v)=d_G(u,v)$.
As above, the \emph{interval} $[u,v]$ between $u$ and $v$ is the set of all vertices on $(u,v)$-geodesics, i.e.  $[u,v]=\{w\in V: d(u,w)+d(w,v)=d(u,v)\}$. 
If $d(u,v)=2$, then $[u,v]$ is called a \emph{2-interval}. 
For a vertex $v$ of $G$ and an integer $r\ge 1$, we denote  by $B_r(v)$
the \emph{ball} in $G$
(and the subgraph induced by this ball)  of radius $r$ centered at  $v$, i.e.,
$B_r(v)=\{ x\in V: d(v,x)\le r\}.$ More generally, the $r$--{\it ball  around a set} $A\subseteq V$
is the set (or the subgraph induced by) $B_r(A)=\{ v\in V: d(v,A)\le r\},$ where $d(v,A)=\mbox{min} \{ d(v,x): x\in A\}$.
As usual, $N(v)=B_1(v)\setminus\{ v\}$ denotes the set of neighbors of a vertex $v$ in $G$.
%
A graph $G=(V,E)$
is {\it isometrically embeddable} into a graph $G'=(W,F)$
if there exists a mapping $\varphi : V\rightarrow W$ such that $d_{G'}(\varphi (u),\varphi(v))=d_G(u,v)$ for $u,v\in V$.
We call a subgraph $H$  of $G$ an \emph{isometric subgraph}  if $d_H(u,v)=d_G(u,v)$
for each pair of vertices $u,v$ of $H$. The induced subgraph $H$ of $G$
(or the vertex set of $H$) is called {\it  (geodesically) convex}
if it includes the interval of $G$ between any pair of its
vertices. The smallest convex subgraph containing a given set $A$
is called the {\it convex hull} of $A$ and is denoted by $\conv(A)$.
An induced subgraph $H$ (or the corresponding vertex set of $H$)
of a graph $G$
is {\it gated} \cite{DrSch} if for every vertex $x$ outside $H$ there
exists a vertex $x'$ in $H$ (the {\it  gate} of $x$)
such that  $x'\in [x,y]$ for any $y$ of $H$. Since the intersection
of gated sets is gated, the gated sets of a graph
define a convexity space. A graph $G$ is a {\it gated amalgamation}
of two
graphs $G_1$ and $G_2$ if $G_1$ and $G_2$ are (isomorphic to) two intersecting
gated subgraphs of $G$ whose union is all of $G.$ Finally, a set $S$
is \emph{monophonically convex} (or \emph{induced path convex}) if for
any $u,v\in S$, $S$ contains all vertices on induced $(u,v)$-paths of $G$.
Gated sets and monophonically convex sets are geodesically convex.

As usually, $C_n$ is the  \emph{cycle} on $n$ vertices, $K_n$ is the
complete graph on $n$ vertices, and
$K_{n,m}$ is the complete bipartite graph with $n$ and $m$ vertices on each
side. $C_4$ is called a \emph{square}  	and $C_3=K_3$ is called a \emph{triangle}. 
By $W_n$ we denote
the \emph{wheel} consisting of $C_n$ and a central vertex adjacent to all vertices of
$C_n$. By $W^-_n$ we mean the wheel $W_n$ minus an edge connecting the central vertex to a vertex of $C_n$.
Analogously $K^-_4$ and $K^-_{3,3}$ are the graphs obtained from $K_4$ and $K_{3,3}$ by removing one edge. 
An $n$--{\it octahedron} $O_n$ (or, a {\it hyperoctahedron}, for short) is the complete graph $K_{2n}$
on $2n$ vertices minus a perfect matching. The unique vertex of $O_n$ not adjacent to a vertex $x$ is denoted by $x^*$ (consequently, $x^**=x$). Let $G_{i}$, $i \in \Lambda$ be an arbitrary family of graphs. The
\emph{Cartesian product} $\prod_{i \in \Lambda} G_{i}$ is a graph
whose vertices are all functions $x: i \mapsto x_{i}$, $x_{i} \in
V(G_{i})$.  Two vertices $x,y$ are adjacent if there exists an index
$j \in \Lambda$ such that $x_{j} y_{j} \in E(G_{j})$ and $x_{i} =
y_{i}$ for all $i \neq j$. Note that a Cartesian product of infinitely
many nontrivial graphs is disconnected. Therefore, in this case the
connected components of the Cartesian product are called {\it weak Cartesian products}.
%
A {\it hypercube} $H(X)$ is a graph having the finite subsets of $X$ as vertices
and two  such sets $A,B$ are adjacent in $H(X)$
iff $|A\triangle B|=1$ (where the {\it symmetric difference} of $A$ and $B$ is written and defined by $A\triangle B=(A\setminus
B)\cup (B\setminus A)$). Hypercubes are (weak) Cartesian products of edges (of $K_2$). More generally, \emph{Hamming graphs} are (weak) Cartesian products of cliques.
A {\it half-cube} $\frac{1}{2}H(X)$ has the vertices
of a hypercube $H(X)$ corresponding to finite subsets of $X$ of even cardinality
as vertices and two  such vertices are adjacent in $\frac{1}{2}H(X)$  iff  their  distance
in $H(X)$ is 2 (analogously one can define a half-cube on finite subsets of odd cardinality).  For a positive integer $k$, the {\it Johnson graph}  $J(X,k)$ has the subsets of $X$ of size $k$ as vertices and
two such vertices are adjacent in $J(X,k)$ iff their  distance in $H(X)$ is $2$.
All Johnson graphs $J(X,k)$ with even $k$ are isometric subgraphs of the half-cube $\frac{1}{2}H(X)$.  If $X$ is finite and $|X|=n$,
then the hypercube, the half-cube, and the Johnson graphs are usually denoted by $H_n, \frac{1}{2}H_n,$ and $J(n,k),$
respectively. 

\subsection{Examples of $S_3$-graphs} In order to give some intuition about $S_3$-graphs and their semispaces, we conclude the preliminary section with several
simple examples of such graphs. Some of these examples are
mentioned in other places of the paper.

\begin{example} (Trees) Each vertex $x_0$ of a tree $T$ has $\deg(x_0)$ semispaces attached at $x_0$: each of them is obtained by removing an edge incident to $x_0$ (but not removing its ends)
and considering the resulting subtree of $T$ not containing the vertex $x_0$. Its complement if the subtree containing $x_0$ and is also convex, thus semispaces of $T$ are halfspaces.
\end{example}

\begin{example} (Hypercubes and partial cubes) The $d$-dimensional hypercube $H_d$ has $2d$ semispaces corresponding to the $2d$ facets of $H_d$.
For each vertex $x_0$ of $H_d$ there are $d$ semispaces having $x_0$ as the attaching vertex.  Each such semispace is defined by an edge $x_0v$ incident to $x_0$: when removing
all edges parallel to the edge $x_0v$ the hypercube $H_d$ is partitioned in two subhypercubes of dimension $d-1$ and the subhypercube not containing $x_0$ is the respective semispace.
Again, semispaces of $H_d$ are halfspaces and, vice-versa, halfspaces of $H_d$ are semispaces. Since $H_d$ has $2^d$ vertices and $2d$ semispaces, $H_d$ is an example of a graph having
a logarithmic number of semispaces (Hamming graphs are other such examples).

Partial cubes and median graphs have the same structure of semispaces as hypercubes and trees. Namely, by Proposition \ref{copoints-bipartite},
for any vertex $x_0$ and any neighbor $v$ of $x_0$, the sets $W(x_0,v)=x_0/v$ and $W(v,x_0)=v/x_0$ are complementary halfspaces (they correspond to the two connected components of $G$ obtained by removing all edges parallel to $x_0v$
and also to the two sets of vertices having the same value 0 or 1 in a fixed coordinate of the isometric embedding of $G$ into a hypercube). Then $W(v,x_0)$ is a semispace at $x_0$: indeed,
any $z$ in $W(x_0,v)$ is closer to $x_0$ than to $v$, thus $x_0\in [z,v]$ and $z$ cannot belong to a semispace containing $W(v,x_0)$ and avoiding $x_0$. Since each halfspace of $G$ has the form $W(x_0,v)$ or $W(v,x_0)$
for some edge $x_0v$ of $G$, halfspaces and semispaces of partial cubes are the same. The square grid in any dimension is a median graph, thus an $S_3$-graph.
\end{example}

\begin{example} (Hyperoctahedra and complete graphs) \label{ex-hyper} The $d$-dimensional hyperocahedron $O_d$ is the dual of the hypercube $H_d$: $O_d$ has $2d$ vertices and $2^d$ maximal cliques (corresponding to the vertices of $H_d$).
$O_d$ has diameter 2. Each vertex $x$ and its unique non-neighbor $x^*$ correspond to complementary halfspaces of $H_d$.  The semispaces attached to a given vertex $x_0$ of $O_d$ are
precisely the $2^{d-1}$ maximal cliques of $Q_d$ containing the unique vertex $x^*_0$ of $O_d$ not adjacent to $x_0$.
Each such semispace $S=\{u_1,\ldots,u_d\}$ is a halfspace: its complement is the maximal clique $S^*=\{u^*_1,\ldots,u^*_d\}$.
The hyperoctahedron $O_d$ is an example of an $S_3$-graph with an exponential number of semispaces.

In a complete graph $K_d$, the semispaces are the sets of the form $V\setminus \{ x_0\}$ and each vertex $x_0$ is the attaching vertex of a unique semispace $V\setminus \{ x_0\}$, which is also a halfspace.
On the other hand, $K_d$ contains halfspaces that are not semispaces: any pair of the form $A,V\setminus A$ with $A\ne\varnothing, V$ are complementary halfspaces, however $A$ is a semispace if and only if $|A|=d-1$.
\end{example}

\begin{example} \label{petersen} (Petersen graph and Platonic solids) The semispaces of the Petersen graph $P_{10}$ are the 5-cycles. Since their complements are also 5-cycles and the 5-cycles of $P_{10}$ are convex, the semispaces of $P_{10}$ are halfspaces.
Each vertex $x_0$ of $P_{10}$ is the attaching vertex of 3 adjacent semispaces.

The semispaces of the graph of the icosahedron are the 5-wheels  (which are convex subgraphs). Since, the complements of 5-wheels are also 5-wheels, the icosahedron is an $S_3$-graph. For each vertex $x_0$ there are 5 semispaces
adjacent to $x_0$ and one semispace non adjacent to $x_0$, which is the 5-wheel centered at the unique vertex at distance 3 from $x_0$. Analogously, one can analyse the semispaces of the dodecahedron and show that it is an $S_3$-graph.
Since the cliques, hyperoctahedra, and hypercubes of all dimensions are $S_3$-graphs, we deduce that 1-skeleta of all Platonic solids in ${\mathbb R}^3$ are $S_3$-graphs. In fact, in \cite{BaChvdV} it is shown that they are
Pasch-Peano graphs (see also \cite[Subsection 4.24.3]{VdV}).
\end{example}

\begin{example} (Triangular and king grids) Let $T_6$ denote the plane graph defined by the tiling of the plane into equilateral triangles with side 1. Each vertex $x_0$ has 6 semispaces attached and adjacent to $x_0$.
To define them, we pick each of the 6 triangles $x_0uv$ incident to $x_0$ and consider the union $W(u,x_0)\cup W(v,x_0)$, which consists of vertices of $T_6$ which are closer to the pair $u,v$ than to $x_0$. This set coincides with the
shadow $\{ u,v\}/x_0$.  Geometrically one can view the shadow $\{ u,v\}/x_0$ as follows: removing from the grid $T_6$ the edges (but leaving their ends) from the zipped zone defined by the edges $x_0u$ and $x_0v$,  $T_6$ will be
partitioned into two connected subgraphs. One can easily show (and this follows from more general results) that both these subgraphs
define complementary halfspaces of $T_6$.  Then $\{ u,v\}/x_0=W(u,x_0)\cup W(v,x_0)$ is those of the two subgraphs which does not contain the vertex $x_0$. Thus each semispace of $G$ is a halfspace. The triangulation $T_6$
can be viewed as the discrete analog of the Euclidean plane and the semispaces of $T_6$ can be viewed as analogs of hypercones  in Euclidean spaces sensu \cite{Ko}.

A similar construction holds for all \emph{bridged  triangulations}, i.e., 2-connected plane graphs in which all inner faces are triangles and all inner vertices have degree $\ge 6$, see \cite[Proof of Theorem 2]{BaCh_wmg}.
Analogs of the triangular grid $T_6$ are the (bridged) triangulations $T_k$ of the hyperbolic plane in which all inner vertices have degree $k$ for some $k\ge 7$.

Now, consider the king grid ${\mathbb Z}^2_{\infty}$, i.e., ${\mathbb Z}^2$ endowed with the $\ell_{\infty}$-metric. Then each vertex $x_0$ belongs to four maximal cliques, which are all $K_4$. The four semispaces at $x_0$ and adjacent to $x_0$ are defined
by these four cliques. Suppose without loss of generality that $x_0$ is the vertex $(0,0)$ and that the maximal clique has the form $K\cup \{ x_0\}$, where $K=\{ (1,0),(0,1),(1,1)\}$. Then the diagonal line defined by the points $(\frac{1}{2},0)$
and $(0,\frac{1}{2})$ partition the grid into two convex subgrids. Then one can easily see that the subgrid not containing $x_0$ has the form $K/x_0$ and is a semispace at $x_0$. Furthermore, this subgrid is a halfspace of ${\mathbb Z}^2_{\infty}$.
\end{example}

It is no coincidence that in several examples the semispaces have the form $K/x_0$ for a maximal clique of the form $K\cup \{ x_0\}$: we will show that all semispaces in $S_3$-graphs satisfying the triangle condition (TC) have this form.

\section{$S_3$-Graphs} \label{s:s3graphs}
In this section, first we prove the characterization of  $S_3$ separation property in general convexity spaces and establish some properties of $S_3$-graphs (with respect to the geodesic convexity).
Then we present a characterization of  semispaces of $S_3$-graphs as shadows of vertices on special sets. Finally,
we present a characterization of $S_3$-graphs, which is a refinement of the characterization of $S_3$-convexity spaces and that uses the structure of semispaces.

\subsection{$S_3$-convexity spaces}

We characterize $S_3$ in terms of convexity of shadows. Proposition \ref{S3shadows} 
is from  the paper \cite{Ch_S3}, which was published in the local University press. 
For completeness, we provide its proof. 
We start with a simple observation:

\begin{lemma} \label{shadow-separation} Let $(X,\C)$ be a convexity space and let $H',H''$ be two complementary halfspaces separating two disjoint convex sets $A$ and  $B$,  say $A\subseteq H'$ and $B\subseteq H''$.
Then the shadow $A/B$ belongs to $H'$ and the shadow $B/A$ belongs to $H''$. Furthermore, $\conv(A/B)\subseteq H'$ and $\conv(B/A)\subseteq H''$.
\end{lemma}

\begin{proof} To prove that $B/A\subseteq H''$,  pick any $y\in B/A$. Then there exists  $z\in \conv(A\cup \{ y\})\cap B$. Necessarily $y$ must belong to $H''$, otherwise, if $y\in H'$, then  $z\in H'\cap B\subseteq H'\cap H''$, which is impossible.
Consequently, $B/A\subseteq H''$. Since $H''$ is convex, $\conv(B/A)\subseteq H''$.
\end{proof}

%

\begin{proposition} \label{S3shadows} For a convexity space $(X,\C)$ the following conditions are equivalent:
\begin{enumerate}
\item[(i)] $\C$ satisfies the separation property $S_3$;
\item[(ii)]  any polytope $P$ and any point $x_0\notin P$ are separable; 
\item[(iii)] for any polytope $P$ and any point $x_0\notin P$, the shadow $x_0/P$ is convex;
\item[(iv)] for any convex set $A$ and any point $x_0\notin A$, the shadow $x_0/A$ is convex.
\end{enumerate}
\end{proposition}

\begin{proof} Clearly, (i)$\Rightarrow$(ii). Now we prove that (ii)$\Rightarrow$(iii). 
Suppose by way of contradiction that for a polytope $P$ and $x_0\notin P$, the shadow  $x_0/P$ is not convex. Then there exists a finite set $Y\subset x_0/P$ and $y\in \conv(Y)\setminus x_0/P$. Let $P'=\conv(P\cup \{ y\})$. By the choice of $y$,
we have $x_0\notin P'$.  Since $\C$ is $S_3$, there exist complementary halfspaces $H',H''$ such that
$x_0\in H'$ and $P'\subseteq H''$. Since $P\subseteq P'$, $H'$ and $H''$ separate $x_0$ and $P$. By Lemma \ref{shadow-separation}, $Y\subseteq x_0/P\subseteq H'$. Since $y\in \conv(Y)\subseteq H'$ and $y\in P'\subseteq H''$, we
obtain a contradiction. Thus $x_0/P$ is convex. By Lemma \ref{shadow-separation}, $\conv(P/x_0)\subseteq H''$ and $x_0\in H'$, thus  $x_0\notin \conv(P/x_0)$. This establishes (i)$\Rightarrow$(ii).

Now we prove (iii)$\Rightarrow$(iv). Suppose $x_0/A$ is not convex for a convex set $A$. Then $x_0\notin A$, otherwise $x_0/A=X$.  Since $\C$ is domain-finite,
there exists a finite set $Z\subseteq x_0/A$ and a point
$z\in \conv(Z)\setminus (x_0/A)$. Let $Z=\{ z_1,\ldots,z_k\}$. Since $x_0\notin A$ and  $x_0\in \conv(A\cup \{ z_i\})$ for any $z_i\in S$,  by domain-finiteness of $\C$ there exists a finite subset $A_i\subseteq A$ such that
$x_0\in \conv(A_i\cup \{ z_i\})$, $i=1,\ldots,k$. Let $P=\conv(\bigcup_{i=1}^k A_i)$. Then $P$ is a polytope and $P\subseteq A$. Since $z_1,\ldots, z_k\in x_0/P$ and $z\in \conv(z_1,\ldots,z)\setminus (x_0/A)$, we get
$z\notin x_0/P$, contrary to assumption (iii).

To prove (iv)$\Rightarrow$(i), suppose that $(X,\C)$ satisfies the property (iii) and let $S$ be a semispace of $x_0$.
By Theorem \ref{copoints-S3} we have to prove that $X\setminus S$ is convex.
By maximality of $S$, we have $x_0\in \conv(S\cup \{ y\})$ for any point $y\in X\setminus S$. Consequently, $X\setminus S\subseteq x_0/S$. Since $S$ is convex and $x_0\notin S$, the sets $S$ and $x_0/S$ are disjoint, thus $x_0/S\subseteq X\setminus S$,
yielding $X\setminus S=x_0/S$.  Since $x_0/S$ is convex (because the shadows , the sets $S$ and $X\setminus S$ are complementary halfspaces.
\end{proof}

The following question can be viewed as a specification of Question \ref{S3problem1}:

\begin{question} \label{question1} Similarly to Theorem \ref{S4arity} and Proposition \ref{sandglass}, does there exist $k\in {\mathbb N}$ depending of the arity $n$ of a convexity space $(X,\C)$ such that  in Proposition \ref{S3shadows}(iii)
the convexity of the shadows $x_0/P$ for arbitrary polytopes $P$ can be replaced by the convexity of the shadows $x_0/P$ for $k$-polytopes $P$?
\end{question}

%
%
%
%
%
%
%

\subsection{Properties of $S_3$-graphs}
We present two useful properties of $S_3$-graphs and semispaces.  

\begin{lemma} \label{S3=>convexintervals} If $G$ is an $S_3$-graph, then the intervals of $G$ are convex.
\end{lemma}

\begin{proof} Suppose by way of contradiction that not all intervals of $G$ are convex and let $u,v$ be a closest pair of vertices such that $[u,v]$ is not convex. Let $x,y$ be a closest pair of vertices
of $[u,v]$ such that there exists $z\in [x,y]\setminus [u,v]$. Then we assert that $[z,y]\cap [u,v]=\{ y\}$. Indeed, if $y'\in [z,y]\cap [u,v]$ and $y'\ne y$, since $y'\in [z,y]$ and $z\in [x,y]$, we get that
$z\in [x,y']$ and $y'\in [u,v]$, contrary to the minimality choice of the pair $x,y$. Thus $[z,y]\cap [u,v]=\{ y\}$ and therefore we can suppose without loss of generality that $z$ is adjacent to $y$.
From the minimality choice of the pair $x,y$ we conclude that each of the vertices $x,y$ is different from each of the vertices $u,v$.
Since $z\sim y$ and $z\notin [u,v]$, at least one of the inclusions $y\in [z,u]$ or $y\in [z,v]$ hold, say $y\in [z,v]$.
Since $y\ne v$, by the minimality choice of $u,v$, the interval $[u,y]$ is convex. Since $G$ is an $S_3$-graph, there exist complementary halfspaces $H',H''$ such that $z\in H'$ and $[u,y]\subseteq H''$.
Since $y\in [z,v]$, $v$ cannot belong to $H'$, thus $v\in H''$. Consequently, $z\in \conv(u,v)\subseteq H''$. Since $z\in H'$, we obtain a contradiction.
\end{proof}


\begin{lemma} \label{copoint} If $S$ is a semispace of $G$, then $S$ is a semispace attached to a vertex $x_0$ adjacent to $S$. Consequently, any convex set $A$ is the intersection of semispaces
at vertices $x_0$ adjacent to $A$.
\end{lemma}

\begin{proof} Let $S$ be a semispace at  $x$. Let $y$ be a closest to $x$ vertex of $S$ and  $x_0$ be a neighbor of $y$ in $[x,y]$.
We assert that $S$ is a semispace at $x_0$.
Otherwise,  there exists a convex set $S'$ properly containing $S$ and avoiding $x_0$.
Since $S$ is a semispace at $x$, necessarily $S'$ contains $x$. Since $x,y\in S'$,  this is impossible because
$x_0\in [x,y]\subset S'$. Thus $S$ is a semispace at $x_0$. Analogously, given a convex subgraph $A$ and a vertex $x\notin A$,
let $y$ be a closest to $x$ of $A$ and $x_0$ be a neighbor of $y$ in $[x,y]$. Then the semispace $S$ at $x_0$ containing $A$
does not contain $x$, establishing the second assertion.
\end{proof}

\subsection{Semispaces in $S_3$-graphs}
For a graph $G=(V,E)$, denote by $\cS:=\cS(G)$ the set of all semispaces of  $G$ and by $\cS_{x_0}$ the set of all
semispaces  having $x_0$ as an attaching vertex and adjacent to $x_0$. Lemma \ref{copoint} shows that $\cS$ is a union
$\bigcup_{x_0\in V }\cS_{x_0}$ (but not necessarily a disjoint union).

In this subsection we characterize the semispaces of $S_3$-graphs in terms of shadows. Before presenting this characterization, we recall   the structure of semispaces in linear spaces, established by
Hammer \cite{Ha}, Klee \cite{Klee}, and K\"othe \cite{Ko}, and which we present based on the book \cite{Ko}. A \emph{cone} with vertex $x_0$ in a linear space $L$
is a subset $K(x_0)$ of $L$ that contains every point $x_0+\rho(x-x_0), \rho>0,$  whenever it contains $x$. A  \emph{convex cone} is a cone which is also a convex set. If $K(\circ)$ is a cone
with vertex at the origin of coordinates $\circ$,  then $-K(\circ)$ is also a cone.
The cone $K^*(x_0):=x_0-K(\circ)$ is called the cone  \emph{diametrically opposite}  to the cone $K(x_0)=x_0+K(\circ)$.  A cone is  \emph{truncated} if it does not contain its vertex.
The cone  \emph{generated} by a set $M$ is the smallest cone with vertex $x_0$ which contains all points of $M$ (any cone can be generated in this way). Finally, a  \emph{hypercone} at $x_0$
is a maximal convex truncated cone with vertex $x_0$. By \cite[pp.188-189]{Ko}, the hypercones at $x_0$ are exactly the semispaces at $x_0$ and for each such hypercone,
the space  $L$ is the disjoint union of $\{ x_0\}$, the hypercone $K(x_0)$ and the diametrically opposite hypercone $K^*(x_0)$.
A  cone $K(x_0)$ generated by a  set $M$ can be viewed as the union of the join $x_0*M$ and of the shadow $M/x_0$.
Then  $K(x_0)$ can be viewed as  directed union of shadows $M_i/x_0$ for a sequence $M_i$ of sections of $K(x_0)$
with parallel hyperplanes converging to a support hyperplane of $K(x_0)\cup \{ x_0\}$ passing via $x_0$.

This geometric intuition is behind the following definitions for graphs.  Roughly speaking, in case of graphs $G=(V,E)$, instead of viewing the semispaces
at $x_0$ as directed unions of shadows, we prove that the semispaces at $x_0$ and adjacent to $x_0$ are the shadows $K/x_0$ of $x_0$ on certain
sets $K$, which we call maximal $x_0$-proximal sets. While in $S_3$-graphs the shadows $A/x_0$ of convex sets $A$ are not
necessarily convex, this is the case for shadows $K/x_0$. Furthermore, analogously to linear spaces,  the complements of such shadows will be the
shadows $x_0/K$ (which will be also convex).  
While in $S_3$-graphs, the structure of  maximal proximal sets can be quite general, in the next section we will show that
in case of $S_3$-graphs satisfying the triangle condition (TC), the  maximal $x_0$-proximal sets together with $x_0$ are precisely the maximal cliques of $G$.

%
%
%
%
%

\begin{definition} [Imprint, $\Imp$-convex set]  The \emph{imprint} of a vertex $x_0$ on a set $A\subseteq V$ is the set $\Imp_{x_0}(A)=\{ z\in A: [x,z]\cap A=\{ z\}\}$. A convex set $A$ adjacent to $x_0$ is called
 \emph{$\Imp$-convex} if $A=\conv(\Imp(A))$.
\end{definition}

Note that for any vertex $y\in A$ there
exists a vertex $y'\in \Imp_{x_0}(A)\cap [x_0,y]$: as $y'$ one can take a closest to $x_0$ vertex of $A\cap [x_0,y]$.
Note also that in case of convex polyhedra and convex polyhedral cones $P$ in linear spaces, $\Imp_{x_0}(P)$ is the set of all facets of $P$ that are visible from $x_0$.

\begin{definition} [$\Upsilon_{x_0}, \Upsilon^*_{x_0},$ and $\preceq_{x_0}$]
 A set $K\subseteq V$ of $G=(V,E)$ is called  \emph{$x_0$-proximal}  if
\begin{enumerate}
\item[(P1)]  $\Imp_{x_0}(K)=K$,
\item[(P2)]  the convex hull $\conv(K)$ of $K$ does not contain the vertex $x_0$.
\end{enumerate}
Denote by $\Upsilon_{x_0}$ the  family of all $x_0$-proximal sets and let $\Upsilon^*_{x_0}=\{ K\in \Upsilon_{x_0}: x_0\sim K\}$.
For two subsets of vertices $K,K'$ of $G$, define $K\preceq_{x_0} K'$ if and only if $K\subseteq K'/x_0$.
\end{definition}

Imprints, shadows, and $x_0$-proximal sets satisfy the following simple properties: 
\begin{enumerate}
\item[(1)] $A\subseteq \Imp_{x_0}(A)/x_0$.
\item[(2)] $\Imp_{x_0}(A)=\Imp_{x_0}(A/x_0)$.
\item[(3)] $B\subseteq A/x_0$ implies that $B/x_0\subseteq A/x_0$.
\item[(4)] $\Imp_{x_0}(\Imp_{x_0}(A))=\Imp_{x_0}(A)$.
\item[(5)] if $K\subset K'$, then $K\preceq_{x_0} K'$.
\end{enumerate}

We continue with some properties of the binary relation $\preceq_{x_0}$ on $\Upsilon_{x_0}$. 

\begin{lemma} \label{proximal3} If $K\in \Upsilon_{x_0}$, then $\Imp_{x_0}(\conv(K))\in \Upsilon_{x_0}$ and $K\preceq_{x_0} \Imp_{x_0}(\conv(K))$.
\end{lemma}

\begin{proof} That $\Imp_{x_0}(\conv(K))\in \Upsilon_{x_0}$ follows from the property (4), the inclusion $\conv(\Imp_{x_0}(\conv(K))\subseteq \conv(K)$, and since $x_0\notin \conv(K)$.
That  $K\preceq_{x_0} \Imp_{x_0}(\conv(K))$ follows from the fact that $K\subset \Imp_{x_0}(K)/x_0\subseteq \Imp_{x_0}(\conv(K))$ and Lemma \ref{proximal}.
\end{proof}

\begin{lemma} \label{proximal} If $K,K'\in \Upsilon(x_0)$ and  $K\ne K'$, then $K\preceq_{x_0} K'$ if and only if $K/x_0\varsubsetneq K'/x_0$.
\end{lemma}

\begin{proof} If $K\preceq_{x_0} K'$, then $K\subseteq K'/x_0$ and thus $K/x_0\subseteq K'/x_0$ by the definition of shadows. If $K/x_0=K'/x_0$, since $K$ and $K'$ are $x_0$-proximal, necessarily $K=K'$.
\end{proof}

Lemma \ref{proximal} implies that $\preceq_{x_0}$ is transitive on  $\Upsilon(x_0)$ and $\Upsilon^*(x_0)$.
Since $\preceq_{x_0}$ is also reflexive and antisymmetric, 
the following holds: 

\begin{lemma} \label{posets} $(\Upsilon_{x_0}, \preceq_{x_0})$ and $(\Upsilon^*_{x_0}, \preceq_{x_0})$ are partially ordered sets.
\end{lemma}


\begin{definition} [Maximal $x_0$-proximal sets and maximal $\Imp$-convex sets]   $\Max(\Upsilon_{x_0})$ (respectively, $\Max(\Upsilon^*_{x_0})$)
denotes the set of all maximal elements of the partial order $(\Upsilon_{x_0},\preceq_{x_0})$ (respectively, of $(\Upsilon^*_{x_0},\preceq_{x_0})$). 
We call the sets of $\Max(\Upsilon_{x_0})$ \emph{maximal $x_0$-proximal sets} and the convex sets of the form $\conv(K), K\in \Max(\Upsilon_{x_0})$ \emph{maximal $\Imp$-convex sets}.
\end{definition}

By Zorn's lemma, $\Max(\Upsilon_{x_0})\ne \varnothing$ for any vertex $x_0$ in any graph $G$. The maximal $x_0$-proximal sets and the maximal $\Imp$-convex sets have the following property:

\begin{lemma} \label{proximal-property} For a graph $G=(V,E)$, a vertex $x_0$ of $G$, and a set $K\in \Max(\Upsilon_{x_0})$, the following properties hold:
\begin{itemize}
\item[(i)] $K=\Imp_{x_0}(\conv(K))$;
\item[(ii)] $\conv(K)/x_0=K/x_0$;
\item[(iii)]  the set $V$ is the  union of the shadows $K/x_0$ and $x_0/K$.
\end{itemize}
\end{lemma}

\begin{proof} To prove (i), first notice that $x_0\notin \conv(\Imp_{x_0}(\conv(K))$ since $x_0\notin \conv(K)$ and $\Imp_{x_0}(\conv(K))\subseteq \conv(K)$. Therefore $K':=\Imp_{x_0}(\conv(K))$ belongs to $\Upsilon_{x_0}$. Since $K\subseteq \conv(K)$ and $\Imp_{x_0}(K)=K$, by property (5) we conclude that $K\prec K'$. From the choice of $K$ from $\Max(\Upsilon_{x_0})$,  $K=K'$, establishing (i). To prove (ii), first, since $K\subseteq \conv(K)$, we have the inclusion $K/x_0\subseteq \conv(K)/x_0$. Conversely, pick any $z\in \conv(K)/x_0$. Then there exists $y\in \conv(K)$ such that $y\in [x_0,z]$. Since $K=\Imp_{x_0}(\conv(K))$, there exists $y'\in K$ such that $y'\in [x_0,y]$. Consequently, $y'\in [x_0,z]$, proving that $z\in K/x_0$.

Finally, to prove (iii), pick any vertex $v\in V\setminus (x_0/K)$. If $x_0\notin \conv(K\cup \{ v\})$, consider the set $K'=(K\setminus v/x_0)\cup \{ v\}$. Since $K'\subseteq K\cup \{ v\}$, we also have $x_0\notin \conv(K')$.
Since $K\in \Upsilon_{x_0}$ and $K'$ was defined as $K\setminus v/x_0)\cup \{ v\}$, we conclude that $\Imp_{x_0}(K')=K'$. Thus $K'$ satisfies the conditions (P1) and (P2), whence $K'\in \Upsilon_{x_0}$.
Since $K\preceq_{x_0} K'$ and $K$ is different from $K'$, we obtain a contradiction with the assumption $K\in \Max(\Upsilon_{x_0})$. This contradiction shows that $x_0\in \conv(K\cup \{ v\})$ for any $v\in V\setminus (x_0/K)$, i.e.,
$V\setminus (K/x_0)\subseteq x_0/K$. Consequently, $V$ is the union of $x_0/K$ and $K/x_0$.
\end{proof}

\begin{remark} The shadows $K/x_0$ and $x_0/K$ from Lemma \ref{proximal-property}(iii) are not necessarily disjoint.
\end{remark}

%

We continue with the  following characterization of semispaces of $S_3$-graphs.

\begin{theorem} \label{S3proximal} Let $G=(V,E)$ be an $S_3$-graph and $x_0$ be an arbitrary vertex of $G$. If $S$ is a semispace at $x_0$, then $\Imp_{x_0}(S)\in \Max(\Upsilon_{x_0})$ and $S=\Imp_{x_0}(S)/x_0$. Conversely,
if $K\in \Max(\Upsilon_{x_0})$, then $K/x_0$ is a semispace at $x_0$. Consequently, there exists a bijection between the semispaces at $x_0$ and the sets of $\Max(\Upsilon_{x_0})$ and a bijection between the semispaces  of
$\cS_{x_0}$  and the sets of $\Max(\Upsilon^*_{x_0})$.
\end{theorem}

\begin{proof} Let $S$ be a semispace at $x_0$. Since $S$ is a convex set not containing $x_0$ and $\Imp_{x_0}(S)\subseteq S$, $\conv(\Imp_{x_0}(S))$ also does not contain the vertex $x_0$. Hence $\Imp_{x_0}(S)$ is $x_0$-proximal and thus
$\Imp_{x_0}(S)\in \Upsilon_{x_0}$. If we assume that $\Imp_{x_0}(S)\notin \Max(\Upsilon_{x_0})$, then there exists $K\in \Upsilon_{x_0}, K\ne \Imp_{x_0}(S)$ such that $\Imp_{x_0}(S)\preceq_{x_0} K$. Then $x_0\notin \conv(K)$.
By the $S_3$-axiom, $x_0$ and $\conv(K)$ can be separated by complementary halfspaces $H'$ and $H''$, where $x_0\in H'$ and $\conv(K)\subseteq H''$.  By Lemma \ref{shadow-separation}, the shadow $\conv(K)/x_0$
is contained in $H''$. From the definition of $\Imp_{x_0}(S)$ and since $\Imp_{x_0}(S)\preceq_{x_0} K$, we deduce that  $S\subseteq \Imp_{x_0}(S)/x_0\subseteq K/x_0\subseteq \conv(K)/x_0\subseteq H''$. Since $S$ is a
semispace at $x_0$, we must have $S=H''$. This implies that $\Imp_{x_0}(S)/x_0=K/x_0$, contrary to the assumption that $K\ne \Imp_{x_0}(S)$ and $\Imp_{x_0}(S)\preceq_{x_0} K$. This contradiction shows that
$\Imp_{x_0}(S)\in \Max(\Upsilon_{x_0})$. That $S=\Imp_{x_0}(S)/x_0$ follows since the shadow $\Imp_{x_0}(S)/x_0$ is convex (which we prove next), $S$ is included in $\Imp_{x_0}(S)/x_0$ by property (1), and $S$ is a semispace.

Pick any set $K\in \Max(\Upsilon_{x_0})$. We assert that the shadow $K/x_0$ is a semispace at $x_0$ (and thus, is convex). Since $x_0\notin \conv(K)$, by the separation axiom $S_3$,  $x_0$ and $\conv(K)$ can be separated by complementary halfspaces $H'$ and $H''$, where $x_0\in H'$ and $\conv(K)\subseteq H''$.
Let $S$ be a semispace at $x_0$ containing $H''$. Consider the imprint $\Imp_{x_0}(S)$. Since $x_0\notin S$ and $\conv(\Imp_{x_0}(S))\subseteq S$, necessarily
$x_0\notin \conv(\Imp_{x_0}(S))$, whence $\Imp_{x_0}(S)\in \Upsilon_{x_0}$. Since $K/x_0\subseteq H''\subseteq S\subseteq \Imp_{x_0}(S)/x_0$, we must have $K\preceq_{x_0}\Imp_{x_0}(S)$. Since  $K\in \Max(\Upsilon_{x_0})$, we deduce that
$K=\Imp_{x_0}(S)$ and thus $K/x_0=H''=S$. Consequently, any shadow $K/x_0$ with $K\in \Max(\Upsilon_{x_0})$ is convex and furthermore is a semispace at $x_0$. This completes  the proof of the direct implication (convexity of $K/x_0$)
and establishes the converse implication.
\end{proof}

 From Theorem \ref{S3proximal}  it follows that $|\cS_{x_0}|=|\Max(\Upsilon^*_{x_0})|$.  Now, we will compare the number  $|\cS|$ of semispaces of $G$ with the size of the set of pairs $\mathcal M=\{(x_0,K): x_0\in V, K\in \Max(\Upsilon^*_{x_0})\}$.
 Notice that $|\mathcal M|=\sum_{x_0\in V} |\Max(\Upsilon^*_{x_0})|$.


\begin{proposition} \label{numberproxima} If $G=(V,E)$  is a graph with $n$ vertices, then $\frac {|\mathcal M|}{n}\le |\cS|\le |\mathcal M|$.
\end{proposition}

\begin{proof}  Consider any fixed total order on the vertices of $G$. Define the following map $f: \cS\rightarrow \mathcal M$: for $S\in \cS$ we set $f(S)=(x_0,K)$ if (1) $(x_0,K)\in \mathcal M$, (2) $S=K/x_0$, and among all pairs
$(x,K')\in \mathcal M$ satisfying the conditions (1) and (2), the chosen pair $(x_0,K)$ has
$x_0$ with the smallest index (with respect to the total order).  By Lemma  \ref{copoint}, any semispace $S$ is a semispace at a vertex $x$ adjacent to $S$. By Theorem \ref{S3proximal},
there exists $K'\in \Max(\Upsilon^*_{x})$ such that $S=K'/x/$. Therefore, the map $f$ is well-defined. Since each pair $(x,K)$ of $\mathcal M$ defines a single semispace $K/x$, the map $f$ is injective,
establishing the upper bound  $|\cS|\le |\mathcal M|$.  To prove the lower bound, pick any semispace $S\in \mathcal S$. By Lemma \ref{copoint}, $S=K/x_0$ for at least one pair $(x_0,K)\in \mathcal M$. On the other hand, since each pair
$(x_0,K)\in \mathcal M$ defines a single semispace $K/x_0$, the semispace $S$ can be generated as $K/x_0$ by at most $n$ pairs $(x_0,K)\in \mathcal M$. This proves that   $|\cS|\ge \frac {|\mathcal M|}{n}$.
\end{proof}

\subsection{Characterization of $S_3$-graphs}
We prove the following characterization of $S_3$-graphs: 


\begin{theorem} \label{S3graphs} For a graph $G=(V,E)$, the following conditions are equivalent:
\begin{enumerate}
\item[(i)] $G$ is an $S_3$-graph;
\item[(ii)] for any $x_0\in V$ and $K\in \Max(\Upsilon^*_{x_0})$, the shadows $K/x_0$ and $x_0/K$ are convex and disjoint;
\item[(iii)] for any $x_0\in V$ and  $K\in \Max(\Upsilon^*_{x_0})$, $x_0$ and $\conv(K)$  are separable. 
\end{enumerate}
\end{theorem}

\begin{proof} 
To prove (i)$\Rightarrow$(ii), let $K\in \Max(\Upsilon^*_{x_0})$. By Theorem \ref{S3proximal}, $S=K/x_0$ is a semispace with $x_0$ as attaching vertex. Since $G$ is an $S_3$-graph and
$S$ is a semispace of $G$, the complement $V\setminus (K/x_0)$ is convex. By Lemma \ref{proximal-property}, $V\setminus (K/x_0)$ is contained in $x_0/K$. It remains to show that $x_0/K$ coincides with $V\setminus (K/x_0)$.
Indeed, otherwise there exists $z\in K/x_0=S$ belonging to $x_0/K$. Then $x_0\in \conv(K\cup \{ z\})\subset S$, in contradiction with the assumption that $S=K/x_0$ is a convex set  not containing $x_0$. This shows that
$x_0/K=V\setminus (K/x_0)$, establishing (i)$\Rightarrow$(iv).

To prove (ii)$\Rightarrow$(iii), by Lemma \ref{proximal-property}(iii), the sets $K/x_0$ and $x_0/K$ cover the vertex set $V$ of $G$. Since
$K/x_0$ and $x_0/K$ are convex and disjoint, the sets $K/x_0$ and $x_0/K$ are complementary halfspaces. By Lemma \ref{proximal-property}(i)$\&$(ii),
$\conv(K)/x_0=K/x_0$, yielding $\conv(K)\subseteq K/x_0$. Consequently, $x_0$ and $\conv(K)$ are separated by the complementary halfspaces $x_0/K$ and $K/x_0$.

Finally, we  prove (iii)$\Rightarrow$(i). Let $S$ be a semispace of $G$ and let $x_0$ be an attaching vertex of $S$ with $x_0\sim S$ ($x_0$ exists by Lemma \ref{copoint}). Let $K'=\Imp_{x_0}(S)$.
Then $\conv(K')\subseteq S$ and $K'=\Imp_{x_0}(K')$, thus $K'\in \Upsilon^*(x_0)$. Let $K\in \Max(\Upsilon^*_{x_0})$ such that $K'\preceq_{x_0} K$. This implies that $K'\subseteq K/x_0$ and since $K'=\Imp_{x_0}(S)$, we have
$S\subseteq K'/x_0$. This implies that $S\subseteq K/x_0$. By (iii), $x_0$ and $\conv(K)$ can be separated by complementary halfspaces $H',H''$, say $x_0\in H'$ and $\conv(K)\in H''$. Lemma  \ref{shadow-separation}
implies that $\conv(K)/x_0\subseteq H''$. Consequently, $S\subseteq \conv(K)/x_0\subseteq H''$. Since $S$ is a semispace at $x_0$ and $x_0\notin H''$, we must have $S=H''$, yielding $S=\conv(K)/x_0=H''$. Thus,
$S$ is a halfspace of $G$.
\end{proof}

\begin{remark} Conditions (ii) and (iii) of Theorem \ref{S3graphs}  can be considered as a kind of compactness criteria for $S_3$-separation in graphs, similar to Theorem \ref{S4arity} for $S_4$-separation for arity $n$. At the difference
with Theorem \ref{S4arity}, they do not involve  a fixed number of vertices. 
\end{remark}

%
%
%

\subsection{Maximal $x_0$-proximal sets} In this subsection, we provide a characterization of maximal $x_0$-proximal sets.
We start with the following  property of proximal sets.

\begin{lemma} \label{proximal-flag} $\Upsilon_{x_0}$ is a simplicial complex on $V\setminus \{ x_0\}$. All sets $K\in \Max(\Upsilon_{x_0})$ are facets of $\Upsilon_{x_0}$.
\end{lemma}

\begin{proof} Let $K\in \Upsilon_{x_0}$ and $K'\subseteq K$. Since $\Imp_{x_0}(K)=K$ and $x_0\notin \conv(K)$ we will also have $\Imp_{x_0}(K')$ and $x_0\notin \conv(K')$, thus $K'$ satisfies the
conditions (P1) and (P2), whence $K'\in \Upsilon_{x_0}$. Now, if we suppose that  $K\in \Max(\Upsilon_{x_0})$ and $K\subset K'$ for $K'\in \Upsilon_{x_0}$, by property (5) we conclude that $K\preceq_{x_0} K'$,
a contradiction with the maximality choice of $K$. Hence $K$ is a facet of $\Upsilon_{x_0}$.
\end{proof}

The \emph{$1$-squeleton} of a simplicial complex $\mathfrak X$ on $V$ is the graph $G({\mathfrak X})$ having $V$ as the set of vertices and $u,v\in V$ are adjacent in
$G({\mathfrak X})$ if and only if $\{ u,v\}$ is a simplex of $\mathfrak X$. A simplicial complex $\mathfrak X$ is called a \emph{flag} (or \emph{clique) complex} if $\sigma\in {\mathfrak X}$ if and only if $\sigma$ is a clique
of $G(\mathfrak X)$. Any flag complex $\mathfrak X$ can be retrieved from its 1-skeleton $G({\mathfrak X})$ by taking the cliques of $G({\mathfrak X})$  as the simplices of $\mathfrak X$.

\begin{remark}
Note that $\Upsilon_{x_0}$ is not a flag simplicial complex: for the $S_3$-graph $\Gamma$  from Figure \ref{non-convexshadows}, each of the pairs $\{ y,z\},\{ z,w\}$, and $\{ y,w\}$ belong to $\Upsilon_{x_0}$ (because
$x_0$ does not belong to the intervals defined by these three pairs), however $\{ y,z,w\}\notin \Upsilon_{x_0}$ since $u\in [y,w]$ and $x_0\in [u,z]$.
\end{remark}

Not every facet of $\Upsilon_{x_0}$ belongs to $\Max(\Upsilon_{x_0})$. Thus in $\Upsilon_{x_0}$ we have two types of maximality: by inclusion and by the partial order $\preceq_{x_0}$. By Lemma
\ref{proximal-flag},  the maximality by $\preceq_{x_0}$ implies the maximality by inclusion. We continue with the characterization of the facets of $\Upsilon_{x_0}$ that are maximal $x_0$-proximal sets. 
This characterization can be viewed as a kind of local optimality condition:

\begin{proposition} \label{proximal4} For an $S_3$-graph $G=(V,E)$ and a set $K\subseteq V$, $K\in  \Max(\Upsilon_{x_0})$ if and only if $K$ is a facet of $\Upsilon_{x_0}$ satisfying the following condition:

\begin{enumerate}
\item[(P3)] for any $y\in V\setminus K$ such that $R:=(y/x_0)\cap K\ne \varnothing$, we have $x_0\in \conv(K\setminus R\cup \{ y\})$.
\end{enumerate}
\end{proposition}

\begin{proof} First, let $K\in \Max(\Upsilon_{x_0})$. By Lemma \ref{proximal-flag}, $K$ is a facet of $\Upsilon_{x_0}$. To establish (P3) pick any vertex $y\in V\setminus K$. First suppose that $y\in K/x_0=\bigcup_{t\in K} t/x_0$, say $y\in z/x_0$ for $z\in K$.
If $(y/x_0)\cap K\ne \varnothing$, say  $u\in (y/x_0)\cap K$, then we deduce that $z\in [x_0,y]$ and $y\in [x_0,u]$, whence $z\in [x_0,u]$, contrary to the assumption that $K$ satisfies condition (P1).
This proves that $(y/x_0)\cap K=\varnothing$ whenever $y\in K/x_0$. Now suppose that $y\notin K/x_0$ and that $R=(y/x_0)\cap K\ne \varnothing$. Suppose by way of contradiction that
$x_0\notin \conv(K\setminus R\cup \{ y\})$. Let $A=\conv(K\setminus R\cup \{ y\})$ and $K'=\Imp_{x_0}(A)$. Then $K'$ satisfies condition (P1) because $\Imp_{x_0}$ is idempotent (property (4)).  Also $K'$ satisfies
condition (P2) since $\conv(K')\subseteq A$ and $x_0\notin A$. Finally, $K\setminus R\cup \{ y\}\subseteq A\subseteq K'/x_0$ from the definition of imprints. Since $R\subseteq y/x_0\subseteq K'/x_0$, we conclude that
$K\subseteq K'/x_0$, yielding  $K\preceq_{x_0} K'$, a contradiction.

Conversely, suppose by way of contradiction that  a facet $K$ of $\Upsilon_{x_0}$ satisfies condition (P3) but  $K\notin \Max(\Upsilon_{x_0})$, namely suppose that $K\preceq_{x_0} K'$ for $K'\in  \Max(\Upsilon_{x_0})$.
Since $K$ and $K'$ are facets of $\Upsilon_{x_0}$, $K\setminus K'\ne\varnothing$ and $K'\setminus K\ne\varnothing$. By definition of $\preceq_{x_0}$, $K\subseteq x_0/K'=\bigcup_{y\in K'} y/x_0$.
Therefore there exists a vertex $y\in K'\setminus K$ such that $R:=K\cap (y/x_0)\ne \varnothing$. By definition of $\Upsilon_{x_0}$, $R\subseteq K\setminus K'$. By  Theorem \ref{S3proximal}
the shadow $K'/x_0$ is a semispace at $x_0$. Since $y\in K'$ and $K\subseteq K'/x_0$, from the convexity of $K'/x_0$ we conclude that $\conv(K\setminus R\cup \{ y\})\subseteq \conv(K\cup \{ y\})\subseteq K'/x_0$.
Since $x_0\notin x_0/K'$, we also have $x_0\notin \conv(K\setminus R\cup \{ y\})$, violating condition (P3). This concludes the proof.
\end{proof}

The conditions (P1)-(P3) of Proposition \ref{proximal4} can be efficiently tested for any given set $K\subseteq V$, therefore Proposition \ref{proximal4} is an efficient characterization of maximal $x_0$-proximal sets.
Additionally, if $K\in \Upsilon_{x_0}$ does not belong to $\Max(\Upsilon_{x_0})$, by Proposition \ref{proximal4} either there exists $y\notin K$ such that $K\cup \{ y\}\in \Upsilon_{x_0}$ or there exists
$y\notin K$ such that   $R:=(y/x_0)\cap K\ne \varnothing$ and $x_0\notin \conv(K\setminus R\cup \{ y\})$. In the first case, $K\preceq_{x_0} K\cup \{ y\}$ and in the second case $K\preceq_{x_0} K\setminus R\cup \{ y\}$.
This leads us to the following observation:

\begin{corollary}  \label{extensionUpsilon} Given  $K\in {\Upsilon_{x_0}}$ (or $K \in {\Upsilon^*_{x_0}}$) of an $S_3$-graph $G$, one can find in polynomial time in the size of $G$ a set $K'\in \Max({\Upsilon_{x_0}})$ (respectively, $K'\in \Max({\Upsilon^*_{x_0}})$ such that $K\preceq_{x_0} K'$.
\end{corollary}

\begin{example} \label{proximal_diameter}
For general $S_3$-graphs, maximal $x_0$-proximal sets may have
arbitrary diameters. For example, let $x_0$ be a vertex of the odd cycle $G=C_{2n+1}$ and $u,v$ be the neighbors of $x_0$. Then $\Max(\Upsilon^*_{x_0})$ consists of two sets $\{ u,u^*\}$ and $\{ v,v^*\}$ of diameter $n$, where
$u^*$ is the vertex opposite to the edge $x_0u$ and $v^*$ is the vertex opposite to the edge $x_0v$.  The graph $G$ has two semispaces $S',S''$ attached to
$x_0$ (which are both adjacent to $x_0$): $S'$ is the path of length $n$ between $u$ and $u^*$ and $S''$ is the path of length $n$ between $v$ and $v^*$.
\end{example}

\section{$S_3$-graphs satisfying (TC)} \label{s:s3graphstc}
In this section, we provide more efficient characterizations of $S_3$-graphs satisfying the triangle condition (TC) and of their semispaces. Namely,
we show that for such graphs,  the set  $\Max(\Upsilon^*_{x_0})$ of  maximal $x_0$-proximal sets consists precisely of the cliques $K$
such that $K\cup \{x_0\}$ is a maximal clique of $G$.
Together with Theorem \ref{S3proximal}, this characterizes semispaces
of $S_3$-graphs satisfying (TC), allowing to efficiently enumerate them.
Furthermore, we show that in graphs satisfying (TC), the separation axiom
$S_3$ is equivalent to the convexity of the shadows $K/x_0$ and of the extended shadows $x_0\SK K$ 
(for the definition, see below) for all maximal cliques $K\cup \{x_0\}$  of $G$.

\subsection{The structure of semispaces}
We start with the following definition and notations:

\begin{definition}[Pointed maximal clique]
A \emph{pointed maximal clique} of a graph $G$ is a pair $(x_0,K)$, where $K$ is a clique, $x_0$ is a vertex not belonging to $K$, and $K\cup \{ x_0\}$ is a maximal by inclusion clique of $G$. Denote by $\cK_{x_0}$ the set of all cliques $K$ such that
$(x_0,K)$ is a pointed maximal clique. Let also $\cK$ denote the set of all maximal cliques of $G$.
\end{definition}

We continue by showing that in graphs satisfying (TC) the imprint $\Imp_{x_0}(A)$ of a vertex $x_0$ on a convex set $A$ adjacent to $x_0$ is a clique:

\begin{lemma} \label{convex-vertex} Let $G$ be a graph satisfying (TC). If $A$ is a convex set of $G$ and $x_0\notin A$ is a vertex adjacent
to $A$, then $K=N(x_0)\cap A$ is a clique and $K=\Imp_{x_0}(A)$.
\end{lemma}

\begin{proof} Since $A$ is convex and $x_0\notin A$, necessarily $K=N(x_0)\cap A$ is a nonempty clique of $G$. To prove that $K=\Imp_{x_0}(A)$,  it suffices to show that for any $y\in A$ there exists $y'\in K$ such that $y'\in [x_0,y]$. If $y\in K$, then this is obvious. So, let $d(y,x_0)=k>1$. Since $A$ is convex and $x_0\notin A$, we conclude that $d(y,z)\le k$ for any $z\in K$. If there exists a vertex $z\in K$ such that $d(y,z)=k-1$, then $z\in [x_0,y]$ and we can set $y'=z$. Now, let $d(y,z)=k$ for any $z\in K$. Since $d(y,x_0)=d(y,z)=k$,  by triangle condition  there exists a vertex $y'\sim x_0,z$ at distance $k-1$ from $y$. Since $y'\in [z,y]$, we conclude that $y'\in A$, and thus $y'\in K$. Consequently, $K=\Imp_{x_0}(A)$. 
\end{proof}

The following result establishes a bijection between the pointed maximal cliques $(x_0,K)$ and the minimal $x_0$-proximal sets in graphs with triangle condition:

\begin{proposition} \label{TC+proximal} If a graph $G$ satisfies (TC), then $\cK_{x_0}=\Max(\Upsilon^*_{x_0})$  for any vertex $x_0$. Consequently, $|\mathcal M|=\sum_{x_0\in V} |\cK_{x_0}|$.
\end{proposition}

\begin{proof} First,  let $K\in \cK_{x_0}$. Since $K$ is a clique,  $\Imp_{x_0}(K)=K$, thus $K$ satisfies (P1). Since $K$ is convex and does not contain $x_0$, $K$ satisfies (P2). Hence $K\in \Upsilon_{x_0}$.
If $K\notin \Max(\Upsilon^*_{x_0})$, then there exists
a set $K'\in \Upsilon_{x_0}$ such that $K\preceq_{x_0} K'$. Since $x_0$ is adjacent to all vertices of $K$, necessarily $K$ is a proper subset of $K'$. Let $u\in K'\setminus K$. Since $x_0\notin \conv(K')$, $x_0\notin [u,y]$ for any $y\in K$.
Since $K'$ is $x_0$-proximal, $K'=\Imp_{x_0}(K')$, we also have $y\notin [x_0,u]$ for any $y\in K$. Consequently, $d(u,y)=d(u,x_0)=k$ for any $y\in K$. By (TC) there exists a vertex $y'\sim x_0,y$ at distance $k-1$ from $u$. Since $y'\in [u,y]\subset \conv(K')$ and $x_0\notin \conv(K')$,
the vertex $y'$ must be adjacent to all vertices $z\in K\setminus \{ y\}$. Since $y'$ is also adjacent to $x_0$, either $y'\notin K$ and we obtain a contradiction with the assumption that $(x_0,K)$ is a pointed maximal clique or $y'\in K$ and we obtain a contradiction that $u$ has distance $k$ to all vertices of $K$. This establishes that $K\in \Max(\Upsilon^*_{x_0})$.

Conversely, pick any $K\in  \Max(\Upsilon^*_{x_0})$. Since $\Max(\Upsilon^*_{x_0})\subseteq \Max(\Upsilon_{x_0})$,
by Lemma \ref{proximal3}, $K=\Imp_{x_0}(\conv(K))$.  Since $x_0$ is adjacent to $K$ (and thus to $\conv(K)$), by  Lemma \ref{convex-vertex} $\Imp_{x_0}(\conv(K))=N(x_0)\cap \conv(K)$ is a clique.
Thus $K$ is a clique of $G$ whose all vertices are adjacent to $x_0$. If $(x_0,K)$ is not a pointed maximal clique and $(x_0,K')$ is a pointed maximal clique with $K\subsetneq K'$, then $K'\in \Upsilon^*_{x_0}$ 
and $K\preceq_{x_0} K'$ from the definition of $\preceq_{x_0}$. This contradicts that  $K\in \Max(\Upsilon^*_{x_0})$.
\end{proof}

Combining Proposition \ref{TC+proximal} and Theorem \ref{S3proximal}, we obtain the following characterization of semispaces in $S_3$-graphs satisfying (TC):

\begin{theorem} \label{semispaces-trianglecodition} If $G$ is an $S_3$-graph satisfying the triangle condition (TC), then $S$ is a semispace at a vertex $x_0$ adjacent to $S$  if and only if there exists a
pointed maximal clique $(x_0,K)$ such that  $S=K/x_0$. Furthermore, the number $|\cS|$ of semispaces of $G$ and the number $|\cK|$ of maximal cliques of $G$ satisfy the inequality
$\frac {|\cK|}{n}\le |\cS|\le |\cK|$.
\end{theorem}

The inequality $\frac {|\cK|}{n}\le |\cS|\le |\cK|$ in Theorem \ref{semispaces-trianglecodition} follows from equality $|\mathcal M|=\sum_{x_0\in V} |\cK_{x_0}|$ of Proposition \ref{TC+proximal}, inequality $\frac {|\mathcal M|}{n}\le |\cS|\le |\mathcal M|$ of Proposition
\ref{numberproxima} and the fact that each maximal clique of $G$ contains at most $n$ and at least one vertex.

%
%

\begin{definition}[Graph with convex point-shadows]
A graph $G$ is a \emph{graph with convex clique-shadows} if for any pointed maximal clique  $(x_0,K)$ of $G$ the shadow $K/x_0$ is convex.
\end{definition}

From Theorem \ref{semispaces-trianglecodition}, any $S_3$-graph satisfying (TC) is a graph with convex clique-shadows. However, there exist graphs
satisfying (TC) and having convex clique-shadows, which are not $S_3$-graphs. The characterization of semispaces in $S_3$-graphs provided by
Theorem \ref{semispaces-trianglecodition} in fact characterizes the graphs with convex clique-shadows:

\begin{proposition} \label{copoints-triangle-condition} For a graph $G$  satisfying (TC) and having convex intervals, the following conditions are equivalent:
\begin{enumerate}
\item[(i)] $G$ is a graph with convex clique-shadows;
\item[(ii)] the semispaces of $G$ are  the shadows $K/x_0$, where $(x_0,K)$ is a pointed maximal clique and $x_0$ is an attaching point of $K/x_0$.
\end{enumerate}
\end{proposition}

\begin{proof} To prove (i)$\Rightarrow$(ii), let $S$ be a semispace at a vertex $x_0\sim S$. 
Then $K$ is a clique of $G$. We assert that $(x_0,K)$ is a pointed maximal clique and that $S=K/x_0$.
Let $K^+$ be any maximal clique of $G$ containing $K\cup \{ x_0\}$ and let $K'=K^+\setminus \{ x_0\}$.  Then $K'/x_0$ is convex by (i). By Lemma \ref{convex-vertex} and since $K\subseteq K'$, we have $S\subseteq K/x_0\subseteq K'/x_0$.
Since $S$ is a semispace at $x_0$, necessarily $S=K'/x_0$. Consequently, $S=K/x_0=K'/x_0$, which implies that $K=K'$ and $K^+=K\cup \{ x_0\}$. Hence $(x_0,K)$ is a pointed maximal clique and $S=K/x_0$.
Now, let $(x_0,K)$ be a pointed maximal clique of $G$. Consider the shadow $K/x_0$, which is convex by condition (i). Let $S$ be a semispace at $x_0$ containing the convex set $K/x_0$. By Lemma \ref{convex-vertex}, $K'=N(x_0)\cap S$ is a clique,
such that $S\subseteq K'/x_0$. Since $K\subset K'$ and $(x_0,K)$ is a pointed maximal clique, we conclude that $K'=K$. Consequently, $S\subseteq K/x_0$, establishing that $S=K/x_0$.  This shows  (i)$\Rightarrow$(ii).

To prove (ii)$\Rightarrow$(i), let $(x_0,K)$ be a pointed maximal clique. We assert that  $K/x_0$ is convex. Let $S$ be a semispace at $x_0$ containing the convex set $K$.
By condition (ii), $S=K'/x_0$, where $K'\cup \{ x_0\}$ is a maximal clique of $G$. Since $K\subseteq S$, $K\subseteq K'$. If $x_0$
has a neighbor $y\in K'\setminus K$, by maximality of $(x_0,K)$,  $K$ contains a vertex $z$ not adjacent to $y$. But then $x_0\in [y,z]$, contrary to the convexity of $S$.
Thus $K=N(x_0)\cap S=K'$. Consequently, $S=K/x_0$, whence $G$ is a graph with convex clique-shadows. 
\end{proof}

We conclude this subsection with the following example of graphs with convex clique-shadows:

\begin{proposition} \label{JHC-convexcliqueshadow}  If $G$ is a JHC-graph satisfying (TC),  then $G$ is a graph with convex clique-shadows.
\end{proposition}

\begin{proof} Let $(x_0,K)$ be a pointed maximal clique and let $S$ be a semispace at $x_0$ containing $K$. Then $K=N(x_0)\cap S$. Indeed, if $N(x_0)\cap S$ contains a vertex $z\notin K$, then
since $x_0\notin S$ and $K\subseteq S$, we conclude that $z$ is adjacent to all vertices of $K\cup \{ x_0\}$, contrary to the assumption that $(x_0,K)$ is a pointed maximal clique. Thus $K=N(x_0)\cap S$.
By Lemma \ref{convex-vertex}, $S\subseteq K/x_0$. To establish the converse inclusion, suppose by way of contradiction that there exists  $u\in (K/x_0)\setminus S$, say $s\in [x_0,u]$ for $s\in K$.
Since $S$ is a semispace at $x_0$, we get $x_0\in \conv(S\cup \{ u\})$. Since $G$ is a JHC-graph, there exists a vertex $v\in S$ such that $x_0\in [u,v]$. By Lemma \ref{convex-vertex},
there exists a vertex $t\in K$ such that $t\in [x_0,v]$.
Since $s\in [x_0,u]$ and $s$ and $t$ are adjacent, we get $d(u,s)+d(s,t)+d(t,v)=d(u,s)+1+d(v,t)<d(u,s)+2+d(t,v)=d(u,x_0)+d(x_0,v)$, contrary to the assumption  $x_0\in [u,v]$. This establishes that $S=K/x_0$ and thus that
the shadow $K/x_0$ is convex.
\end{proof}

\begin{remark} The third graph of Figure \ref{non-S3} satisfies (TC), is JHC, but not $S_3$.
\end{remark}
%


\subsection{Characterization of $S_3$-graphs} The goal of this subsection is to characterize $S_3$-graphs satisfying (TC) in terms of convexity of shadows defined with respect to
pointed maximal cliques $(x_0,K)$ of $G$.  From Lemma \ref{proximal-property} and Proposition \ref{TC+proximal}, for any pointed maximal clique $(x_0,K)$, the shadows $x_0/K$ and $K/x_0$
cover the vertex set $V$ of $G$. The next lemma gives a different description of the set $V\setminus (K/x_0)$:

\begin{lemma} \label{extended} Let $K'$ be a clique of a graph $G=(V,E)$. Then for any vertex $x_0\in K'$ and $K=K'\setminus \{ x_0\}$, the vertex-set $V$ of $G$ is the disjoint union of the sets $K/x_0$, $W_=(K'):=\{ v\in V: d(v,y)=d(v,z) \mbox{ for all } y,z\in K'\}$, and
$x_0|K:=\bigcup_{y\in K}~~ x_0/y$.
\end{lemma}

\begin{proof} Since $x_0\sim y$ for any $y\in K$, for any vertex $z$ of $G$, we have $|d(z,x_0)-d(z,y)|\le 1$. Therefore, for any $z\in V$,  exactly one of the possibilities holds:
\begin{enumerate}
\item[(1)] $d(z,y)<d(z,x_0)$ for some $y\in K$, in which case $z\in y/x_0\subseteq K/x_0$;
\item[(2)]  $d(z,x_0)<d(z,y)$ for some $y\in K$, in which case $z\in x_0/y\subseteq x_0|K$;
\item[(3)] $d(z,x_0)=d(z,y)$ for any $y\in K$, in which case $z\in W_=(K')$.
\end{enumerate}
This concludes the proof.
\end{proof}

We refer to the set $x_0|K=\bigcup_{y\in K}~~ x_0/y$ as the \emph{union shadow} and we call $x_0\SK K:=x_0|K \cup W_=(K\cup \{ x_0\})$ the \emph{extended shadow} of $x_0$ with respect to $K$. By Lemma \ref{extended}, $V$ is the disjoint union
of the shadow $K/x_0$ and of the extended shadow $x_0\SK K$.  Notice that $V$ can be also written as the union of $W_=(K\cup \{ x_0\})$ and the shadows $K/x_0$  and $x_0/K$. However, while the sets $W_{=}(K\cup \{ x_0\})$ and $K/x_0$ are disjoint, the shadow $x_0/K=\{ u\in V: x_0\in \conv(K\cup \{ u\})$ is
not necessarily disjoint from $W_=(K\cup \{ x_0\})$ and from $K/x_0$. 
%
Using all this, we can characterize the $S_3$-graphs with (TC) as follows:

\begin{theorem} \label{S3graphs-trianglecondition} For a  graph $G=(V,E)$ satisfying (TC) the following conditions are equivalent:
\begin{enumerate}
\item[(i)] $G$  is an $S_3$-graph;
\item[(ii)] for any pointed maximal clique $(x_0,K)$, the shadows $x_0/K$ and $K/x_0$ are convex and disjoint;
\item[(iii)] for any pointed maximal clique $(x_0,K)$, $x_0$ and $K$ are separable; 
\item[(iii)]  for any pointed maximal clique $(x_0,K)$,  the shadow $K/x_0$ and the extended shadow  $x_0\SK K$ are convex.
\end{enumerate}
\end{theorem}

\begin{proof} The equivalence between the conditions (i), (ii), and (iii) follows from  Theorem \ref{S3graphs} and Proposition \ref{TC+proximal}.  
To prove the implication (i)$\Rightarrow$(iv), let $G$ be an $S_3$-graph satisfying (TC) and let $(x_0,K)$ be a pointed maximal clique of $G$. By Theorem \ref{semispaces-trianglecodition},
$K/x_0$ is a semispace with attaching point $x_0$. Since $G$ is an $S_3$-graph, the complement of $K/x_0$ is convex by Proposition \ref{S3graphs}. By Lemma \ref{extended},  $V\setminus (K/x_0)$ is the extended shadow
$x_0\SK K$. Consequently, both sets $K/x_0$ and $x_0\SK K$ are convex.

To prove the implication (iv)$\Rightarrow$(i), let $G$ be a graph satisfying (TC) such that $K/x_0$ and   $x_0\SK K$ are convex for any pointed maximal clique $(x_0,K)$. Let
$S$ be a semispace with an adjacent attaching vertex $x_0$. Since the shadows $K/x_0$ are convex for all pointed maximal cliques $(x_0,K)$, we can apply Proposition \ref{copoints-triangle-condition}.
By this result, we have $S=K'/x_0$, where $K'=N(x_0)\cap S$ and $(x_0,K')$ is a pointed maximal clique of $G$. Since the extended shadow $x_0\SK K'$ is convex and coincides with $V\setminus (K/x_0)=V\setminus S$,
$S$ is a halfspace of $G$.
\end{proof}

%
%

\begin{remark} Theorems \ref{semispaces-trianglecodition} and \ref{S3graphs-trianglecondition}  
provide simple structural characterizations  of $S_3$-graphs satisfying  (TC) and of their semispaces.
Most likely, Theorem \ref{S3graphs-trianglecondition} is the best one can get for $S_3$-graphs.  Nevertheless, it does not provide a polynomial-time
recognition algorithm of $S_3$-graphs satisfying (TC). We are inclined to believe that the following question has a positive answer:
\end{remark}

\begin{question} Are $S_3$-graphs satisfying (TC) recognizable in polynomial time?
\end{question}

\begin{remark}
The structure of semispaces in non-$S_3$-graphs $G$ satisfying (TC) is much more complicated and widely open. In particular,
the semispaces are not of the form $K/x_0$, where $K\cup \{ x_0\}$ is a (non-necessarily
maximal) clique of $G$ and $G$ may have a polynomial number of maximal cliques but an exponential number of semispaces.
\end{remark}


\begin{remark} The meshed (in fact, weakly modular) graph $\Gamma$ from Figure \ref{non-convexshadows} shows that in the characterization of $S_3$-graphs from Theorem \ref{S3graphs-trianglecondition}
we cannot replace the requirement that the extended shadow $K\SK x$ is convex by the requirement that the union shadow $K|x$ is convex (which would be a natural relaxation of
condition (iii) of Proposition \ref{S3graphs}). Indeed, one can check that the convexity of $H$ is $S_3$. However, if we set $K=\{ y,z\}$, then $u,v\in K|x$, while $w\in [u,v]\setminus (K|x)$.
Notice that $w$ has distance 2 from all vertices of the maximal clique $K'=\{x,y,z\}$. Thus the set $W_=(K')$ is nonempty because it contains the vertex $w$. In many subclasses of weakly modular graphs
(such as bridged or Helly graphs), $W_=(K')=\varnothing$ for any maximal clique $K'$. Therefore, for such classes of graphs, $S_3$ is equivalent to the convexity of the shadow $x/K$ and of the union
shadow $K|x$ for each pointed maximal clique $(x,K)$.

Finally notice that the convexity of $H$ is not JHC:  indeed, $w\ast \{ y,z\}=w\ast [y,z]=[w,y]\cup [w,z]=\{ w,u,s,t,v,y,z\}$, however $x\in [u,v]\subset \conv(w,y,z)$. This shows that the separation property $S_3$ in
meshed (weakly modular graphs) does not imply JHC.
\end{remark}

\begin{remark} Theorem \ref{S3graphs-trianglecondition} can be viewed as a generalization of Proposition \ref{copoints-bipartite}, characterizing
the bipartite $S_3$-graphs as bipartite graphs in which  the shadows $x/y$ and $y/x$ are convex for any two adjacent vertices $x$ and $y$.  Indeed, bipartite
graphs satisfy (TC) because they do not contain triplets $u,v,w$ such that $d(x,y)=d(x,z)$ and $d(y,z)=1$.
Furthermore, the edges are exactly the maximal cliques of bipartite graphs. Finally, $W_=(xy)=\varnothing$ for any edge $xy$.
\end{remark}

\begin{figure}
  \begin{center}
    \begin{tikzpicture}[x=1cm,y=1cm]

\filldraw[black]  (0,0) circle (2pt);
\filldraw[black]  (-1.5,-1.5) circle (2pt);
\filldraw[black]  (1.5,-1.5) circle (2pt);
\filldraw[black]  (-0.3,-1.2) circle (2pt);
\filldraw[black]  (0.4,-1.7) circle (2pt);
\filldraw[black]  (0,-3.2) circle (2pt);
\filldraw[black]  (-1.2,-4) circle (2pt);
\filldraw[black]  (1.2,-4) circle (2pt);
\draw (0,0) -- (-1.5,-1.5);
\draw (0,0) -- (1.5,-1.5);
\draw (0,0) -- (-0.3,-1.2);
\draw (0,0) -- (0.4,-1.7);
\draw (-1.5,-1.5) -- (-1.2,-4);
\draw (-1.2,-4) -- (1.2,-4);
\draw (1.5,-1.5) -- (1.2,-4);
\draw (0,-3.2) -- (-1.2,-4);
\draw (0,-3.2) -- (1.2,-4);
\draw (-0.3,-1.2) -- (-1.5,-1.5);
\draw (-0.3,-1.2) -- (1.5,-1.5);
\draw (-0.3,-1.2) -- (0,-3.2);
\draw (0,-3.2) -- (-1.5,-1.5);
\draw (0,-3.2) -- (1.5,-1.5);
\draw (0.4,-1.7) -- (1.2,-4);
\draw (0.4,-1.7) -- (-1.2,-4);
\draw (-1.5,-1.5) -- (0.4,-1.7);
\draw (-0.3,-1.2) -- (-1.2,-4);

\node[] at (0,0.3) {$w$};
\node[] at (-1.5,-1.3) {$u$};
\node[] at (1.5,-1.3)  {$v$};
\node[] at (-0.4,-1.)  {$s$};
\node[] at (0.6,-1.5)  {$t$};
\node[] at  (0.4,-3.2)    {$x_0$};
\node[] at (-1.,-4.15)  {$y$};
\node[] at  (1.,-4.15) {$z$};
\draw[color=red!80, very thick](0,-3.2) circle (4pt);
\draw [color=red!80, very thick] plot [smooth cycle] coordinates {(-1.2,-4.3) (-1.2,-3.7)  (1.2,-3.7)  (1.2,-4.3)};
  \end{tikzpicture}
  \end{center}
  	\caption{A meshed $S_3$-graph $\Gamma$ with non-convex union shadow $K|x_{0}$ with $K=\{ y,z\}$}\label{non-convexshadows}
\end{figure}
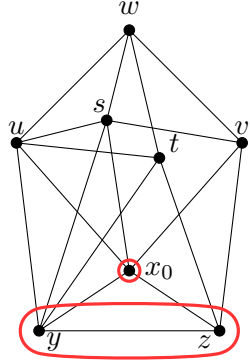

\subsection{Enumeration of semispaces}

In this subsection, using Theorem \ref{semispaces-trianglecodition} we show how to efficiently enumerate the semispaces of $S_3$-graphs satisfying (TC). 


\subsubsection{Enumeration of semispaces in convexity spaces}
Finite convexity spaces (under the guise of closure spaces) have found numerous applications in Computer Science: in formal concept analysis \cite{GaWi}, database theory, and
propositional logic \cite{Kh}. Implication bases and intersection bases are two adopted ways of  compact representation of  a finite convexity space $(X,\C)$. An \emph{implication} is an expression $A\rightarrow B$, where $A$ and
$B$ are subsets of $X$ modeling a causality relation between $A$ and $B$ in  $\C$:  if a convex set includes $A$, it must also include $B$ (this is equivalent to $B\subseteq \conv(A)$). Every convexity space $(X,\C)$ can be represented by a set $\Sigma$ of implications, called an \emph{implication base} such that the convexity space defined by this set of implications coincides with the input space $\C$. 
On the other hand, as we mentioned in the introduction, the semispaces of $(X,\C)$  constitute  the unique minimal collection of sets of $\C$ from which  $\C$ can be reconstructed
by taking intersections. In lattice theory, semispaces are called \emph{meet-irreducibles} \cite{DaPr} and in Horn logic they are called \emph{characteristic models} \cite{Kh}. Khardon \cite{Kh} formulated the problem of algorithmic translation between the
implication base and the intersection base of a closure/convexity space.\footnote{I learned about this problem from my colleague Oscar Defrain.} The question of existence of output polynomial algorithms is open in both directions. In particular, it is open for geodesic convexity in graphs. 
On the other hand, it was shown in \cite{LaLeKa} that there does not exist an output-polynomial algorithm for enumerating the facets of a finite simplicial complex $(X,{\mathfrak X})$ unless $P=NP$ (given the subrutine which indicates in unit time whether or not a given subset of $X$ is a simplex of $\mathfrak X$). Such algorithms exist for flag simplicial complexes, i.e., for maximal cliques of a graph \cite{JhYaPa,TsIdArSh}. Since the number
of semispaces of the simplicial convexity of $(X,{\mathfrak X})$ is linearly related to the number of facets of $\mathfrak X$ and each semispace is either a facet or is derived from a facet
by removing one element (Example \ref{exemple-simplicial-convexity}), the result
of \cite{LaLeKa} implies that enumerating semispaces of the simplicial convexity 
is hard (but the translation from implication bases to intersection bases is easy since any implication base must contains an implication for each facet of $\mathfrak X$). It is also hard to enumerate semispaces of a
convexity space attached to a given point $x_0$ \cite{KaSiSt}. In case of geodesic convexity in graphs (and, more generally, in finite convexity
spaces of fixed arity), any implication base will have length polynomial in the size of the vertex-set $V$, therefore the translation from the implication base to the intersection
base is equivalent to the problem of enumeration of the semispaces of a graph $G$ in time polynomial in the number
of vertices and the number of semispaces of $G$. For other related results and bibliography on this topic, see  Vilmin \cite[Chapter 2]{Vi}.


\subsubsection{Enumeration of semispaces of $S_3$-graphs satisfying (TC)}
Let $G$ be an $S_3$-graph satisfying (TC). By Theorem \ref{S3proximal}, there is a bijection between the set $\cS_{x_0}$ of semispaces attached at $x_0$ and adjacent to $x_0$ and the set $\Max(\Upsilon^*_{x_0})$
of maximal $x_0$-proximal sets adjacent to $x_0$. By Theorem \ref{semispaces-trianglecodition}, $\Max(\Upsilon^*_{x_0})$ is the set of all maximal cliques of the subgraph $G(N(x_0))$ of $G$
induced by the neighborhood $N(x_0)$ of $x_0$. Therefore, the semispaces of $\cS_{x_0}$  can be enumerated in the following way. First we enumerate all maximal cliques of the graph $G(N(x_0))$
using an algorithm for enumerating the maximal cliques of a graph with polynomial delay. For this, one can use the algorithms of \cite{TsIdArSh} or \cite{JhYaPa} for enumerating maximal independent sets
of a graph. For each such maximal clique $K$ of $G(N(x_0))$, $(x_0,K)$ is a pointed maximal clique of $G$, and by Theorem \ref{semispaces-trianglecodition}, the shadow $K/x_0$ is a semispace at $x_0$ and adjacent to $x_0$. Furthermore,
for any two different maximal cliques $K,K'$ of $G(N(x_0))$, $K/x_0$ and $K'/x_0$ are different semispaces and each semispace of $\cS_{x_0}$ occurs in this way.

To enumerate the semispaces $S\in {\mathcal S}$ of an $S_3$-graph $G$ satisfying (TC) in output polynomial time, we fix a total order on the vertices of $G$ and enumerate the maximal cliques of $G$ using \cite{TsIdArSh} or \cite{JhYaPa}.
Then we sort the set $\cK$ of maximal cliques of $G$ in lexicographic order (alternatively, the algorithm of \cite{JhYaPa} generate the maximal cliques of $G$ in increasing lexicographic order). By Theorem \ref{semispaces-trianglecodition},
$\frac {|\cK|}{n}\le |\cS|\le |\cK|$, thus for an output polynomial algorithm for $\cS$ we can first enumerate and then sort $\cK$. For each  $K'\in \cK$ we define an array $A(K')$ with $|K'|$ entries, corresponding to the $|K'|$
pointed maximal cliques $(x_0,K)$ with $K'=K\cup \{ x_0\}$ and $x_0\notin K$. The corresponding entry of $A(K')$ contains  $(x_0,K)$ and  a boolean variable $\varphi(x_0,K)$, which we initialize by setting $\varphi(x_0,K)=0$. 
We traverse the lexicographically ordered list $\cK$ again and perform the following operation. Let $K'$ be the current clique of $\cK$. We traverse the array $A(K')$ and consider each pointed maximal
clique $(x_0,K)$ with $K'=K\cup \{ x_0\}$ and such that $\varphi(x_0,K)=0$ (if $\varphi(x_0,K)=1$ for all $(x_0,K)$ with $K'=K\cup \{ x_0\}$, then we pass to the next clique of the list $\cK$). Then, following Theorem \ref{semispaces-trianglecodition} we return the shadow $S=K/x_0$ to the output list of semispaces of $G$ and set $\varphi(x_0,K)=1$. Furthermore, we consider each vertex $y_0\in V\setminus S$ which is adjacent to $S$ and the set $L=N(y_0)\cap S$ (since $S$ is convex, $L$ is a clique). Then we test if $L':=L\cup \{ y_0\}$ is a maximal clique of $G$ and if this is the case, then we construct the shadow $S'=L/y_0$. If $S'=S$ (this means that the semispace $S'$ has been already discovered as $S$), then using \emph{binary search} in the sorted list $\cK$ we find the entry for  $L'$ and  in the array $A(L')$ of $L'$ we search for the entry of the pointed maximal clique $(y_0,L)$ and set $\varphi(y_0,L)=1$. This ensures that each semispace is enumerated only once.

For each $K'\in \cK$, the complexity of all steps done to discover new semispaces of the form $K/x_0$ with $K'=K\cup \{ x_0\}$ is polynomial in $n$, denote it by $\poly(n)$. Indeed, for each $y_0$ adjacent to $K/x_0$,  computing the clique $L$, testing if $L'=L\cup \{ y_0\}$ is a maximal clique, computing the shadow $L/y_0$ and testing if $L/y_0=K/x_0$ requires polynomial time. Finally, the binary search over the sorted list $\cK$ of maximal cliques requires $\log |\cK|=O(n)$ time. The correctness of the algorithm
follows from Theorem \ref{semispaces-trianglecodition}. Consequently, we obtain the following result:

\begin{theorem} \label{enumerate} The set $\cS$ of semispaces of an $S_3$-graph $G$ with $n$ vertices and satisfying (TC) can be computed in  $O(\poly(n)|\cS|)=O(\poly(n)|\cK|)$, where $\cK$ is the set of maximal cliques of $G$.
\end{theorem}

\begin{remark} Our algorithm is output polynomial but we do not know if it has polynomial delay (i.e., whether the time between the  enumeration of two consecutive semispaces is polynomial).
\end{remark}

\begin{remark}  The $S_3$-graphs satisfying (TC)  may have an exponential number of cliques and thus of semispaces.  This is the case of the $d$-dimensional hyperoctahedron $O_d$. One can easily check that $O_d$ satisfies the triangle and the quadrangle conditions (and thus $O_n$ is weakly modular, see the definition below). As noticed in Example \ref{ex-hyper}, $O_d$ contains $2d$ vertices and $2^d$ semispaces corresponding to the $2^d$ maximal cliques.
Since the class of $S_3$-graphs satisfying (TC) is closed by taking Cartesian products and gated amalgamations, we can construct other examples of $S_3$-graphs with an exponential number of cliques. $O_d$ is the
graph $K_{2,\ldots,2}$, where the graph $K_{a_1,\ldots,a_k}$ consists of $k$ disjoint independent sets $A_1,\ldots,A_k$ of sizes $a_1,\ldots,a_k$ and all edges $uv$ with $u\in A_i,v\in A_j$ and $i\ne j$. The graph $K_{3,\ldots,3}$ with $k=\frac{n}{3}$ is
the \emph{Moon-Moser graph} and is known to have the maximum number of maximal cliques among all graphs with $n$ vertices. While all graphs $K_{a_1,\ldots,a_k}$ are weakly modular (and thus meshed and satisfy (TC)) most of them are not $S_3$.
\end{remark}

\begin{remark}
To enumerate the semispaces of an $S_3$-graph $G$ we have to be able to enumerate the sets of $\Max(\Upsilon^*_{x_0})$, i.e.,  the maximal $x_0$-proximal sets of $G$ adjacent to $x_0$.
However,  to $\Max(\Upsilon^*_{x_0})$ one cannot associate a flag simplicial complex (as was the case of graphs satisfying (TC)).
Therefore, we cannot use the method of   \cite{TsIdArSh} and \cite{JhYaPa} for enumerating maximal cliques. The \emph{proximity search} proposed in \cite{CoGrMaUnVe} is a more general and powerful enumeration paradigm, which applies
to the enumeration of all maximal by inclusion solutions of some problem on a finite universe.  Roughly speaking, 
the proximity search constructs a \emph{neighboring relation} between the maximal solutions 
and a \emph{proximity relation} $\theta$ between all pairs $R,R'$ of maximal solutions 
such that each maximal solution has a polynomial number of neighbors and for each pair $R,R'$ of maximal solutions, $R$ has a neighbor $R_0'$ such that $|\theta(R_0',R')|>|\theta(R,R')|$.
While the family $\Max(\Upsilon^*_{x_0})$ does not consists exactly of the maximal by inclusion subsets
of $\Upsilon^*_{x_0}$, we still hope (based on Proposition \ref{proximal4} and Corollary \ref{extensionUpsilon}) that
the proximity search can be used to enumerate $\Max(\Upsilon^*_{x_0})$. Therefore, we believe that the following
question has a positive answer:
\end{remark}

\begin{question} Design an output polynomial  algorithm for enumerating the semispaces of an $S_3$-graph.
\end{question}

\section{Meshed graphs}\label{s:meshed}
In this section, we recall the definitions and the properties of meshed and weakly modular graphs. We also establish that in meshed graphs local convexity implies convexity
and that pre-median meshed graphs are fiber-complemented. The two results were previously known only for weakly modular graphs.  The first result is heavily used in the proof of
the characterization of $S_3$-meshed graphs via forbidden subgraphs. Since $S_3$-meshed graphs are pre-median,  the second result sheds some light about the product-amalgamation structure of
$S_3$-meshed graphs. We hope that both results will be also used in other contexts.

\subsection{Meshed and weakly modular graphs}

A graph $G$ is called \emph{weakly modular}  \cite{BaCh_helly,Ch_metric} if $G$ satisfies the triangle condition (TC) and the following quadrangle condition:
\begin{enumerate}
\item[(QC)] for any $u,v,w,z\in V$ with
$d(v,z)=d(w,z)=1$ and  $2=d(v,w)\leq d(u,v)=d(u,w)=d(u,z)-1,$ there
exists a common neighbor $x$ of $v$ and $w$ such that
$d(u,x)=d(u,v)-1.$
\end{enumerate}

Three vertices $v_1,v_2,v_3$ of a graph $G$ form a {\it metric triangle}
$v_1v_2v_3$ if the intervals $[v_1,v_2], [v_2,v_3],$ and
$[v_3,v_1]$ pairwise intersect only in the common end-vertices, i.e.,
$[v_i, v_j] \cap [v_i,v_k]= \{v_i\}$ for any $1 \leq i, j, k \leq 3$.
If $d(v_1,v_2)=d(v_2,v_3)=d(v_3,v_1)=k,$ then this metric triangle is
called {\it equilateral} of {\it size} $k.$ An equilateral metric triangle
$v_1v_2v_3$ of size $k$ is called \emph{strongly equilateral} if $d(v_1,v)=k$ for
all $v\in [v_2,v_3]$. Recall the following characterization of weakly modular graphs:

\begin{lemma} \label{strongly-equilateral}  \cite{Ch_metric} A graph $G$ is weakly modular
if and only if all metric triangles of $G$ are strongly equilateral.
\end{lemma}

A metric triangle
$v_1v_2v_3$ of $G$ is a {\it quasi-median} of the triplet $x,y,z$
if the following metric equalities are satisfied:
$$\begin{array}{l}
d(x,y)=d(x,v_1)+d(v_1,v_2)+d(v_2,y),\\
d(y,z)=d(y,v_2)+d(v_2,v_3)+d(v_3,z),\\
d(z,x)=d(z,v_3)+d(v_3,v_1)+d(v_1,x).\\
\end{array}$$
If $v_1$, $v_2$, and $v_3$ are the same vertex $v$,
or equivalently, if the size of $v_1v_2v_3$ is zero,
then this vertex $v$ is called a {\em median} of $x,y,z$.
While a median may not exist and may not be unique,  a quasi-median of
every triplet $x,y,z$ always exists.
%
We continue with the definition of meshed graphs. They have been introduced in
the unpublished paper \cite{BaMuSo} and further studied in the papers \cite{BaCh_median,BeChChVa,ChChChJa,CCHO,Ch_delta}.

\begin{definition} [Meshed graph] A graph $G=(V,E)$ is called {\it meshed}  if for any vertex $u$
its distance function $d$ satisfies the following \emph{Weak Quadrangle Condition}:

\begin{enumerate}
\item[(QC$^-$)]  for any  $u,v,w\in V$ with
$d(v,w)=2,$ there exists a common neighbor $x$ of $v$ and $w$ such that $2d(u,x)\le d(u,v)+d(u,w).$
\end{enumerate}
\end{definition}

(QC$^-$) seems to be a relaxation of (QC), but it implies the triangle condition (TC) (that is not implied
by (QC). Conversely, (TC) and (QC) imply
(QC$^-$) and thus weakly modular graphs are meshed. Furthermore, in meshed graphs,
any metric triangle $xyz$ is equilateral~\cite{BaCh_median}.

\subsection{Properties of meshed graphs} \label{classes}

For completeness, in this subsection we recall and prove some basic properties of meshed graphs.

Convexity of the distance function is an important property in metric geometry: the radius (distance) function of CAT(0) spaces is convex and the Busemann spaces
are defined by the convexity of the distance function between any two geodesics. In the discrete setting of graphs $G=(V,E)$,
recall that a function $f:V\rightarrow {\mathbb R}$ is called
\emph{weakly convex} if for any two vertices $u,v,$ and a real
number $\lambda$ between $0$ and $1$ such that $\lambda d(u,v)$ and
$(1-\lambda)d(u,v)$ are integers, there exists a vertex $x$ such that
$d(u,x)=\lambda d(u,v), d(v,x)=(1-\lambda)d(u,v),$ and
$f(x)\leq (1-\lambda )f(u)+\lambda f(v)$. Weakly convex functions can be
characterized in the following local-to-global way:

\begin{lemma} \label{wcf} \cite{BaCh_median} For a real-valued function $f$ defined on the vertex-set of a graph $G$ the
following conditions are equivalent:

\begin{enumerate}
\item[(i)]  $f$ is weakly convex;
\item[(ii)] for any two non--adjacent vertices $u$ and $v$ there exists
$w\in [u,v], w\ne u,v,$ such that
$d(u,v)\cdot f(w)\leq d(v,w)\cdot f(u)+d(u,w)\cdot f(v)$;
\item[(iii)] any two vertices $u$ and $v$ at distance 2 have a common
neighbour $w$ with $2f(w)\leq f(u)+f(v)$.
\end{enumerate}
\end{lemma}

In view of this lemma, meshed graphs are characterized by weak convexity
property of the radius functions $r_u(v)=d(v,u)$ for all $u\in V$. This
condition ensures that all balls in  a meshed graph $G$ induce
isometric subgraphs of $G$:

\begin{lemma} \label{balls-isometric} If $G$ is a meshed graph,
then  for any vertex $u$ and any integer $k\ge 0$, the ball $B_k(u)$ induces an isometric subgraph of $G$.
\end{lemma}

\begin{proof} Pick any $x,y\in B_k(u)$. To prove that $x$ and $y$ can be connected
by a shortest path in $B_k(u)$ it suffices to prove that there exists  $z\in [x,y]\cap B_k(u)$ different from $x,y$. Indeed, then reiterating this to the pairs $x,z$ and $z,y$, we will
construct the required shortest path inside $B_k(u)$.   Since $G$ is meshed, the radius function  $r_u(v)=d(v,u)$ is weakly convex. By Lemma \ref{wcf}, there exists $z\in [x,y], z\ne x,y,$
such that $d(x,y)\cdot d(u,z)\leq d(y,z)\cdot d(u,x)+d(z,x)\cdot d(u,y)$. Since $d(u,x)\le k$ and $d(u,y)\le k$, we get $d(x,y)\cdot d(u,z)\leq k(d(y,z)+d(z,x))=kd(x,y)$ and thus $d(u,z)\le k$.
\end{proof}

\begin{lemma} \label{equilateral} \cite[Remark 2]{BaCh_median} All metric triangles of a meshed graph $G=(V,E)$ are equilateral.
In particular, meshed graphs satisfy the triangle condition (TC).
\end{lemma}

\begin{proof}
To see this, suppose the contrary: let $uvw$ be a metric triangle in $G$ with
$d(u,w)<d(v,w).$ The definition of weak convexity applied to $f=d(\cdot,w),$
the pair $u,v,$ and $\lambda=1-1/d(v,w)$ then provides us with a neighbour
$x$ of $v$ which necessarily belongs to $[u,w]\cap [v,w],$ a contradiction.
\end{proof}

\begin{lemma} \label{equilateral_bis} If $uvw$ is a metric triangle of size $k$ of a meshed graph $G$, then there exists a shortest path $P(v,w)$ such that $d(u,x)=k$ for all $x\in P(v,w)$.
\end{lemma}

\begin{proof} Let $P(v,w)=(v=v_0,v_1,\ldots,v_k=w)$ be a shortest path between $v$ and $w$ along which the radius function $r_u$ is convex (such a shortest path exists by Lemma \ref{wcf}).
This implies that $d(u,v_i)\le k$ for any $v_i\in P(v,w)$. Since $[v,u]\cap [v,w]=\{ v\}$, we conclude that $d(u,v_1)=k$. Suppose that $P(v,w)$ contains a vertex $v_i$ such that $d(u,v_i)<k$.
Since the function $r_u$ is convex on $P(v,w)$ and $r_u(v_0)=r_u(v_1)=k$ and $r_u(v_i)<k$ we obtain a contradiction with the inequality $k=r_u(v_1)\leq (1-\lambda )r_u(v_1)+\lambda r_u(v_i)<k$.
\end{proof}

\begin{lemma} \cite{ChChChJa} A graph
  $G$ is  meshed  if and only if for any metric triangle
  $vxy$, if $d(x,y) = 2$, then $d(v,x) = d(v,y) =2$ and there exists
  $z \sim x,y$ such that $d(v,z) = 2$.
\end{lemma}

For a set $A\subset V$ and a vertex $v\notin A$, let $d(v,A)=\min \{ d(v,x): x\in A\}$. The set $\Proj_v(A)$ of all
vertices $x\in A$ such that $d(v,x)=d(v,A)$ is called the \emph{metric projection} of $v$ on $A$. Notice that
for any graph $G$ and any subset $A$ of $G$ and $v\notin A$,  $\Proj_v(A)$  is included in the imprint $\Imp_v(A)$.
The following result shows that for  graphs with equilateral metric triangles (a subclass of graphs
satisfying (TC) and a superclass of meshed graphs), the projections and imprints on convex sets coincide:

\begin{lemma} \label{metric-projection} \cite[Lemma 4]{BaCh_helly} Let $G$ be a graph in which all metric triangles
are equilateral. Then, for every convex set $A$ and any vertex $v$ outside $A,$
$\Proj_v(A)=\Imp_v(A)$. 
\end{lemma}

\begin{proof} As we noted above, $\Proj_v(A)\subseteq \Imp_v(A)$. To prove the converse inclusion, pick any two vertices $x,y\in \Imp_{v}(A)$.
Let $v'$ be a vertex of $G$ such that $v'\in [v,x]\cap [v,y]$ and $[v',x]\cap [v',y]=\{ v'\}.$ By the choice of $x$
and $y$ from $\Imp_v(A)$ and since $A$ is convex, we conclude that  the three vertices $v',x,y$ form a metric triangle $v'xy$.
This triangle $v'xy$ is equilateral by hypothesis, whence $d(v,x)=d(v,y).$
\end{proof}

\begin{remark} \label{local-to-global-impossible} Analogously to weakly modular graphs, the triangle-square
complexes of meshed graphs are simply connected \cite{CCHO}. Moreover, it is shown in \cite[Theorem 9.1]{CCHO} that meshed graphs satisfy
the quadratic isoperimetric inequality, an important feature of non-positive curvature. However, there is a significant difference between
weakly modular and meshed graphs, noticed in \cite{CCHO}. Namely, in \cite[Theorem 3.1]{CCHO} it is shown that a graph $G$ is weakly modular
if and only if its triangle-square complex is simply connected and $G$ satisfies (TC) and (QC) locally (i.e., with $d(u,v)=d(u,w)\in \{ 2,3\}$ for (TC) and
$d(u,z)=3$ for (QC)). Furthermore, similarly to classical Cartan-Hadamard theorem, if $G$ satisfies the triangle and the quadrangle conditions locally, then the graph of the
universal cover of the triangle-square complex of $G$ is weakly modular. This local-to-global characterization (which is a  feature typical for nonpositive curvature) significantly generalizes previous characterizations
of median graphs and bridged graphs from \cite{Ch_CAT} and of weakly bridged graphs from \cite{ChOs}. In contrast to weakly modular graphs, there is no  local-to-global characterization for meshed graphs
\cite[Subsection 3.4]{CCHO}.
\end{remark}

%
%

\subsection{Local convexity implies convexity}

A connected induced subgraph $H$ of a graph $G$ is called \emph{locally-convex} if $[x,y]\subseteq V(H)$ whenever
$x,y\in V(H)$ and $d(x,y)=2$.

\begin{theorem} \label{lc->c} (Local convexity implies convexity) A connected
induced subgraph $H$ of a meshed graph $G$ is convex if and only if  $H$ is locally-convex.
\end{theorem}

\begin{proof} We will show that $[u,v]\subseteq V(H)$ for any $u,v\in V(H)$
by induction on the distance $k=d_H(u,v)$ in $H$ between $u$ and $v$,  the
case $d_H(u,v)=2$ being covered by local-convexity of $H$. Suppose by way of contradiction that
one can find  $u,v\in V(H)$ with $d_H(u,v)=k\ge 3$  and  a vertex
$w\in [u,v]\setminus V(H).$ Additionally assume that, if there are
several such pairs, then the selected one has the least distance $d_G(u,v)$
in $G.$ Let $P$ be a shortest $(u,v)$-path in $H$ and let $z$ be the neighbor of $u$ in $P$.
Then $d_H(z,v)=k-1$, thus $[z,v]\subseteq V(H)$ by the  induction hypothesis. In particular, this
implies that $d_H(z,v)=d_G(z,v)=k-1$.

If $d_G(u,v)=d_G(z,v),$ then by (TC), $u$ and $z$ have a common neighbor $v'$ in $[z,v]$  one step closer to $v$. Since $[v',v]\subset [z,v]\subseteq H$ and $d_G(v',v)=d_G(z,v)-1=k-2$,
the vertices $u$ and $v$ can be connected in $H$ by a path of length $k-1$,  contrary to the choice
of $P$ as a shortest $(u,v)$-path in $H$.  On the other hand, if $d_G(z,v)>d_G(u,v)$, then $w\in [u,v]\subset [z,v]\subset V(H)$,
which is impossible by the assumption that $w\in [u,v]\setminus V(H)$. Thus $d_G(u,v)=d_G(z,v)+1=d_H(u,v)=k.$

Let $t$ be a neighbor of $u$
on a shortest $(u,v)$-path of $G$ passing via $w$. If $t\sim z,$ by (TC) $t$ and $z$
have  a common neighbor $q\in [z,v]\subset V(H)$ one step closer to $v$. Since $u\nsim
q,$ the local-convexity of $H$ implies that $t\in [q,u]\subset H$. Furthermore, $d_H(t,v)=d_G(t,v)=k-1$, thus by the choice of the pair $u,v$
we conclude that $[t,v]\subset V(H)$, which is impossible since $w\in [t,v]\setminus V(H).$
Consequently, the vertices $t$ and $z$ are not adjacent. Since $d_G(z,v)=d_G(t,v)=k-1$ and $d_G(z,t)=2$, by (QC$^-$) there exists $s\sim t,z$
having distance $k-1$ or $k-2$ to $v$. If $d_G(s,v)=k-2=d_G(z,v)-1$, then $s\in [z,v]\subset V(H)$.
By local-convexity of $H$ we conclude that $t\in [s,u]\subset V(H)$. Since $d_H(t,v)=d_G(t,v)=d_G(u,v)-1=d_H(u,v)-1=k-1$
by the induction hypothesis, $[t,v]\subset V(H)$, which is impossible because $w\notin V(H)$. This shows that $d_G(s,v)=d_G(z,v)=d_G(t,v)=k-1$.
By (TC), there exists a vertex $p\sim z,s$ having distance $k-2$ from $v$. Since $p\in [z,v]$, the vertex $p$ belongs to $H$. If $p\sim t$, then
$t\in [u,p]\subset V(H)$ by local-convexity of $H$. Since $d_H(t,v)=d_G(t,v)=k-1$, we conclude that $[t,v]\subset V(H)$. Since $w\in [t,v]\setminus V(H)$,
we obtain a contradiction. Thus $t\nsim p$.  Applying (TC) once again to $s,t$ and $v$,  there exists a vertex $r\sim s,t$ having distance $k-2$ from $v$.

First, suppose that $u\sim s$.
Then $s\in [u,p]\subset V(H)$ by local-convexity of $H$. Since $s$ belongs to $[u,v]$ in $G$ and in $H$, by the minimality choice of the pair
$u,v$ we conclude that $[s,v]\subset V(H)$. Since $r\in [s,v]$, $r$ belongs to $H$.  Since $u\nsim r,p$, by local-convexity of $H$, $t\in [u,r]\subset V(H)$. Since
$d_H(t,v)=d_G(t,v)=k-1$, the pair $t,v$ satisfies the induction hypothesis. Hence $[t,v]\subset V(H)$, which is impossible because $w\in [t,v]$ and $w\notin V(H)$.
Consequently, $u\nsim s$.

By local-convexity of $H$ applied to the inclusions $t\in [u,s]$ and $s\in [t,z]$, we conclude that
$t$ belongs to $H$ if and only if $s$ belongs to $H$.  First suppose that the vertices $s$ and $t$ belong to $H$.
Then $d_H(s,v)=d_H(z,v)<k$, thus the pair $s,v$ satisfies the induction hypothesis, yielding $[s,v]\subset V(H)$.
This implies that $d_H(t,v)\le d_H(u,v)$. Since $d_G(t,v)=d_G(u,v)-1$, from the choice of the pair $u,v$ minimizing $d_H(u,v)$ and $d_G(u,v)$, we conclude that the pair $t,v$ satisfies the induction hypothesis and thus
$[t,v]\subset V(H)$.  Since $w\in [t,v]\setminus V(H)$, this is impossible.

Finally, suppose that $s$ and $t$ do not belong to $H$.
Since 
$t$ is not adjacent to $p$ and $s$ is adjacent to $t,p,z$, we have  $d_G(p,t)=d_G(p,u)=2$. By (TC),
there exists a vertex $s'$ adjacent to  $t,u,$ and $p$. By local-convexity of $H$, $s'\in [u,p]\subset V(H)$. Since $s'$ is adjacent to $p$ and $[p,v]\subset [z,v]\subset V(H)$,
$d_H(s',v)=k-1$. By induction hypothesis, $[s',v]\subset V(H)$ and $d_G(s',v)=k-1$.  Since $d_G(s',v)=d_G(t,v)$, by (TC) there exists a common neighbor $r'$ of $t$ and $s'$
belonging to $[s',v]\subset V(H)$ and one step closer to $v$. Since $d_G(r',v)=k-2$, $u$ is not adjacent to $r'$. Consequently, $t\in [u,r']$ and therefore $t$ belongs to $V(H)$ by the local convexity of $H$.
This contradicts our assumption that $t\notin V(H)$. This final contradiction shows that the locally-convex subgraph $H$ must be convex.
\end{proof}

In the important case of median graphs (graphs in which triplet of vertices has a unique median), the proof of the Theorem \ref{lc->c} is much shorter:

\begin{proposition} \label{lc->c-median}  A connected induced subgraph $H$ of a median graph $G$ is convex if and only if  $H$ is locally-convex.
Equivalently, $H$ is convex if and only if the intersection of $H$ with any square of $G$ is  empty, a vertex, an edge, or the square itself.
\end{proposition}

\begin{proof} We will show that $[u,v]\subseteq V(H)$ for any $u,v\in V(H)$
by induction on the distance $k=d_H(u,v)$ between $u$ and $v$ in $H$,  the
case $d_H(u,v)=2$ being covered by local-convexity of $H$. Suppose by way of contradiction that
one can find  $u,v\in V(H)$ with $d_H(u,v)=k\ge 3$  and  a vertex
$w\in [u,v]\setminus V(H).$  Let $P$ be a shortest $(u,v)$-path in $H$ and let $z$ be the neighbor of $u$ in $P$.
Then $d_H(z,v)=k-1$, thus $[z,v]\subseteq V(H)$ by the  induction hypothesis. In particular, this
implies that $d_H(z,v)=d_G(z,v)=k-1$ and thus $d_G(u,v)\le k$. Since $G$ is bipartite, either $d_G(u,v)=k$ or $d_G(u,v)=k-2$.
But if $d_G(u,v)=k-2$, then $w\in [u,v]\subset [z,v]\subset V(H)$,
which is impossible. Thus $d_G(u,v)=k.$ Let $t$ be a neighbor of $u$
on a shortest $(u,v)$-path of $G$ passing via the vertex $w$. By (QC), there exists $s$ adjacent to $t$ and $z$
at distance $k-2$ from $v$. Since $s\in [z,v]\subseteq H$ and $t\in [u,v]$, by local-convexity of $H$ we conclude that $t\in H$. Since $d_H(t,v)=d_H(s,v)+1=k-1$,
by induction hypothesis $[t,v]\subseteq H$. Since $w\in [t,v]$, this leads to a contradiction.
\end{proof}

\begin{remark}
Theorem \ref{lc->c} is a discrete analog of classical theorem in Euclidean convexity of Tietze \cite{Ti} and Nakajima \cite{Na} from 1928:
if a subset $A$ of ${\mathbb R}^n$  is closed, connected, and locally convex (i.e., each point $x\in A$ has a open neighborhood $N(x)$ such that $N(x)\cap A$ is convex), then $A$ is convex.
Theorem \ref{lc->c} is a generalization of an analogous result for weakly modular graphs established in \cite[Theorem 7(a)]{Ch_metric}.
The proof  for weakly modular graphs is simpler but runs along the same principles. Both results generalize Proposition \ref{lc->c-median}. That local-convexity implies convexity in basis graphs of matroids and of even $\Delta$-matroids
was established in \cite{Ch_delta}. A similar local-to-global characterization of convexity in ample partial cubes (generalizing median graphs) was established in \cite[Lemma 10]{CCMW}.
Finally, the paper \cite{StHoWr} consider the learning problem  of locally convex sets in metric spaces.
\end{remark}

\begin{remark}
Recently, Sakai and Sakuma \cite{SaSa} presented a proof of a local-to-global characterization of convex subcomplexes in CAT(0) cube complexes. They cast this result as well-known among experts but whose full
proof is difficult to find. 
Namely, they call a subcomplex $W$ of a CAT(0) cube complex $X$ \emph{combinatorially locally convex} if for any vertex $w$ of $W$ its link $\Lk(w,W)$ in $W$ is a full subcomplex of its link $\Lk(w,X)$ in $X$
(recall that a subcomplex $K$ of a simplicial complex $L$ is \emph{full} if any simplex of
$L$ whose vertices are in $K$ is in fact entirely contained in $K$).   Sakai and Sakuma \cite{SaSa} proved that a connected subcomplex $W$ of a  finite dimensional CAT(0) cube complex $X$ is convex (with respect to
its intrinsic $\ell_2$-metric)  if and only if $W$ is combinatorially locally convex in $X$. Proposition \ref{lc->c-median} paves an alternative proof of this result. Namely, the fact that $W$ is connected and combinatorially locally convex implies that
the 1-skeleton of $W$ is a connected locally-convex subgraph of the 1-skeleton $G(X)$ of $X$. By \cite[Theorem 6.1]{Ch_CAT}, $G(X)$ is a median graph, thus by Proposition \ref{lc->c-median} the 1-skeleton  $G(W)$ of $W$ is a convex subgraph of $G(X)$.
By a result of Sageev \cite[Theorem 4.13]{Sa}, each hyperplane of $X$ is convex (with respect to the geodesic $\ell_2$-metric). Since each convex subgraph of a median graph is an intersection of halfspaces (because median graphs are $S_3$-graphs, see Lemma \ref{copoints-bipartite}) and each halfspace is bounded by a convex hyperplane, the halfspaces are convex, and thus $W$ is a convex subcomplex of $X$ as the intersection of convex halfspaces. Finally, a similar in spirit local-to-global characterization of convex subgraphs of hypercellular graphs
(which generalize median graphs) was given in  \cite[Proposition 9]{ChKnMa}.
\end{remark}

\subsection{$\Delta$-closedness implies gatedness}
A connected induced subgraph $H$ of a graph $G$ is called \emph{$\Delta$-closed} if $z$ belongs to $H$ whenever $z$ is adjacent to two distinct vertices $x,y$ of $H$. Clearly,
each $\Delta$-closed subgraph is locally-convex. The following result generalizes a similar result \cite[Theorem 7(b)]{Ch_metric} for weakly modular graphs.

\begin{proposition} \label{deltacl->gated} ($\Delta$-closedness implies gatedness) A connected
induced subgraph $H$ of a meshed graph $G$ is gated if and only if  $H$ is $\Delta$-closed.
\end{proposition}

\begin{proof} Clearly, each gated subgraph is $\Delta$-closed. Conversely, let $H$ be a connected $\Delta$-closed subgraph of $G$. Then $H$ is locally-convex. By Theorem \ref{lc->c}, $H$ is a convex subgraph of $G$.
Let $u$ be an arbitrary vertex of $G$ not belonging to $H$. Let $x$ be a closest to $u$ vertex of $H$. Suppose by way of contradiction that $x$ is not the gate of $u$ in $H$. Then $H$ contains a vertex $v$ such that $x\notin [u,v]$.
Let $u'v'x'$ be a quasi-median of $u,v,w$. From the choice of $x$ we conclude that $x'=x$. Since $H$ is convex, $v'\in [x,v]\subseteq H$. By Lemma \ref{equilateral}, $u'v'x$ is an equilateral metric triangle. Since $x\notin [u,v]$, $u'v'x$ is a metric triangle of size $k>0$.  From the choice of $x$ we conclude that $u'$ does not belong to $H$. By Lemma \ref{equilateral_bis}, there exists a neighbor $y\in [x,v']\subseteq H$ of $x$ at distance $k$ from $u'$. By (TC), there exists a common neighbor $z$ of $x$ and $y$ at distance $k-1$ from $u'$. Since $d(u,z)\le d(u,u')+d(u,z)=d(u,x)-1$, the choice of $x$ implies that $z$ does not belong to $H$. Since $x,y\in H$, this contradicts the fact that $H$ is $\Delta$-closed.
\end{proof}


\subsection{Fiber-complemented meshed graphs}
A classical result about median graphs (which is now a kind of folklore) is a result of Isbell \cite{Is} that any finite median graph can be obtained
from hypercubes ((Cartesian products of edged)) by a sequence of gated amalgamations. Analogously,
any quasi-median graph can be obtained by gated amalgamations from Hamming graphs \cite{BaMuWi}.  A similar decomposition result was obtained in \cite{BaCh_wmg} for weakly median graphs.
Generalizing this approach, Chastand \cite{Cha1,Cha2} presented a general framework
for which this kind of decomposition theorems hold.
A gated subset $H$ of a graph $G$ gives rise to a partition $F_{a}$
$(a\in V(H))$ of the vertex-set of $G;$ viz., the {\em fiber} $F_{a}$ of
$a$ relative to $H$ consists of all vertices $x$ (including $a$
itself) having $a$ as their gate in $S$.  According to Chastand
\cite{Cha1,Cha2}, a graph $G$ is called {\it fiber-complemented} if
for any gated set $H$ all fibers $F_{a}$ $(a\in H)$ are gated sets of
$G$. A graph $G$ is said to be {\it elementary} if the only proper gated
subgraphs of $G$ are the singletons.  A graph $G$ with at least two vertices
is said to be \emph{prime} if it is neither a Cartesian product nor a
gated amalgamation of smaller graphs. The prime gated subgraphs of a graph
$G$ are called the {\it primes} of $G$.  We continue with the main result
of Chastand about fiber-complemented graphs:

\begin{theorem} \cite{Cha1,Cha2} \label{fiber-compemented-Chastand} A graph $G$ is a
fiber-complemented graph if and only if $G$ can be obtained from Cartesian products of elementary graphs by a sequence
of gated amalgamations. Any fiber-complemented graph $G$ embeds isometrically into the Cartesian
product of its elementary graphs. 
\end{theorem}

A graph $G$ is called {\it pre-median} \cite{Cha1,Cha2} ({\it pm-graph}, for short)
if $G$ is a weakly modular graph without induced $K_{2,3}$ and $W^-_4$ (the first two graphs from Fig.~\ref{non-S3}).
Here are the main properties of
pre-median graphs:

\begin{theorem} \label{pre-median Chastand} \cite{Cha1,Cha2} For a pre-median graph $G$, the following properties hold:
\begin{itemize}
\item[(i)] $G$ is elementary if and only if $G$ is prime;
\item[(ii)] $G$ is fiber-complemented;
\item[(iii)] $G$ is isometrically embeddable in a weak Cartesian product of its primes;
\item[(iv)] if $G$ is finite, then $G$ can be obtained by gated amalgamations from Cartesian products of
its prime subgraphs.
\end{itemize}
\end{theorem}

A {\it prime pre-median graph} (a {\it ppm-graph} for short) is a pre-median graph which is a prime graph.
The unique prime median graph is $K_2$, the prime quasi-median graphs are the $K_n, n\ge 2$, and the prime weakly median graphs are the 5-wheel $W_5$, the
octahedra $O_d$ and their 2-connected subgraphs, and 2-connected $K_4$-free plane bridged triangulations \cite{BaCh_wmg}. The (weakly) bridged graphs are precisely the
primes of bucolic graphs \cite{BrChChGoOs}, a subclass of weakly modular graphs which is a common generalisation of median graphs and bridged graphs.  Notice that the prime
pre-median graphs are \emph{irreducible graphs} sensu \cite{GrWi},  i.e., in any isometric embedding into the Cartesian product of graphs, they appear as
isometric subgraphs of a factor.  In our $S_3$-setting, since the class of $S_3$-graphs is closed by taking Cartesian products and gated amalgamation \cite{BaChvdV}, Theorem \ref{pre-median Chastand}
implies that in order to characterize weakly modular $S_3$-graphs, it is sufficient to characterize prime $S_3$-graphs (i.e., in case of bucolic $S_3$-graphs, it is sufficient to characterize
weakly bridged $S_3$-graphs).

Chastand \cite{Cha1,Cha2} asked which pre-median graphs are prime and the answer was provided by the following general topological
characterization:

\begin{theorem} \cite{CCHO} \label{prime-pre-median} A pre-median graph is prime if and only if its clique complex is simply connected.
\end{theorem}

Now, we extend the definition of pre-median graphs from weakly modular to meshed graphs. With some abuse of terminology, we say
that a meshed graph $G$ is \emph{pre-median} if $G$ does not contain $K_{2,3}$ and $W^-_4$ as induced subgraphs.  We have
the following generalization of Theorem \ref{pre-median Chastand}(i)$\&$(ii):

\begin{theorem} \label{meshed-pre-median} Any meshed pre-median graph $G$ is fiber-complemented. Consequently, $G$ can be obtained from Cartesian products of elementary graphs by a sequence
of gated amalgamations. Furthermore, $G$ is elementary if and only if $G$ is prime.
\end{theorem}

\begin{proof} Let $H$ be a gated subgraph of $G$. We will use the following simple property of fibers of $H$:

\begin{claim} \label{fibers-claim}  If $x\in F_a, y\in F_b$ and $x\sim y, a\ne b$, then $a\sim b$ and $d(a,x)=d(b,y)$.
\end{claim}

\begin{proof} Since $a\in [b,x]$ and $b\in [a,y]$, we have $d(a,y)=d(a,b)+d(b,y)$ and $d(b,x)=d(b,a)+d(a,x)$. Since $x\sim y$, $d(b,x)\le d(b,y)+1$ and $d(a,y)\le d(a,x)+1$. From these expressions we get $d(a,b)=1$ and then $d(a,x)=d(b,y)$.
\end{proof}

Let $F_a$ be the fiber of a vertex $a$ of $H$. First notice that if $x\in F_a$, then $[a,x]\subseteq F_a$, thus $F_a$ induces a connected subgraph. Therefore, by Proposition \ref{deltacl->gated} to prove that  $F_a$ is gated it suffices to establish
that $F_a$ is $\Delta$-closed. Suppose by way of contradiction that $F_a$ contains two vertices $x,y$ both adjacent to a vertex $z\in F_b$ with $b\ne a$. By Claim \ref{fibers-claim}, $a\sim b$ and $d(a,x)=d(a,y)=d(b,z)=:k$.
Suppose that the vertices $a,b,x,y,z$ violating the $\Delta$-closedness of $F_a$  are selected to minimize $k$.

If $x\nsim y$, by (QC$^-$) there exists a common neighbor $u$ of $x$ and $y$ at distance at most $k$ from $a$. If $x\sim y$, by ($\TC$) there exists a common neighbor $u$ of $x$ and $y$ at distance $k-1$ from $a$. In both cases,  if $d(u,a)=k-1$, then $u\in [x,a]\subseteq F_a$, thus $d(u,b)=k$. This implies  $z\nsim u$. Again by (QC$^-$), there exists a common neighbor
$v$ of $z$ and $u$ at distance at most $k$ from $b$. If $v\notin F_b$, say $v\in F_c$ for $c\in H$, then $d(v,c)\le k-1$, which contradicts Claim \ref{fibers-claim}. Consequently, $v\in F_b$. Since $v\sim u$, $u\in F_a$ and $d(u,a)=k-1$, from
Claim \ref{fibers-claim} we conclude that $d(b,v)=k-1$. Since $d(b,x)=d(b,y)=k+1$, $v\nsim x,y$. Consequently,  $x,y,z,u,v$ induce a forbidden $K_{2,3}$ if $x\nsim y$ and a forbidden $W_4^-$ if $x\sim y$.

Now suppose that $d(u,a)=k$ (this implies that $x\nsim y$). By (TC) there exists a common neighbor $s$ of $x$ and $u$ at distance $k-1$ to $a$. Then $s\in [x,u]\subseteq F_a$.  Claim \ref{fibers-claim} implies that $z\nsim s$.
Since $d(b,z)=d(b,s)=k$, by (QC$^-$)  there exists a common neighbor $t$ of $z$ and $s$ at distance at most $k$ from $b$. As before, if $t\in F_c$ with $c\ne b$, then $d(t,c)\le k-1$ and we obtain a contradiction
with Claim \ref{fibers-claim} because $t\sim z\in F_b$ and $d(b,z)=k$. If $t\in F_b$, since $t\sim s\in F_a$ and $d(a,s)=k-1$, we conclude that $d(b,t)=k-1$. Therefore $t\nsim y$ and since $d(a,t)=d(a,y)=k$,  applying (QC$^-$)  we will find a common
neighbor $p\sim t,y$ at distance at most $k$ from $a$. Since $p$ is adjacent to $y\in F_a$ and $d(a,y)=k$, applying the same argument as before, we conclude that $p\in F_a$ and that $d(a,p)=k-1$. Since $t\sim s,p$ and $t\in F_b, s,p\in F_a$ and $d(t,b)=d(s,a)=d(p,a)=k-1$, we obtain a contradiction with the minimality choice of the vertices $a,b,x,y,z$. This contradiction establishes that $F_a$ is $\Delta$-closed and thus gated. This establishes the first assertion.
The second assertion follows from Theorem \ref{fiber-compemented-Chastand}. The proof of the last assertion is the same as the proof of the amalgamation theorem of \cite{BaCh_wmg} or the proof of Theorem \ref{pre-median Chastand}(i).
\end{proof}

\begin{question} It will be interesting to establish the assertions (iii) and (iv) of Theorem \ref{pre-median Chastand} and the characterization of Theorem \ref{prime-pre-median} for meshed pre-median graphs.
\end{question}

\subsection{Helly, Radon, and Carath\'eodory numbers of meshed graphs: complexity}
In this subsection we show that computing Helly, Radon, and Carath\'eodory numbers  of geodesic convexity in meshed graphs  is NP-hard.  That the computation of Radon and Carath\'eodory numbers of  geodesic convexity of general
graphs is NP-hard was already known \cite{Anetal,CoDoSa,Douetal,Douetalbis}, however our reduction is  simple and applies to the three numbers.
A set $A$ of vertices of a graph $G$ is \emph{ $h$-independent} if $\bigcap_{a\in A} \conv(A\setminus \{ a\})=\varnothing$ and the \emph{Helly number}  $h(G)$ is the size of a largest $h$-independent set.
 A \emph{Radon partition}  of a set $A$ is the partition $A_1,A_2$ of $A$
such that $\conv(A_1)\cap \conv(A_2)\ne \varnothing$. A set $A$  is \emph{ $r$-independent} if it does not admit a Radon partition and the \emph{Radon number}  $r(G)$ is the size of a largest $r$-independent set.
Finally, a set $A$ is called  \emph{ $c$-independent} if $\conv(A)\setminus (\bigcup_{a\in A} \conv(A\setminus \{ a\})\ne\varnothing$ and the \emph{Carath\'eodory  number}  $c(G)$ is the size of a largest $c$-independent set.

By the classical Helly, Radon, and Carath\'eodory theorems for Euclidean convexity in $\mathbb{R}^d$, the Helly and Carath\'eodory numbers are equal to $d$ and the Radon number is  equal to $d+1$. For general
convexity spaces, it is known that the Helly number is at most the Radon number \cite{Le}. For other results about these numbers in convexity spaces, see the books \cite{So,VdV}. By \cite{Du_retracts}, for any convexity space $(X,\C)$
 one can construct a graph whose geodesic convexity has the same Helly, Radon, and Carath\'eodory numbers as $(X,\C)$. On the other hand, Duchet and Meyniel \cite{DuMe} proved that for any convexity
 in a graph $G$ in which the convex sets induce connected subgraphs (in particular, for geodesic convexity), the Helly and the Radon numbers are bounded by $\eta(G)$ and $2\eta(G)$,
 respectively, where $\eta(G)$ is the \emph{Hadwiger number} of $G$  and is
 the size of the largest complete graph which is a minor of $G$.
It was shown in \cite{BaCh_helly}  that the Helly number $h(G)$ of a weakly modular graph $G$ coincides with the \emph{clique number} $\omega(G)$ of $G$, which is the size of
the largest clique of $G$ (for any graph $G$, $h(G)\ge \omega(G)$ holds).   In \cite{BaPe_radon} was shown that the Radon numbers of Helly graphs
is $\omega(G)$ and in \cite{Ch_triangulated} it was shown that the Radon number of chordal graphs
is $\omega(G)$ except when $\omega(G)=3$, in which case $r(G)$ is 3 or 4. The Radon number of median graphs was expressed in \cite{vdV_matching} as the size of
the superextension, which has a complex combinatorial structure.

For a graph $H$ with at least one edge, let $G:=G(H)$ be obtained from $H$ by adding two new vertices $x',x''$, which are adjacent to all vertices of $H$. One can easily check that $G$ is a Helly graph, thus it is meshed.
Obviously, the size of a largest clique of $G$ is $\omega(H)+1\ge 3$.

\begin{proposition} \label{complexity}  $h(G)=r(G)=c(G)=\omega(G)=\omega(H)+1$.  Consequently, deciding if the Helly number, the Radon number, or the Carath\'eodory number of geodesic convexity of a Helly graph
(or of a weakly modular or meshed graph) is at most $k$ is NP-complete.
\end{proposition}

\begin{proof}  First notice that each convex set of $G$ is either a clique or coincides with the vertex-set $V(G)$  of $G$. This is because if a set $A$ contains two non-adjacent vertices $u,v$, then either $\{u,v \}=\{x',x'' \}$ or $u,v$ are two non-adjacent vertices of $H$ and their interval $[u,v]$ in $G$ contains both $x',x''$. Since the interval $[x',x'']$ in $G$ contains all vertices of $H$, we deduce that $\conv(u,v)=V(G)$.
Note also that each set $A$ defining a clique of $G$ is $h$-independent, $r$-independent, and $c$-independent.

Since $G$ is Helly, $G$ is weakly moduler, thus  we have  $h(G)=\omega(G)=\omega(H)+1$ by  \cite{BaCh_helly}.  For our special graph $G$, this can be checked directly by noticing that each $h$-independent  set  $A$ of size at least 3
is a clique of $G$. Indeed, if $A$ contains two non-adjacent vertices $u,v$, then $\conv(u,v)=V(G)$, thus any vertex $w\in A\setminus \{ u,v\}$  will belong to the intersection $\bigcap_{a\in A} \conv(A\setminus \{ a\})$.
Now, we show that any $r$-independent set $A$ of size at least 3 is also a clique. Suppose not and suppose that it contains two non-adjacent vertices $u,v$ and a third vertex $w$. Then $w\in V(G)=\conv(u,v)$, thus the sets
$\{ u,v\}$ and $A\setminus \{ u,v\}$ define a Radon partition.  Finally, we show that any $c$-independent set $A$ with at least three vertices is a clique of $G$. Again suppose that $A$ contains two non-adjacent vertices $u,v$
and the third vertex $w$. But then $\conv(A)=V(G)=\conv(u,v)=\conv(A\setminus  \{ w\})$, contrary to the
assumption that $\conv(A)\setminus (\bigcup_{a\in A} \conv(A\setminus \{ a\})\ne\varnothing$.
\end{proof}

We conclude with two open questions:

\begin{question} \label{helly-radon-bridged} Is it true that for any meshed graph $G$, $h(G)=\omega (G)$ holds? Is it true that the Radon numbers of bridged graphs and of basis graphs of matroids and even $\Delta$-matroids
are upper bounded by a linear function of their clique number?
\end{question}

There exists already a significant difference between the expressions for the Radon numbers for known classes of graphs (Helly and chordal graphs from the one hand and median graphs from the other hand), so we do not expect
that there exists a unifying result about the Radon number of all
weakly modular or  meshed graphs. On the other hand, it is not clear at all how to investigate the Carath\'eodory number in such classes of graphs since there exist chordal graphs with clique-number 3 and arbitrarily
large Carath\'eodory numbers \cite{Ch_thesis}.

\section{Meshed $S_3$-graphs}\label{s:mesheds3graphs}

While Theorem \ref{semispaces-trianglecodition} provides us with an efficient characterization of semispaces of $S_3$-graphs satisfying (TC) and, as we  believe, Theorem \ref{S3graphs-trianglecondition} is the best one can get as a
characterization of such graphs, 
this theorem does not involve a condition on a fixed number of vertices and does not lead to a polynomial time algorithm
for recognizing $S_3$-graphs satisfying (TC). The goal of this section is to prove a bounded compactness characterization of meshed $S_3$-graphs,
which can be used to recognize them in polynomial time:

\begin{theorem} \label{S3meshed} A meshed graph $G$ is $S_3$
if and only if $G$ does not contains the graphs from Figure~\ref{non-S3} as induced  subgraphs. Furthermore,
any meshed $S_3$-graph is pre-median, and thus is fiber-complemented.
\end{theorem}

\begin{figure}
  \begin{center}
    \begin{tikzpicture}[x=0.8cm,y=0.8cm]
\filldraw[black] (0,0.5) circle (2pt);
\filldraw[black] (-1,-1) circle (2pt);
\filldraw[black] (1,-1) circle (2pt);
\filldraw[black] (-1,-3) circle (2pt);
\filldraw[black] (1,-3) circle (2pt);
\draw (0,0.5) -- (1,-1);
\draw (0,0.5) -- (-1,-1);
\draw (-1,-1) -- (-1,-3);
\draw (1,-3) -- (1,-1);
\draw (1,-3) -- (-1,-1);
\draw (-1,-3) -- (1,-1);
\draw[color=red!80, very thick](-1,-1) circle (4pt);
\draw[color=red!80, very thick](1,-1) circle (4pt);

     \begin{scope}[xshift=3.5cm]
     \filldraw[black] (0,0.5) circle (2pt);
\filldraw[black] (-1,-1) circle (2pt);
\filldraw[black] (1,-1) circle (2pt);
\filldraw[black] (-1,-3) circle (2pt);
\filldraw[black] (1,-3) circle (2pt);
\draw (0,0.5) -- (1,-1);
\draw (0,0.5) -- (-1,-1);
\draw (-1,-1) -- (-1,-3);
\draw (1,-3) -- (1,-1);
\draw (1,-3) -- (-1,-1);
\draw (-1,-3) -- (1,-1);
\draw (-1,-3) -- (1,-3);
\draw[color=red!80, very thick](-1,-1) circle (4pt);
\draw[color=red!80, very thick](1,-1) circle (4pt);
     \end{scope}

     \begin{scope}[xshift=7cm]
     \filldraw[black] (0,0.5) circle (2pt);
\filldraw[black] (-1,-1) circle (2pt);
\filldraw[black] (1,-1) circle (2pt);
\filldraw[black] (-1,-3) circle (2pt);
\filldraw[black] (1,-3) circle (2pt);
\draw (0,0.5) -- (1,-1);
\draw (0,0.5) -- (-1,-1);
\draw (-1,-1) -- (-1,-3);
\draw (1,-3) -- (1,-1);
\draw (1,-3) -- (-1,-1);
\draw (-1,-3) -- (1,-1);
\draw (-1,-1) -- (1,-1);
\draw[color=red!80, very thick](-1,-1) circle (4pt);
\draw[color=red!80, very thick](1,-1) circle (4pt);
     \end{scope}

     \begin{scope}[xshift=10.5cm]
     \filldraw[black] (0,0.5) circle (2pt);
\filldraw[black] (-1,-1) circle (2pt);
\filldraw[black] (1,-1) circle (2pt);
\filldraw[black] (-1,-3) circle (2pt);
\filldraw[black] (1,-3) circle (2pt);
\draw (0,0.5) -- (1,-1);
\draw (0,0.5) -- (-1,-1);
\draw (-1,-1) -- (-1,-3);
\draw (1,-3) -- (1,-1);
\draw (1,-3) -- (-1,-1);
\draw (-1,-3) -- (1,-1);
\draw (-1,-1) -- (1,-1);
\draw (-1,-3) -- (1,-3);
\draw[color=red!80, very thick](1,-3) circle (4pt);
\filldraw[black] (0,-4.5) circle (2pt);
\draw (-1,-3) -- (0,-4.5);
\draw (0,-4.5) -- (1,-3);
\draw (-1,-1) -- (0,-4.5);
\draw [color=red!100, very thick] plot [smooth cycle] coordinates {(-1.3,-0.7) (-0.7,-0.7)  (0.3,-4.7)  (-0.3,-4.7) (-1.3,-2.7)};
     \end{scope}

     \begin{scope}[xshift=14cm]
     \filldraw[black] (0,0.5) circle (2pt);
\filldraw[black] (-1,-1) circle (2pt);
\filldraw[black] (1,-1) circle (2pt);
\filldraw[black] (-1,-3) circle (2pt);
\filldraw[black] (1,-3) circle (2pt);
\draw (0,0.5) -- (1,-1);
\draw (0,0.5) -- (-1,-1);
\draw (-1,-1) -- (-1,-3);
\draw (1,-3) -- (1,-1);
\draw (1,-3) -- (-1,-1);
\draw (-1,-3) -- (1,-1);
\draw (-1,-1) -- (1,-1);
\draw (-1,-3) -- (1,-3);
\filldraw[black] (0,-4.5) circle (2pt);
\draw[color=red!80, very thick](-1,-1) circle (4pt);
\draw (-1,-3) -- (0,-4.5);
\draw (0,-4.5) -- (1,-3);
\draw (-1,-1) -- (0,-4.5);
\draw (-1,-3) -- (0,0.5);
\draw [color=red!100, very thick] plot [smooth cycle] coordinates {(-1.2,-3.2) (-1.2,-2.65)  (0.7,-0.75)  (1.2,-0.75) (1.2,-3.2)};
     \end{scope}
\end{tikzpicture}
\end{center}
  \caption{Meshed non-$S_3$ graphs.}\label{non-S3}
  \label{fig:def-cR-qcR}
\end{figure}
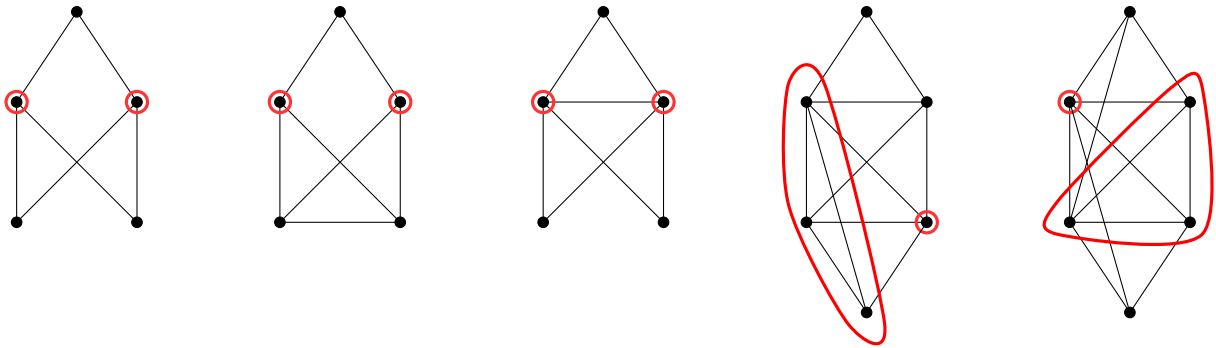

%

The main part of the proof is showing sufficiency. For this, in the subsequent subsections we present properties of meshed graphs not containing the graphs from Figure~\ref{non-S3} as induced  subgraphs.
In all subsequent results, we suppose that $G=(V,E)$ is a meshed graph not containing the graphs from Figure~\ref{non-S3} as induced  subgraphs.


\subsection{Positioning condition}

A graph $G=(V,E)$ satisfies the \emph{positioning condition (PC)} if for every vertex $b$ and every square $v_1v_2v_3v_4$
of $G,$ the equality $d(b,v_1)+d(b,v_3)=d(b,v_2)+d(b,v_4)$ holds \cite{Mau}. The positioning condition was used by Maurer \cite{Mau} to characterize the basis graphs of matroids and by the
author \cite{Ch_delta,Ch_johnson} to characterize the basis graphs of even $\Delta$-matroids and the isometric subgraphs of Johnson graphs.
It was also used in the local-to-global topological characterization
of basis graphs of matroids and even $\Delta$-matroids given in \cite{ChChOs_bgm}; namely, in this characterization, the positioning condition is replaced by a local positioning condition.
The positioning condition is equivalent to asserting that in the partition of $V$ into levels $L_i=\{ v\in V: d(b,v)=i\}$ with respect to $b$ (the \emph{leveling of $G$}), the vertices of
each square $v_1v_2v_3v_4$ lies in
one of the three positions: 1) all in one level; (2) in two levels, two adjacent vertices in each or (3) in three levels, two non-adjacent vertices in the middle level.

\begin{proposition} [Positioning condition] \label{poscond} 
$G$ satisfies the positioning condition (PC).
\end{proposition}

\begin{proof} Let $b$ be any vertex and $v_1v_2v_3v_4$ be any square of $G$. We proceed by induction on the distance sum $\sigma(v_1,v_2,v_3,v_4)=d(b,v_1)+d(b,v_2)+d(b,v_3)+d(b,v_4)$. Suppose by way of contradiction that
$d(b,v_1)+d(b,v_3)\ne d(b,v_2)+d(b,v_4)$. Then we have to consider the following three cases.

\begin{case-pc} $d(b,v_1)=d(v,v_3)=k+1$ and $d(b,v_2)=d(b,v_4)=k$.
\end{case-pc}

\begin{proof} By (QC$^-$) there exists a vertex $u\sim v_2,v_4$ at distance at most $k$ from $b$. If $d(b,u)=k-1$, then $u\nsim v_1,v_3$ and  $u,v_1,v_2,v_3,v_4$ induce the first forbidden subgraph. Thus $d(b,u)=k$.
If $u\nsim v_1$ or $u\nsim v_3$, say the first, then the square $v_1v_2uv_4$ also violates (PC). Since $d(b,v_1)+d(b,u)\ne d(b,v_2)+d(b,v_4)$, we get  $\sigma(v_1,v_2,u,v_4)<\sigma(v_1,v_2,v_3,v_4)$, contrary to
induction hypothesis. Thus $u\sim v_1,v_3$.
Since $d(b,v_2)=d(b,u)$, we can apply (TC) and derive a vertex $w\sim v_2,u$ at distance $k-1$ from $b$. But then $w\nsim v_1,v_3$ and  the vertices $w,u,v_1,v_2,v_3$ induce the third forbidden graph.
\end{proof}

\begin{case-pc} $d(b,v_1)=d(b,v_2)=d(b,v_4)=k$ and $d(b,v_3)=k+1$.
\end{case-pc}

\begin{proof}
By (TC) there exists  $u\sim v_1,v_2$ and having distance $k-1$ to $b$.  Then $u\nsim v_3$. If $u\sim v_4$, then $u,v_1,v_2,v_3,v_4$ induce the second forbidden graph. Thus $d(u,v_4)=d(u,v_3)=2$
and by (TC) there exists $w\sim u,v_3,v_4$. Since $d(b,v_3)=k+1$ and $d(b,u)=k-1$, we have $d(b,w)=k$. If $w\nsim v_1$, then the square $uv_1v_4w$ also violates (PC). Since $d(b,u)+d(b,v_4)=2k-1\ne 2k=d(b,v_1)+d(b,w)$, we obtain that
$\sigma(u,v_1,v_4,w)<\sigma(v_1,v_2,v_3,v_4)$, contrary to induction hypothesis. Thus $w\sim v_1$. If $w\nsim v_2$, then $v_1,w,v_2,v_3,v_4$ induce the second forbidden graph. Thus $w\sim v_2$. Since $d(b,v_4)=d(b,w)=k$ and $u\nsim v_4$, by (TC) there exists $t\sim w,v_4$ at distance $k-1$ from $b$ and different from $u$. Then $t\nsim v_3$.
If $t\nsim v_1$,  then $v_1,v_3,v_4,w,t$ induce the third forbidden graph. Thus $t\sim v_1$. Then $t\nsim v_2$, else $t$ can play the role of $u$ and we know that $u\nsim v_4$ and $t\sim v_4$. Then $t\sim u$, otherwise $v_1,w,u,v_2,t$ induce the third forbidden graph.  But then $v_1,u,t,w,v_2,v_4$ induce the fifth forbidden graph.
\end{proof}

\begin{case-pc} $d(b,v_1)=k$ and $d(b,v_2)=d(b,v_3)=d(b,v_4)=k+1$.
\end{case-pc}

\begin{proof}
By (TC) there exists $u\sim v_2,v_3$ at distance $k$ from $b$. If $u\sim v_1$, then to avoid that $u,v_1,v_2,v_3,v_4$ induce the second forbidden graph we must have $u\sim v_4$.
Analogously, if $u\sim v_4$, if $u\nsim v_1$ then the square $v_2v_1v_4u$ violates (PC) and $\sigma(v_2,v_1,v_4,u)<\sigma(v_1,v_2,v_3,v_4)$, contrary to the induction hypothesis.
Therefore the vertex $u$ is either adjacent to both $v_1$ and $v_4$ or to neither of these vertices.  First suppose that  $u\sim v_1,v_4$. Since $d(b,u)=d(b,v_1)=k$, by (TC) there exists
$t\sim u,v_1$ at distance $k-1$ from $b$. Then $t,u,v_1,v_2,v_4$ induce the third forbidden graph.

Now, suppose that $u\nsim v_1,v_4$. Since $d(u,v_1)=d(u,v_4)=2$, by (TC) there exists  $w\sim u,v_1,v_4$. Note that $d(b,w)\in \{ k,k+1\}$. If $w\nsim v_2$, then $d(b,u)+d(b,v_1)=2k$ and $d(b,w)+d(b,v_2)\ge 2k+1$. Since
$\sigma(u,v_2,v_1,w)<\sigma(v_1,v_2,v_3,v_4)$, the square $uv_2v_1w$ violates the induction hypothesis. Hence $w\sim v_2$. If $w\nsim v_3$, then  $u,w,v_2,v_3,v_4$ induce the second forbidden graph.
Thus $w\sim v_3$.

If $d(b,w)=k=d(b,v_1)$, by (TC) there exists a vertex $t\sim w,v_1$ at distance $k-1$ from $b$. Then the vertices $t,w,v_1,v_2,v_4$ induce the third forbidden graph.
Thus $d(b,w)=k+1$. By (QC$^-$) there exists a vertex $s\sim u,v_1$ at distance at most $k$ from $b$. Necessarily, $s\ne w,v_2$. If $d(b,s)=k-1$, then $s\nsim v_2,w$ and the vertices $s,u,v_1,v_2,w$ induce the second forbidden graph.
Thus $d(b,s)=k$. If $s$ is not adjacent to one of the vertices $v_2$ or $w$, say $s\nsim w$, then $uwv_1s$ is a square violating (PC) with $\sigma(u,w,v_1,s)<\sigma(v_1,v_2,v_3,v_4)$, a contradiction with induction hypothesis.
Hence $s\sim v_2,w$. Since $d(b,s)=d(b,v_1)=k$, by (TC) there exists a vertex $t\sim s,v_1$ at distance $k-1$ from $b$. If $t\nsim u$, then  $s,t,u,w,v_1,v_2$ induce the fourth forbidden graph. But if $t\sim u$, then $t,u,w,v_1,v_2$ induce the second forbidden graph. This contradiction concludes the analysis of the last case.
\end{proof}
This concludes the proof that $G$ satisfies (PC).
\end{proof}

\subsection{Convexity of intervals}

In this subsection, we prove the  following useful result.

\begin{proposition} [Intervals are convex] \label{meshed-interval} 
The intervals of $G$ are convex.
\end{proposition}

\begin{proof} Since the intervals induce connected subgraphs, by Theorem \ref{lc->c}
it suffices to prove that the intervals of $G$ are locally-convex.
Suppose by way of contradiction that there exist $u,v\in V,$ $x,y\in [u,v]$
with $d(x,y)=2,$ and  $z\in [x,y]\setminus [u,v].$ If $d(u,v)=2$, then  $u,v,x,y,z$ induce one of the
first two forbidden graphs (depending of whether $z$ is adjacent or not to one of  $u$ or $v$).
Thus $d(u,v)=k\ge 3.$  Further,  $[u,x]\cap [u,y]=\{ u\}$,
otherwise $u$ can be replaced  by a  closest to $x$ and $y$ vertex from the
intersection.  Analogously, $[v,x]\cap [v,y]=\{ v\}.$ We distinguish the following cases:

\begin{case-int} $u$ is adjacent to $x$ and $y$.
\end{case-int}

\begin{proof}
Then $d(x,v)=d(y,v)=k-1$ and $k-1\le d(z,v)\le k$. If $u\nsim z$, then   $u,x,z,v$ induce a square which violates (PC) with respect to $v$, contrary to Proposition \ref{poscond}.
Thus $u\sim z$. This implies that $d(z,v)=k$, otherwise  $z\in [u,v]$, contrary to our choice of $z$.   By (QC$^-$) there exists  $w\sim x,y$ at distance $k-1$ or $k-2$ from $v$. Since $x,y\in [u,v]\cap [z,v]$, necessarily
$w\ne u,z$.
If $d(w,v)=k-2$, then $w\nsim u,z$ and $x,y,z,u,w$ induce the second forbidden graph. Thus $d(w,v)=k-1$. If $w$ is not adjacent to $u$ or to $z$, then one of the quadruplets
$x,w,y,u$ or $x,w,y,z$ induce a square violating (PC). Thus $w\sim u,z$. Since $d(w,v)=d(x,v)=k-1$, by (TC) there exists  $t\sim x,w$ at distance $k-2$ from $v$. If $t\nsim y$, then
$t,x,y,z,u,w$ induce the fourth (if $t\nsim z$) or the fifth (if $t\sim z$) forbidden graph. If $t\sim y$, then  $t,x,y,z,u$ induce the second forbidden graph.
\end{proof}

The same proof shows that $v$ cannot be adjacent to both $x$ and $y$. So, further we can assume that $d(u,v)=k\ge 3$ and the vertices $x$ and $y$ are not both adjacent to $u$ or to $v$.

\begin{case-int} $d(u,x)=d(u,y)=k'\ge 2$.
\end{case-int}

\begin{proof}
Then $d(v,x)=d(v,y)=k-k'=:k''\ge 2$. Since $[u,x]\cap [u,y]=\{ u\}$, this also implies that $[x,u]\cap [x,y]=\{ x\}$ and $[y,u]\cap [y,x]=\{ y\}$. Indeed, if there exists $s\in [x,u]\cap [x,y]=\{ x\}$ different from $x$, then $s$ is a common neighbor of $x$ and $y$ at distance $k'-1$ from $u$, which is impossible because $[u,x]\cap [u,y]=\{ u\}$. Thus $uxy$ is a metric triangle. Since  metric triangles in meshed graphs are equilateral, we conclude that $k'=d(u,x)=d(u,y)=d(x,y)=2$. Analogously,
we conclude that $k''=2$ and $vxy$ is a metric triangle of size 2. Consequently, $d(u,v)=4$. Since $z\notin [u,v]$, either $d(u,z)=3$ or $d(v,z)=3$, say the first.

Since $x,y\in [z,u]$, by (QC$^-$) there exists  $w\sim x,y$ different from $z$ at distance at most $2$ from $u$. Since $[u,x]\cap [u,y]=\{ u\}$, necessarily $d(u,w)=2$. If $w\nsim z$, since $d(u,z)+d(u,w)=5$ and $d(u,x)+d(u,y)=4$,  $xzyw$ is a square violating (PC), contrary to Proposition \ref{poscond}. Thus $w\sim z$. Since $d(u,w)=d(u,x)=d(u,y)=2$, by (TC) there exist vertices $x'\sim u,x,w$ and $y'\sim u,y,w$. Since $[u,x]\cap [u,y]=\{ u\}$, $x'\ne y'$ and $x'\nsim y, y'\nsim x$. Analogously, if $d(w,v)=2$, then by (TC) there exist $x''\sim x,w,v$ and $y''\sim y,w,v$. Since $[v,x]\cap [v,y]=\{ v\}$, we conclude that $x''\ne y''$ and $x''\nsim y, y''\nsim x$.
Notice also that $z\ne x'',y''$. To avoid the third forbidden graph induced either by  $x',x,w,z,x''$ or by  $y',w,y,z,y''$,  $z$ must be adjacent to $x''$ and $y''$ (notice that $z\nsim x',y'$ because $d(u,z)=3$).
Then $x''$ and $y''$ must be adjacent, otherwise $z\in [x'',y'']$ and we are in the conditions of the first case. But then $x,x'',w,z,y,y''$ induce the fifth forbidden graph. This implies that $d(w,v)=3$ and thus $w\in [x,y]\setminus [u,v]$.
If $x'\nsim y'$, then $w\in [x',y']\setminus [u,v]$ and we are in the conditions of the first case. Thus $x'\sim y'$.

Since $d(w,v)=3$, we have $x,y\in [w,v]$ and $z\in [x,y]$. Consequently, either $z\notin [w,v]$ and we are in the conditions of the first case and we get a contradiction or $z\in [w,v]$ holds. Thus we can suppose that $d(z,v)=2$. By (TC), there exist vertices $x'''\sim x,z,v$ and $y'''\sim y,z,v$. Since $[v,x]\cap [v,y]=\{ v\}$, necessarily $x'''\ne y'''$ and $x'''\nsim y, y'''\nsim x$. If $x'''\nsim y'''$, then $x''',y'''\in [u,v]$ and $z\in [x''',y''']\setminus [u,v]$ and we are in the conditions of the first case. Thus $x'''\sim y'''$.  Since $d(x',v)=d(y',v)=3$ and $x'\sim y'$, by (TC) there exists  $s\sim x',y'$ at distance 2 from $v$. Since  $d(w,v)=3$ and $x\nsim y', y\nsim x'$, $s$ is
different from $x,y,w$. Then $s\sim w$, otherwise  $u,x',y',s,w$ induce the third forbidden graph. If $s\nsim x$, then $s,x\in [x',v]$ and $w\in [s,x]\setminus [x',v]$ because $[x',v]\subset [u,v]$.
Consequently, $[x',v]$ is a non-convex interval and $d(x',v)=3<d(u,v)$, contrary to the choice of the interval $[u,v]$. Thus $s\sim x$. Similarly we conclude that $s\sim y$. Then $u,x',y',s,w,x$ induce the fourth forbidden graph and $x,x',s,w,y',y$ induce the fifth forbidden graph.
\end{proof}

\begin{case-int}  $d(u,x)<d(u,y)$ and $d(v,x)>d(v,y).$
\end{case-int}

\begin{proof}
Then $d(u,y)=d(u,x)+1, d(v,x)=d(v,y)+1,$ otherwise $z\in [u,v].$ Consider a quasi-median $u'x'y'$ of the triplet $u,x,y$.
Since $[u,x]\cap [u,y]=\{ u\}$, necessarily $u'=u$. Since $ux'y'$ is an equilateral metric triangle and $d(u,y)=d(u,x)+1,d(x,y)=2$, from all this we conclude that
$ux'y'$ has size 1, i.e., $x'=x, y'\sim x,u,y,$ and $u\sim x$. Analogously, any quasi-median of the triplet $x,y,v$ is a metric triangle of size 1 of the form $vx''y$, where $x''$ is adjacent to $x$.
Consequently, $d(u,v)=3$. Since $z\notin [u,v]$ we have $d(z,u)=d(z,v)=2$.

First suppose that $x''\sim y'$. If $z$ is not adjacent to one of the vertices $x'',y'$, say $z\nsim y$, then we get a square $xzyy'$ violating (PC): $d(v,x)+d(v,y)=3$ and $d(v,z)+d(v,y')=4$.
Hence $z\sim x'',y'$. But then each of the sets $u,x,y',z,x'',y$ and $v,x'',y,z, x,y'$ induces the fourth forbidden graph. Thus $x''\nsim y'$. To avoid the first or the second forbidden graph
induced by $x,y,z,x'',y'$, $z$ must be adjacent to $x''$ and $y'$. Since $d(x,v)=d(y',v)=2$, by (TC) there exists  $s\sim x,y',v$. Since $d(z,v)=2$ and $x''\nsim y'$, $s$ is different from $x'',y,z$.
Then $s\sim z$, otherwise $u,x,y',z,s$ induce the third forbidden graph.
To avoid the second forbidden graph induced by $z,x'',s,y,v$, $s$ must be adjacent to  $x''$ or to $y$. If $x''\sim s$, then $u,x,y',z,x'',s$ induce the fourth forbidden graphs and if $s\sim y$, then
$u,x,y',z,s,y$ also induce the fourth forbidden graph.
\end{proof}

This  concludes the proof of the convexity of intervals $[u,v]$ of $G$.
\end{proof}

\subsection{Convexity of the shadows $y/x$}

We continue with the first result about convexity of shadows.

\begin{proposition} [Shadows $y/x$ are convex] \label{meshed-shadow} 
All shadows $y/x$ of $G$ are convex.
\end{proposition}

\begin{proof} First suppose that $x$ and $y$ are adjacent and we will prove that the shadow $y/x=\{ z\in V: y\in [z,x]\}$ is convex. Note that if $v\in y/x$, then $[y,v]\subseteq y/x$, thus  $y/x$ induces a connected subgraph of $G$. By Theorem \ref{lc->c}
it suffices to prove that $y/x$ is locally-convex. Let $z_1,z_2\in y/x, d(z_1,z_2)=2$, and suppose that there exists $z\in [z_1,z_2]$ not belonging to the shadow $y/x$. Among all triplets violating the local-convexity of $y/x$, suppose that the triplet $(z_1,z_2,z)$
lexicographically minimizes the vector $p(z_1,z_2,z)=(p_1,p_2,p_3)=(d(y,z_1)+d(y,z_2),d(y,z),d(x,z))$.

First, let $p_1=2$. If $z_1$ or $z_2$ coincides with $y$, say $z_1=y$, then $z_1=y\in [x,z_2]$ and since $z\in [z_1,z_2]$ we conclude that $y=z_1\in [y,z]$, yielding $z\in y/x$.
So, suppose that $z_1$ and $z_2$ are different from $y$. Since $p_1=2$, this implies that $y\sim z_1,z_2$. Then $d(x,z_1)=d(x,z_2)=2$. Since $z\notin y/x$, we get
$d(x,z)\le 2$. If $d(x,z)=1$, then $x,y,z_1,z_2,z$ induce the first forbidden graph if $z\nsim y$ and the second forbidden graph if $z\sim y$. If $d(x,z)=2$, then we can suppose that $d(y,z)=2$, otherwise $z\in y/x$. By (TC)
there exists $w$ adjacent to $x,y,z$. Then $w$ must be adjacent to both $z_1$ and $z_2$, otherwise  $y,z,w,z_1,z_2$ induce  either the first or the second forbidden graph. But then $x,w,y,z_1,z_2$ induce the
third forbidden graph. This concludes the analysis of the case $p_1=2$.

Now, suppose that $d(z_1,y)+d(z_2,y)=p_1>2$. Suppose without loss of generality that $d(x,z_2)\ge d(x,z_1)=:k$. Since $y\in [x,z_1]\cap [x,z_2]$, we have $d(y,z_2)\ge d(y,z_1)=k-1$ and $d(x,z_2)-d(x,z_1)=d(y,z_2)-d(y,z_1)$.
Therefore we can assume that $d(x,z_2)-d(x,z_1)\le 1$, otherwise $[z_1,z_2]\subset [y,z_2]$ and thus $z\in x/y$. Furthermore, we can suppose that $d(x,z)\le k$, otherwise, if $d(x,z)>d(x,z_1)$, then $z_1\in [z,x]$ and since $y\in [z_1,x]$,
we conclude that $y\in [x,z]$ and again $z\in y/x$. Furthermore, if $d(y,z)<d(y,z_1)$, then $z\in [z_1,y]$, and since $y\in [z_1,x]$, we would conclude again that $y\in [x,z]$ and $z\in y/x$. For the same reason, $d(y,z)\ge d(y,z_2)$.
Summarizing, we deduce  that  $d(z,y)\ge d(z_2,y)\ge d(z_1,y)=k-1$ and $d(x,z)\le d(x,z_1)=k\le d(x,z_2)$. We distinguish two cases.

\begin{case-x|y} $d(x,z_1)=d(x,z_2)=k$.
\end{case-x|y}

\begin{proof}
Then $d(y,z_1)=d(y,z_2)=k-1$. First suppose that $d(z,y)=k-1$. Then applying (TC), we will find  $z'\sim z_1,z$ at distance $k-2$ from $y$. Then $z'\in [y,z_1]\subseteq y/x$. Since $d(y,z')+d(y,z_2)<p_1$, $z'$ must be adjacent to
$z_2$, otherwise $z\in [z',z_2]$ and we obtain a contradiction with the minimality choice of the triplet $(z_1,z_2,z)$. Since $z'\in [z_1,y]$ and $y\in [z_1,x]$, we conclude that $y\in [x,z']$, i.e., $d(x,z')=k-1$. On the other hand,
we know that $d(x,z)\le k$. If $d(x,z)=k$, since $d(z,y)\le d(z',y)+1=k-1$, then we get  $y\in [x,z]$ and whence $z\in y/x$. Thus  $d(x,z)=k-1=d(x,z')$. By (TC) there exists  $w\sim z,z'$ at distance $k-2$ from $x$.
Since $d(x,z_1)=d(x,z_2)=k$, $w$ is not adjacent to $z_1,z_2$, thus $z_1,z_2,z,z',w$ induce the third forbidden graph. This contradiction shows that $d(y,z)=k$. Consequently, $d(x,z)=\{ k-1,k\}$.

Since $z_1,z_2\in [z,y]$, by (QC$^-$) there exists  $s\sim z_1,z_2$ different from $z$ at distance $k-1$ or $k-2$ from $y$. If $d(y,s)=k-2$, then $s\in [y,z_1]\subseteq x/y$, thus $d(x,s)=k-1$. Since $d(y,z)=k$,
$s$ and $z$ are not adjacent. But then the square $sz_1zz_2$ violates the positioning condition (PC) because $d(x,z_1)+d(x,z_2)=2k$ and $d(x,s)+d(x,z)\le k-1+k=2k-1$. Consequently,   $d(y,s)=k-1$.
Since $d(y,s)+d(y,z)=2k-1$ and $d(y,z_1)+d(y,z_2)=2k-2$, (PC) implies that $sz_1zz_2$ cannot be a square, i.e., $s\sim z$. Since $d(y,s)<d(y,z)$ and $s\in [z_1,z_2]$, we get that $p(z_1,z_2,s)<p(z_1,z_2,z)$.  From the
minimality choice of the triplet $(z_1,z_2,z)$ it follows that $s\in y/x$. This implies that $d(x,s)=d(y,s)+1=k$. Since $d(y,z_1)=d(y,s)=k-1$, by (TC) there exists $w\sim z_1,s$ at distance $k-2$ from $y$.
Since $w\in [y,z_1]$, we have $w\in y/x$, and thus $d(x,w)=k-1$.

If $d(x,z)=k=d(x,s)$, by (TC) there exists $t\sim s,z$ at distance $k-1$ from $x$. To avoid the third forbidden graph induced by $t,s,z,z_1,z_2$, the vertex $t$ must be adjacent to at least one of the vertices $z_1,z_2$, say $t\sim z_1$.
If $t$ is also adjacent to $z_2$, then $t\in [z_1,z_2]$ and $d(y,t)\le k=d(y,z)$ and $d(x,t)=k-1<d(x,z)$ and we conclude that $p(z_1,z_2,t)<p(z_1,z_2,z)$. By the minimality choice of  $(z_1,z_2,z)$
we conclude that $t\in y/x$, which would yield $d(y,t)=k-2$. But this is impossible because $z\sim t$ and $d(y,z)=k$. This proves that $t\nsim z_2$. Since $z\in [t,z_2]$ and $p(t,z_2,z)<p(z_1,z_2,z)$, from
the minimality choice of  $(z_1,z_2,z)$ we conclude that $t\notin y/x$, thus $d(y,t)\ge k-1$.
Furthermore, $w\ne t$ (since $d(y,t)=k-1$) and $w\nsim z,z_2$ ($w\nsim z$ since $d(y,z)=k$ and $d(y,w)=k-2$ and $w\nsim z_2$ because, by what has been shown above, $z_1$ and $z_2$
do not have a common neighbor at distance $k-2$ from $y$). If $w\sim t$, then  $s,t,w,z,z_1,z_2$ induce the fourth forbidden graph.
If $w\nsim t$, then $s\in [t,w]$. Since $d(x,s)=k=d(x,z_1)$ and $t,w\in [x,z_1]$, we obtain a contradiction with the convexity of the interval $[x,z_1]$ (Proposition \ref{meshed-interval}).
Finally, suppose that $d(x,z)=k-1=d(x,w)$.  Since $d(x,z_1)=k$, we have $z,w\in [z_1,x]$. Since $w\nsim z$, we have $s\in [w,z]$ and since $d(x,s)=k$, we obtain a
contradiction with the convexity of the interval $[x,z_1]$ (Proposition \ref{meshed-interval}). This completes the analysis of the case $d(x,z_1)=d(x,z_2)=k$.
\end{proof}

\begin{case-x|y} $d(x,z_1)=k$ and $d(x,z_2)=k+1$.
\end{case-x|y}

\begin{proof}
Then $d(y,z_1)=k-1$ and $d(y,z_2)=k$. Also recall that $d(y,z)\ge d(y,z_2)=k$ and $d(x,z)\le d(x,z_1)=k$. Since $z$ is adjacent to $z_1$ and $z_2$,  from the equalities $d(y,z_1)=k-1$ and $d(x,z_2)=k+1$ we deduce that $d(y,z)=d(y,z_2)=k$ and $d(x,z)=d(x,z_1)=k$.
By (TC) there exist vertices $t\sim z,z_1$ and $s\sim z,z_2$ such that $d(x,t)=k-1$ and $d(y,s)=k-1$. Note that $s\in [y,z_2]\subset y/x$, thus $d(x,s)=k$. Since $d(x,z_2)=k+1$, necessarily $t\nsim z_2$ and thus $t\ne s$. If $s\nsim z_1$, then
$z\in [z_1,s]$. Since $p(z_1,s,z)<p(z_1,z_2,z)$, we obtain a contradiction with the minimality choice of the triplet $(z_1,z_2,z)$. Hence $s\sim z_1$. Since $z\in [t,z_2]$ and $d(x,t)=k-1, d(x,z_2)=k+1$, the vertex $t$ cannot belong to  $y/x$.
We assert that $t$ can be chosen to be adjacent to $s$. Indeed, since $d(x,s)=d(x,z)=k$, by (TC) there exists $t'\sim s,z$ at distance $k-1$ from $x$. Since $d(x,z_2)=k+1$, $z_2\nsim t'$. To avoid the third forbidden graph
induced by $s,t',z,z_1,z_2$, the vertices $t'$ and $z_1$ must be adjacent. Thus $t'$ can play the role of $t$, i.e., further we can suppose that $s\sim t$.

Since $d(y,z_1)=d(y,s)=k-1$, by (TC) there exists  $w\sim z_1,s$ at distance $k-2$ from $y$. Since $w\in [y,z_1]\subset y/x$ and $t\notin x/y$,  $t$ and $w$ are different. If $t\nsim w$, then $t,w\in [x,z_1]$ and $s\in [t,w]$.
Since $d(x,s)=d(x,z_1)=k$, necessarily $s\notin [x,z_1]$, and we get a contradiction with the convexity of  $[x,z_1]$ (Proposition \ref{meshed-interval}). Consequently, $t\sim w$. But then $s,t,w,z,z_1,z_2$ induce the fourth forbidden graph.
This complete the analysis of the case $d(x,z_1)=k$ and $d(x,z_2)=k+1$ and the proof of the local-convexity of $y/x$. Thus $y/x$ is convex if $x$ and $y$ are adjacent.
\end{proof}

Now suppose that $x$ and $y$ are not adjacent and we proceed by induction on $d(x,y)$. The first part of the proof covered the case $d(x,y)=1$. Now suppose that $d(x,y)=k>1$. Let $x'$ be any neighbor of $x$ in  $[x,y]$.
Since $d(x',y)=k-1$, the shadow $y/x'$ is convex. By the first part of the proof, the shadow $x'/x$ is also convex. Now, pick any  $z_1,z_2\in y/x$ and any $z\in [z_1,z_2]$. Since $x'\in [x,y]\subset [x,z_1]\cap [x,z_2]$, we
conclude that $x'\in [x,z_1]\cap [x,z_2]$ and that $y\in [x',z_1]\cap [x',z_2]$. Consequently, $z_1,z_2\in x'/x$ and $z_1,z_2\in y/x'$. Since $x'/x$ and $y/x'$ are convex, we conclude that $z\in x'/x\cap y/x'$. Consequently, $y\in [x',z]$ and $x'\in [x,z]$, yielding $y\in [x,z]$.  Thus  $z\in y/x$.
\end{proof}

\subsection{Convexity of the shadows $K/x_0$} 

In this subsection, we prove the following result: 
%
%

\begin{proposition} [Shadows of cliques are convex] \label{meshed-shadow-clique1} 
For any pointed maximal clique $(x_0,K)$, the shadow $K/x_0$ is convex.
\end{proposition}

\begin{proof} Suppose by way of contradiction that the shadow $K/x_0$ is not convex. Since $K/x_0=\bigcup_{y\in K} y/x_0$, $K/x_0$ induces a connected subgraph of $G$. By Theorem \ref{lc->c} there exist
$z_1,z_2\in K/x_0$ with $d(z_1,z_2)=2$ and a common neighbor $z$ of $z_1,z_2$ such that $z\notin K/x_0$. Then the there exist  $u,v\in K$ such that $z_1\in x_0/u$ and $z_2\in x_0/v$. We can suppose that $u\ne v$, otherwise
$z_1,z_2\in u/x_0$ and we obtain a contradiction with the convexity of $u/x_0$ (Proposition \ref{meshed-shadow}). Furthermore, we can suppose that $z_1\notin v/x_0$ and $z_2\notin u/x_0$, yielding
$d(x_0,z_1)=d(v,z_1)=d(u,z_1)+1, d(x_0,z_2)=d(u,z_2)=d(v,z_2)+1$. This implies that $|d(u,z_1)-d(v,z_2)|\le 1$. Assume without loss of generality that  $k:=d(u,z_1)\le d(v,z_2)$. Suppose that among all choices of the counterexample,
we selected  the sextet $(x_0,u,v,z_1,z_2,z)$, which lexicographically minimizes the vector $p(x_0,u,v,z_1,z_2,z)=(p_1,p_2)$, where $p_1=d(u,z_1)+d(v,z_2)$ and $p_2=d(x_0,z)$. In the remaining part of the proof, first we consider two basic cases.
Then we show that in all remaining cases we always can find a violating sextet  with a lexicographically smaller vector.

\begin{case-x|k} $p_1=1$.
\end{case-x|k}

\begin{proof}
Since $d(u,z_1)\le d(v,z_2)$, we have  $z_1=u$ and $z_2$ is adjacent to $v$. Since $z\notin u/x_0$, necessarily $z\sim x_0$. Since $z_2\in v/x_0$, $x_0\nsim z_2$. To avoid the second forbidden subgraph induced by $x_0,u,v,z,z_2$, the vertices $z$ and $v$ must be adjacent.
Since $z\notin x_0/K$, the vertex $z$ does not belong to $K$. Since $(x_0,K)$ is a pointed maximal clique and $z\sim x_0,u,v$, there exists $y\in K$ not adjacent to $z$. If $y\nsim z_2$, then $x_0,u,v,y,z,z_2$ induce the fourth forbidden graph.
If $y\sim z_2$, then  $[u,z_2]$ is not convex because $y,z\in [u,z_2]$ and $x_0\in [y,z]\setminus [u,z_2]$, a contradiction with Proposition \ref{meshed-interval}.
\end{proof}

\begin{case-x|k} $p_1=2$ and $p_2=1$.
\end{case-x|k}

\begin{proof}
Since $|d(u,z_1)-d(v,z_2)|\le 1$, $p_1=2$ implies $d(u,z_1)=d(v,z_2)=1$. Consequently, $u\sim z_1, v\sim z_2, x_0\sim z$. Furthermore, $x_0\nsim z_1,z_2$ and $u\nsim z_2, v\nsim z_1$. If $u\nsim z$, the square $uz_1zx_0$ and  $z_2$ violate (PC): $z_2$ is adjacent to $z$ and has distance 2 to $x_0,u,z_1$. Thus $u\sim z$. Analogously, we conclude that $v\sim z$. Since $z\notin K/x_0$, $z$ cannot belong to the clique $K$. Since $(x_0,K)$ is a pointed maximal clique
containing $x_0,u,v$ and $z\sim x_0,u,v$, there exists $y\in K$ not adjacent to $z$. If $y$ is not adjacent to $z_1$ or $z_2$, say $y\nsim z_1$, then  $x_0,y,u,v,z,z_1$ induce the fourth forbidden graph. Thus $y\sim z_1,z_2$. Then
$z_1,z_2\in y/x_0$. Since $z\in [z_1,z_2]\setminus y/x_0$ we obtain a contradiction with the convexity of  $y/x_0$ (Proposition \ref{meshed-shadow}).
\end{proof}

Now, we will consider the remaining general cases. Then either $p_1>2$ or $p_1=2$ and $p_2\ge 2$. We group these cases in two cases: $d(v,z_2)=d(u,z_1)+1=k+1$ and $d(v,z_2)=k=d(u,z_1)$. 

%

\begin{case-x|k} $d(v,z_2)=d(u,z_1)+1=k+1$.
\end{case-x|k}

\begin{proof} Since $z_1\in u/x_0$ and $z_2\in v/x_0$,  we get $d(x_0,z_1)=k+1, d(x_0,z_2)=k+2$. Since $z_1\notin v/x_0$ and $z_2\notin u/x_0$, we also have $d(u,z_2)=k+2$ and $d(v,z_1)=k+1$. This implies that $d(u,z)=k+1$, and thus $z\in [u,z_2]$.
If $d(x_0,z)=k+2$, then $z\in u/x_0$, a contradiction with the choice of $z$. Thus $d(x_0,z)\le k+1$. Since $d(x_0,z_2)=k+2$ and $z\sim z_2$, we conclude that $d(x_0,z)\ge k+1$, thus $d(x_0,z)=k+1=d(u,z)$.

Since $d(z,x_0)=d(z,u)=k+1$, by (TC) there exists  $s\sim x_0,u$ at distance $k$ from $z$. Since $d(u,z_2)=k+2$, we have $s,v\in [u,z_2]$. Therefore, if $s\nsim v$, then $x_0\in [s,v]\setminus [u,z_2]$ (because $d(x_0,z_2)=k+2$) and
we obtain a contradiction with the convexity of $[u,z_2]$ (Proposition \ref{meshed-interval}). Consequently, $s\sim v$.
If $s\in K$, then $s\in [x_0,z]$, yielding $z\in s/x_0$, contrary with the assumption that $z\notin K/x_0$. Therefore $s\notin K$. Since $(x_0,K)$ is a pointed maximal clique of $G$, there exists $y\in K$ such that $y\nsim s$.
Since $d(x_0,z_2)=k+2$, we conclude that $s\in [x_0,z_2]$. Consequently, $d(z_2,v)=d(z_2,s)=k+1$ and by (TC) there exists  $t\sim s,v$ at distance $k$ from $z_2$. Since $d(u,z_2)=d(x_0,z_2)=k+2$, $t$ is not adjacent to $u$ and $x_0$.
If $t\nsim y$, then $u,v,x_0,y,s,t$ induce the fourth forbidden graph. On the other hand, if $y\sim t$, then  $u,x_0,y,s,t$ induce the second forbidden graph. This concludes the analysis of the case $d(v,z_2)=d(u,z_1)+1=k+1$.
\end{proof}

\begin{case-x|k} $d(u,z_1)=k=d(v,z_2)$.
\end{case-x|k}

\begin{proof}  First we prove that $d(u,z)=d(v,z)=k+1$.
Suppose by way of contradiction that $d(u,z)\le k$. Then $d(u,z)=k$, otherwise $z\in [u,z_1]\subset K/x_0$. By (TC) there exists $z'_1\sim z_1,z$ at
distance $k-1$ from $u$.  If $z'_1$ is not adjacent to $z_2$, then $z'_1\in [u,z_2]\subset u/x_0$ and $z\in [z'_1,z_2]$, thus $(x_0,u,v,z'_1,z_2,z)$ is a violating sextet with a smaller
distance sum $d(u,z'_1)+d(v,z_2)=2k-1$, a contradiction. Thus $z'_1\sim z_2$. But this is impossible because $d(u,z_2)=k+1$. Consequently, $d(u,z)=d(v,z)=k+1$. 
%
%
Since $x_0,u,v$ are pairwise adjacent and $z\notin u/x_0\cup v/x_0$, the equality $d(u,z)=d(v,z)=k+1$ implies that $d(x_0,z)=k$ or $d(x_0,z)=k+1$.

\begin{subcase-x|k} $d(x_0,z)=k$.
\end{subcase-x|k}

\begin{proof} Since $d(v,z_1)=k+1$ and $d(v,z_2)=k$, by (QC$^-$) there exists  $s\sim z_1,z_2$ at distance $k$ from $v$. Since $d(v,z)=k+1$, $s\ne z$. By (TC), there exists  $t\sim s,z_2$ at distance $k-1$ from $v$.
Since $d(v,z_1)=k+1=d(v,z)$, $t$ cannot be adjacent to $z_1$ and $z$.  Since $d(x_0,s)\le k+1$ and $d(x_0,z_1)=d(x_0,z_2)=k+1$, if $s\nsim z$, then the square $sz_1zz_2$ and the vertex $x_0$ violate (PC). Thus $s\sim z$. If
$d(x_0,s)=k+1$, then $z,t\in [x_0,s]$. Since $z_2\in [t,z]$ (because $t\nsim z$ and $z_2\sim t,z$)  and $d(x_0,z_2)=k+1$, we obtain a contradiction with the convexity of $[x_0,s]$ (Proposition \ref{meshed-interval}). Thus $d(x_0,s)=k=d(x_0,z)$.
By (TC) there exists $p\sim s,z$ at distance $k-1$ from $x_0$. Since $d(x_0,z_1)=d(x_0,z_2)=k+1$, $p\nsim z_1,z_2$ and  $p,s,z,z_1,z_2$ induce the third forbidden graph. This concludes the analysis of the subcase $d(x_0,z)=k$.
%
%
%
%
\end{proof}

\begin{subcase-x|k} $d(x_0,z)=k+1$.
\end{subcase-x|k}

\begin{proof} Since $d(z,u)=d(z,x_0)=k+1$, by (TC) there exists  $s\sim x_0,u$ at distance $k$ from $z$. First suppose that $s\sim v$. If $s\in K$, then $s\in [x_0,z]$, yielding $z\in s/x_0$, contrary to the assumption that $z\notin K/x_0$.
Therefore $s\notin K$. Since $(x_0,K)$ is a pointed maximal clique of $G$, there exists  $y\in K$ such that $y\nsim s$. If $d(s,z_1)=k+1$ and $d(s,z_2)=k+1$, then $z_1\in u/s$ and $z_2\in v/s$. Since $d(u,z)=d(v,z)=k+1$ and $d(s,z)=k$,
we also conclude that $z\in s/u\cap s/v$. This implies that for any pointed maximal clique $(s,K')$ containing the vertices $u,v,s$ we have $z_1,z_2\in K'/s$ and $z\notin K'/s$, thus $(s,u,v,z_1,z_2,z)$ is a violating sextet for $(s,K')$. Since
$d(s,z)=k<d(x_0,z)$ we obtain a contradiction with the minimality choice of the violating sextet $(x_0,u,v,z_1,z_2,z)$. Consequently, at least one of the distances $d(s,z_1)$ or $d(s,z_2)$ is equal to $k$, say $d(s,z_2)=k$. By (TC) there
exists $t\sim s,v$ at distance $k-1$ from $z_2$. If $t\nsim y$, then $x_0,y,u,v,s,t$ induce the fourth forbidden graph. On the other hand, if $y\sim t$, then $x_0,y,u,s,t$ induce the second forbidden graph.

Therefore  $s$ and $v$ are not adjacent. If $d(s,z_2)=k$, then $s,v\in [u,z_2]$ and $x_0\in [s,v]\setminus [u,z_2]$ (because $d(x_0,z_2)=k+1$), a contradiction with the convexity of $[u,z_2]$ (Proposition \ref{meshed-interval}).
Consequently, $d(z_2,s)=d(z_2,u)=k+1$ and by (TC) there exists  $p\sim u,s$ at distance $k$ from $z_2$. This implies that $z_2\in p/s$. If $d(s,z_1)=k+1$, then $z_1\in u/s$. Since $s\in [u,z]$ we also have $z\in s/u$. Consequently,
for any pointed maximal clique $(s,K'')$ of $G$ containing the vertices $u,s,p$ we will have $z_1,z_2\in K''/s$ and $z\notin  K''/s$. Thus $(s,u,p,z_1,z_2,z)$ is a violating sextet. Since $d(s,z)=k<d(x_0,z)$,
we obtain a contradiction with the minimality choice of the violating sextet $(x_0,u,v,z_1,z_2,z)$. Consequently,  $d(z_1,s)=k$. Since $d(z_1,u)=k$, by  (TC) there exists  $q\sim u,s$ at distance $k-1$ from $z_1$.
In order to avoid the third forbidden graph induced by $x_0,u,s,p,q$, there exists at least one edge between the vertices $x_0,p,q$. Since $d(x_0,z_1)=k+1$, $x_0$ and $q$ cannot be adjacent. If $x_0\sim p$, since $v,p\in [u,z_2]$ and
$x_0\in [v,p]\setminus [u,z_2]$, we obtain a contradiction with the convexity of  $[u,z_2]$ (Proposition \ref{meshed-interval}). Consequently, $x_0\nsim p,q$ and thus $p\sim q$. Let $(s,K''')$ be any pointed maximal clique of $G$ containing the pairwise adjacent vertices $u,s,p,q$.
Since $q\in [s,z_1], p\in [s,z_2]$, and $s\in [u,z]$, we conclude that $z_1,z_2\in K'''/s$ and $z\notin K'''/s$. Hence $(s,q,p,z_1,z_2,z)$ is a violating sextet for $(s,K''')$.
Since $d(q,z_1)+d(p,z_2)<d(u,z_1)+d(v,z_2)$ and $d(s,z)<d(x,z)$, we obtain a contradiction with the minimality choice of  $(x_0,u,v,z_1,z_2,z)$. This concludes the analysis of the subcase  $d(x_0,z)=k+1$.
\end{proof}
This finishes the proof of the case  $d(u,z_1)=d(v,z_2)=k$.
\end{proof}

This concludes the proof of the convexity of the shadow $K/x_0$. 
\end{proof}

\subsection{Convexity of the extended shadows  $x_0\SK K$}

Here we prove the following result: 
%
%

\begin{proposition}[Extended shadows are convex] \label{meshed-shadow-clique2} 
For any pointed maximal clique $(x_0,K)$ of $G$ the extended shadow  $x_0\SK K$ is convex.
\end{proposition}

\begin{proof} Let $K'=K\cup \{ x_0\}$. First notice that the extended shadow  $x_0\SK K$ induces a connected subgraph. Indeed, by definition, $x_0\SK K$ is the union of the sets $x_0|K=\bigcup_{y\in K}  x_0/y$ and $W_=(K')$.
Since $[z,x_0]\subseteq x_0/y$ for any $z\in x_0/y$, the union shadow  $K|x_0$  induces a connected subgraph. Now, pick any  $z\in W_=(K')$. We assert that the interval $[x_0,z]$ is included in $x_0\SK K$.
We proceed by induction on $k=d(x_0,z)$. Notice that $[x_0,z]\setminus \{ z\}$ is the union of the intervals $[x_0,z']$ taken over all neighbors $z'$ of $z$ in $[x_0,z]$. So, pick any neighbor $z'$ of $z$ in $[x_0,z]$. If $z'$ belongs to $\bigcup_{y\in K} x_0/y$,
say $z'\in x_0/y$ for $y\in K$, then, as we noticed above, $[x_0,z']\subseteq x_0/y\subset x_0\SK K$. If $z'\in W_=(K')$, then $d(x_0,z')=k-1$ and $[x_0,z']\subset x_0\SK K$ by induction hypothesis. Finally suppose that
there exists $z'\in [x_0,z]$ adjacent to $z$ and belonging to $y/x_0$ for some $y\in K$. Since $z'\in [x_0,z]$ and $y\in [x_0,z']\subset [x_0,z]$, we conclude that $z\in y/x_0$, contrary to our choice of $z$ from $W_=(K')$.
Consequently, $[x_0,z']\subset x_0\SK K$ for any neighbor $z'\in [x_0,z]$ of $z$, establishing that $[x_0,z]\subset x_0\SK K$. This proves that $x_0\SK K$ induces a connected subgraph of $G$.

Suppose by way of contradiction that the extended shadow  $x_0\SK K$ is not convex. Since  $x_0\SK K$ is connected, by Theorem \ref{lc->c} there exist
$z_1,z_2\in  x_0\SK K$ with $d(z_1,z_2)=2$ and a common neighbor $z$ of $z_1,z_2$ such that $z\notin  x_0\SK K$. Since $V\setminus (x_0\SK K)=K/x_0$, there exists  $y\in K$
such that $z\in y/x_0$. Notice that the vertices $z_1,z_2$ either both belong to $\bigcup_{s\in K} x_0/s$, or both belong to $W_=(K')$, or one belongs to $\bigcup_{s\in K} x_0/s$ and the second belongs to $W_=(K')$.
We suppose that the pointed maximal clique $(x_0,K)$ and the violating quadruplet $(x_0,z_1,z_2,z)$ are selected to minimize the distance sum $p(x_0,z_1,z_2,z)=k_1+k_2+k$, where $k_1=d(z_1,K'), k_2=d(z_2,K'),$ and $k=d(z,x_0)$ (recall that $d(a,K')$ is the minimum of all distances
from $a$ to a vertex of $K'$). Since $z\in [z_1,z_2]$, $d(z_1,K')$ and $d(z_2,K')$ cannot be simultaneously 0; thus $d(z_1,K')+d(z_2,K')\ge 1$. We distinguish three cases. In each of these cases we distinguish the base cases, which together
will cover the cases $p(x_0,z_1,z_2,z)=1$ and $p(x_0,z_1,z_2,z)=2$.

\begin{case-ext} $z_1,z_2\in W_=(K')$.
\end{case-ext}

\begin{proof} Then $z_1,z_2\notin K$. Furthermore, since $K'$ is a maximal clique of $G$, $d(z_1,K')>1$ and $d(z_2,K')>1$. Since $z$ is a common neighbor of $z_1,z_2$, this implies that $z\notin K'$. Consequently, we get $p(x_0,z_1,z_2,z)\ge 3$.
Recall that $d(z_1,K')=k_1$ and $d(z_2,K')=k_2$. Since $d(z_1,x_0)=d(z_1,y)=k_1$, by  (TC) there exists  $u\sim x_0,y$ at distance $k_1-1$ from $z_1$. Analogously, since $d(z_2,x_0)=d(z_2,y)=k_2$, by (TC) there exists
$v\sim x_0,y$ at distance $k_2-1$ from $z_2$. If $u=v$, then $z_1,z_2\in u/y$. Since $z\in y/x_0$ and $u\sim x_0,y$, we conclude that $z\notin u/y$, a contradiction with Proposition \ref{meshed-shadow}. Hence $u\ne v$. If $u\sim v$, then $z_1\in u/y$ and
$z_2\in v/y$, thus $z_1,z_2\in K^+/y$ for any pointed maximal clique $(y,K^+)$ containing  $u,v,x_0,y$. Since $z\in y/x_0$, again we conclude that $z\notin K^+/y$, a contradiction with Proposition \ref{meshed-shadow-clique1}. Hence $u\nsim v$.
Since $z_1,z_2\in W_=(K')$, the vertices $u$ and $v$ do not belong to $K$ (and $K'$). Therefore there exist vertices $s,t\in K$ such that $s\nsim u$ and $t\nsim v$. If $s=t$, then  $x_0,y,u,v,s$ induce the third forbidden graph. Hence $s\ne t$
and $u\sim t, v\sim s$. Then  $x_0,y,u,v,s,t$ induce the fifth forbidden graph.
\end{proof}

\begin{case-ext} $z_1\in x_0|K$ and $z_2\in W_=(K')$.
\end{case-ext}

\begin{proof} Then there exists $s_1\in K$ such that $z_1\in x_0/s_1$. Since  $d(z_2,x_0)=d(z_2,y)=k_2$, by (TC) there exists $v\sim x_0,y$ at distance $k_2-1$ from $z_2$. Since $z_2\in W_=(K')$, $v$ does not belong to $K'$. Therefore there exists $s\in K$, $s\nsim v$ (if $v\nsim s_1$, then  $s=s_1$). First suppose that $d(y,z_1)=k_1+1$, i.e., $x_0\in [y,z_1]$. Let $(y,K'')$ be any pointed maximal clique of $G$ containing the vertices $x_0,y,v$. Since $x_0\in [z_1,y]$ and $v\in [z_2,y]$, we conclude that
$z_1,z_2\in /x_0/y\cup v/y\subset K''/y$. On the other hand, since $y\in [x_0,z]$, we conclude that $z\in y/x_0$, i.e., $z\notin K''/y$. Since $z\in [z_1,z_2]$, this contradicts the convexity of the
shadow $K''/y$ (Proposition \ref{meshed-shadow-clique1}). This contradiction establishes that $d(y,z_1)\le k_1$.

Since $x,y,s_1\in K'$ and $z_1\in x_0/s_1$, we conclude that $d(y,z_1)\ge d(x_0,z_1)=k_1$, thus $d(y,z_1)=k_1$. By (TC) there exists  $u\sim x_0,y$ at distance $k_1-1$ from $z_1$.
Then $u$ must be adjacent to at last one of the vertices $v,s$, otherwise the vertices $u,x_0,y,s,v$ induce the third forbidden graph. First suppose that $u\sim v$. Let $(y,K''')$ be any pointed maximal clique of $G$
containing the pairwise adjacent vertices $x_0,y,u,v$. Then $u\in [y,z_1]$ and $v\in [y,z_2]$, yielding $z_1,z_2\in u/y\cup v/y\subset K'''/y$. On the other hand, $y\in [x_0,z]$, thus $z\in x_0/y$, i.e.,
$z\notin K'''/y$. Since $z\in [z_1,z_2]$, this contradicts the convexity of the shadow $K'''/y$ (Proposition \ref{meshed-shadow-clique1}). Consequently,  $u$ is not adjacent to $v$, thus $u\sim s$.
Since $x_0\in [s_1,z_1]$, this implies that $s\ne s_1$, and thus $v\sim s_1$ and $s\sim s_1$. Since $u\nsim v$ and $v\sim s_1, v\nsim s$, we conclude that
$u,v,x_0,y,s,s_1$ induce the fifth forbidden graph. This contradiction concludes the analysis of the case  $z_1\in \bigcup_{s\in K} x_0/s$ and $z_2\in W_=(K')$.
\end{proof}

\begin{case-ext} $z_1, z_2\in x_0|K$. 
\end{case-ext}

\begin{proof} Then there exist  $s_1,s_2\in K$ such that $z_1\in x_0/s_1$ and $z_2\in x_0/s_2$. If $s_1=s_2$, then we obtain a contradiction with the convexity of the shadow $x_0/s_1$ (Proposition \ref{meshed-shadow}) since $z_1,z_2\in x_0/s_1$ and $z\in [z_1,z_2]\setminus (x_0/s_1)$. Thus $s_1\ne s_2$ and we can suppose that
$z_1\notin s_2/x_0$ and $z_2\notin s_1/x_0$. This implies that $d(z_1,s_2)=k_1=d(z_1,s_1)-1=d(z_1,x_0)$ and $d(z_2,s_1)=k_2=d(z_2,s_2)-1=d(z_2,x_0)$. By (TC) there exist  $t_1\sim x_0,s_1$ at distance $k_2-1$ from $z_2$ and $t_2\sim x_0,s_2$ at distance $k_1-1$ from $z_1$.
Now, we compare the distance $k=d(z,x_0)$ to $k_1$ and $k_2$. First suppose that $k<\max\{ k_1,k_2\}$, say $k<k_1$. Then $y,t_2\in [x_0,z_1]$ and thus $d(z_1,y)=d(z_1,t_2)<d(z_1,x_0)$. But this is impossible because $s_1\sim y$ and we supposed that $x_0\in [s_1,z_1]$. Consequently, $k\ge \max\{ k_1,k_2\}$.

%

\begin{subcase-ext} $k=k_1=k_2$.
\end{subcase-ext}
\begin{proof} First suppose that one of the vertices $s_1$ and $s_2$, say $s_2$, can be chosen to coincide with $y$. This implies that $s_2\in [x_0,z]$, yielding $d(s_2,z)=k-1$. Since  $x_0\in [s_2,z_2]$, we also have $d(s_2,z_2)=1+d(x_0,z_2)=k+1$.
Since $z\sim z_2$ and  $d(s_2,z)=k-1$, this is impossible. Therefore $y$ is different from $s_1$ and $s_2$, yielding   $s_1,s_2\notin [x_0,z]$. Since $s_1,s_2\sim y$, we get $d(s_1,z)=d(s_2,z)=k$. From the definition of $t_1$ and $t_2$
also follows that they are distinct and are  different from $y$. 

First suppose that $y\sim t_1,t_2$. If $t_1\nsim t_2$, then  $s_1,t_1,x_0,y,s_2,t_2$ induce the fifth forbidden graph. So, let $t_1\sim t_2$. If $d(z,t_1)=k-1$ or $d(z,t_2)=k-1$, say $d(z,t_1)=k-1$, then by (TC) there exists
$u\sim y,t_1$ at distance $k-2$ from $z$. Since $d(s_1,z)=d(s_2,z)=k$, we have $u\nsim s_1,s_2$. If $u\nsim t_2$, then $u,y,s_1,t_1,t_2$ induce the third forbidden graph. If $u\sim t_2$, then  $u,x_0,y,t_1,t_2,s_2$ induce the fifth forbidden graph.
This implies that $d(z,t_1)=d(z,t_2)=k$. Consequently, $z_1\in [z,t_2]$. Since $d(z_2,s_2)=k+1,$ $s_2\sim t_2,y$ and $y,t_2\sim t_1$, and since $d(y,z)=d(t_1,z_2)=k-1$, we conclude that $t_2,z\in [s_2,z_2]$.
Since $d(s_2,z_1)+d(z_1,z_2)=k+2$, we conclude that $z_1\notin [s_2,z_2]$. Since $z_1\in [z,t_2]$ we obtain a contradiction with the convexity of  $[s_2,z_2]$ (Proposition \ref{meshed-interval}). This shows that the case when $y$ is adjacent to $t_1$ and $t_2$ is impossible.

Now suppose that $y$ is not adjacent to one of the vertices $t_1,t_2$, say $y\nsim t_1$. If $d(z,t_1)=k-1$, then $y,t_1\in [z,x_0]$, $s_1\in [y,t_1]$, but $s_1\notin [x_0,z_1]$ because $d(s_1,z)=k$; this contradicts the convexity of the interval $[x_0,z]$
(Proposition \ref{meshed-interval}). Thus $d(z,t_1)=k$. If $d(z_1,t_1)=k$, then $t_1,z\in [z_1,s_1]$. Since the interval $[z_1,s_1]$ is convex and $z_2\in [z,t_1]$, we must have $z_2\in [z_1,s_1]$. But this is impossible because $d(z_1,s_1)=k+1$ while $d(s_1,z_2)=k$ and $d(z_2,z_1)=2$.
Hence $d(z_1,t_1)=k+1$. Consequently, $x_0,z\in [z_1,t_1]$. Since $y\in [z,x]$, the convexity of the interval $[z_1,t_1]$ implies that $y\in [z_1,t_1]$. Since $y\nsim t_1$, this implies that $d(z_1,y)=k-1$. Since $y\sim s_1$ and $d(s_1,z_1)=k+1$,
this is impossible. This concludes the analysis of the case when $k=k_1=k_2$.
\end{proof}

\begin{subcase-ext} $k>\max\{ k_1,k_2\}$.
\end{subcase-ext}
\begin{proof}
Since $z\sim z_1,z_2$ and $d(x_0,z_1)=k_1$ and $d(x_0,z_2)=k_2$, this implies that $k_1=k_2=k-1$. Then $t_1,t_2,y\in [x_0,z]$. First suppose that $y$ is adjacent to one of the vertices $t_1,t_2$, say $t_2\sim y$.
By (TC) there exists $v\sim t_2,y$ at distance $k-2$ from $z$. Since $d(x_0,z)=k$,  $v\nsim x_0$.
If $v$ is not adjacent to both vertices $s_1,s_2$, then $v,t_2,y,x_0,s_1,s_2$ induce the fourth forbidden graph. If $v\sim s_2$ and $v\nsim s_1$, then  $v,t_2,y,x_0,s_1,s_2$ induce the fifth forbidden graph.
Thus we can suppose that $v\sim s_1$. If $v\nsim s_2$, then $t_2,s_1\in [s_2,v]$ and $x_0\in [s_1,t_2]\setminus [s_2,v]$, a contradiction with the convexity of  $[s_2,v]$ (in fact $t_2,v,x_0,s_1,s_2$
induce the second forbidden graph). Hence, $v$ is adjacent to both  $s_1,s_2$, yielding $s_1,s_2\in [x_0,z]$. Consequently, both  $s_1,s_2$ may play the role of $y$. But if $s_1$ plays the role of $y$, then
$s_1$ must be adjacent to $t_2$ (because we supposed that $y$ is adjacent to $t_2$). But this is impossible because $d(z_1,s_1)=k_1+1$ and $d(z_1,t_2)=k_1-1$. Therefore $y$ is not adjacent to  $t_1$ and $t_2$.

Then $s_1\in [y,t_1]$ and $s_2\in [y,t_2]$. Since $y,t_1,t_2\in [x_0,z]$, the convexity of  $[x_0,z]$ implies that $s_1,s_2\in [x_0,z]$, whence $d(s_1,z)=d(s_2,z)=k-1$. 
Thus both vertices $s_1$ and $s_2$ may play the role of $y$ and we can remove $y$ from our further analysis. Summarizing, we have the following equalities:
$d(z_1,t_2)=k-2=d(z_2,t_1), d(s_1,z_1)=k=d(s_2,z_2),$ and $d(z,s_1)=d(s,z_2)=d(x_0,z_1)=d(x_0,z_2)=k-1$.  By (TC) there exists $w\sim s_1,s_2$ at distance $k-2$ from $z$. Necessarily, $w\ne x_0,y$. 

If $d(z_1,t_1)=k$, then $z,x_0\in [z_1,t_1]$. Since $w\in [z,s_1]\cap [z,s_2]\subset [z,x_0]$, the convexity of  $[z_1,t_1]$ implies that $w\in [z_1,t_1]$. If $t_1\nsim w$ this implies that $d(z_1,w)=k-1$.
But this is impossible since $d(s_1,z_1)=k$ and $w\sim s_1$.
Thus $w\sim t_1$ and $d(z_1,w)=k-1$. Since $s_2\nsim t_1$ and $w\nsim x_0$, the vertices $w,t_1,x_0,s_2$ define a square of $G$. Since $d(z_1,w)=k-1=d(z_1,x_0)=d(z_1,s_2)$ and  $d(z_1,t_1)=k,$, we obtain a contradiction with
(PC) applied to the square $wt_1x_0s_2$ and the vertex $z_1$.

Finally, if $d(z_1,t_1)=k-1$, then $t_1,z\in [s_1,z_1]$. If $z_2\in [z,t_1]$, the convexity of the interval $[z_1,s_1]$ implies that $z_2\in [z_1,s_1]$. But this is impossible  because $d(s_1,z_2)=k-1$ and $d(z_2,z_1)=2$.
Hence $z_2\notin [z,t_1]$, i.e.,  $d(z,t_1)=k-2$. But this is impossible because $d(x_0,z)=k$ and $x_0\sim t_1$. This contradiction concludes the analysis of the subcase $k>\max\{ k_1,k_2\}$.
\end{proof}

This concludes the proof of the case $z_1, z_2\in x_0|K=\bigcup_{s\in K} x_0/s$.
\end{proof}

Consequently, for any pointed maximal clique $(x_0,K)$,
the extended shadow  $x_0\SK K$ is convex.
\end{proof}

\subsection{Proof of Theorem \ref{S3meshed}} First notice that the five forbidden graphs have diameter 2, thus they occur in $G$ as induced subgraphs if and only if they occur as isometric subgraphs. If a meshed graph
$G$ contains one of the graphs of Figure~\ref{non-S3} as an isometric subgraph, then one can directly check that the encircled vertex and the encircled convex set (of size 1 in the first three graphs and of size 3 in the
last two forbidden graphs) cannot be separated by complementary halfspaces in the forbidden graph, and thus in the whole graph $G$ as well. Conversely, let $G$ be a meshed graph not containing the graphs  from
Figure~\ref{non-S3} as induced subgraphs.
Since meshed graphs satisfy the triangle condition, by Theorem \ref{S3graphs-trianglecondition} $G$  is an $S_3$-graph if for
any pointed maximal clique $(x_0,K)$ of $G$,  the shadow $K/x_0$ and the extended shadow  $x_0\SK K$ are convex. This is proved in Propositions \ref{meshed-shadow-clique1} and \ref{meshed-shadow-clique2}.
The second assertion of Theorem \ref{S3meshed} follows from Theorem \ref{meshed-pre-median}.

\section{Examples of $S_3$-graphs}\label{s:examples}
The classes of $S_3$-graphs, meshed graphs, weakly modular graphs, and graphs satisfying (TC) are closed by taking Cartesian products and  gated amalgamations (for $S_3$ property, see \cite{BaChvdV} and \cite{VdV}).
Meshed $S_3$-graphs (and, more generally, meshed graphs with convex clique-shadows $K/x_0$) do not contain the first two graphs from Figure~\ref{non-S3}, and therefore they are pre-median graphs. By Theorem
\ref{meshed-pre-median} any such graph $G$  is fiber-complemented and thus $G$ can be obtained from Cartesian products of prime graphs by a sequence
of gated amalgamations. Therefore, in order to establish the global structure of meshed $S_3$-graphs it suffices to characterize their primes. 
In general, this is an interesting and non-trivial
problem. For example, $K_2$ is the only prime median graph \cite{Is}, which is equivalent to saying that any finite median graph can be obtained by gated amalgamations from cubes (Cartesian products of edges).
The prime quasi-median graphs are the complete graphs $K_n, n\ge 2$.
The weakly modular $S_4$-graphs have been characterized in \cite{Ch_local,Ch_separa} using 4 forbidden subgraphs. They are exactly the weakly median graphs and it was shown in \cite{BaCh_wmg} that the prime
weakly median  graphs are
the hyperoctahedra and their subgraphs, the 5-wheel $W_5$, and the bridged plane triangulations. 
We can formulate a similar question for meshed $S_3$-graphs:

\begin{question} \label{prime-meshed} Characterize prime meshed $S_3$-graphs.
\end{question}

In this section we provide several examples of classes of prime $S_3$-meshed graphs, giving a partial answer to Question \ref{prime-meshed}. We also present
several classes of $S_3$-graphs and $S_3$-graphs satisfying (TC), some of them generalizing partial cubes. Most of those classes consist of graphs that can be  isometrically  embedded in Johnson graphs,
Hamming graphs, or half-cubes.



\subsection{Partial Johnson graphs satisfying (TC)}
Recall that a graph $G=(V,E)$ is called a partial cube if $G$ can be isometrically embedded into a hypercube. Partial cubes have been nicely characterized by Djokovi\'{c} \cite{Dj} as the
graphs $G$ which are bipartite and in which for any edge $uv$ the sets $W(u,v)=\{ x\in V: d(x,u)<d(x,v)\}=u/v$ and $W(v,u)=\{ x\in V: d(x,v)<d(x,u)\}=v/u$ are convex (and thus are complementary halfspaces).
A \emph{partial Johnson graph} (respectively, a \emph{partial Hamming graph} or a \emph{partial half-cube}) is a graph which can be isometrically embedded in a Johnson graph (into a Hamming graph or into a half-cube, respectively).
Each partial cube is a partial Hamming graph satisfying (TC), each partial Hamming graph is a partial Johnson graph, and each partial Johnson graph is a partial half-cube. Partial Johnson graphs have been
characterized in Djokovi\'{c}'s style by the author of this paper
in \cite{Ch_johnson} and partial Hamming graphs have been characterized in \cite{Ch_Hamming,Wil} in a similar way (the question of characterizing partial half-cubes was raised in
\cite{DeLa} and is open). Namely, it was shown in \cite{Ch_johnson} that a graph $G$ is a partial Johnson graph if and only if it satisfies a link condition (which we will not specify here) and  for any edge $uv$ of $G$,
the set $W_=(u,v)=\{ x\in V: d(x,u)=d(x,v)\}$ induces at most two connected components
$W'_=(u,v)$ and $W''_=(u,v)$ of $G$ (which are allowed to be empty) and each of the pairs $\{ W(u,v)\cup W'_=(u,v),W(v,u)\cup W''_=(u,v)\}$ and  $\{ W(u,v)\cup W''_=(u,v),W(v,u)\cup W'_=(u,v)\}$ consists of
complementary halfspaces of $G$. We use the second condition to establish the following result:

\begin{proposition} \label{partial_Johnson} If $G=(V,E)$ is a partial Johnson graph satisfying the triangle condition (TC), then $G$ is an $S_3$-graph. Any 2-connected meshed partial Johnson graph $G$ whose triangle complex is simply connected
is elementary/prime.
\end{proposition}

\begin{proof} By Theorem \ref{S3graphs-trianglecondition} we have to prove that for any pointed maximal clique $(x_0,K)$ the shadow $K/x_0$ and the extended shadow $x_0//K=x_0|K\cup W_=(K\cup \{ x_0\})$ are convex.
Pick any vertex $y\in K$ and consider the edge $x_0y$. Let $K'=K\setminus \{ y\}$. Note that all vertices of $K'$ belongs to the same connected component of $W_=(x_0,y)$, say $K'\subseteq W'_=(x_0,y)$.
We assert that $K/x_0$ coincides with $W(y,x_0)\cup W'_=(x_0,y)$. First, we establish the inclusion $K/x_0\subseteq W(y,x_0)\cup W'_=(x_0,y)$. Pick any $u\in K/x_0=\bigcup_{z\in K} z/x_0$. If $u\in y/x_0$, then $u\in W(y,x_0)$ and we are done.
Now suppose that $u\in (z/x_0)\setminus (y/x_0)$ for $z\in K'$.  Since $x_0,y,z$ are pairwise adjacent and $u\notin y/x_0$, we conclude that $d(u,z)<d(u,x_0)=d(u,y)$. This implies that $u$ belongs to $W_=(x_0,y)$. Since
$u\in (z/x_0)\cap (z/y)$, $u$ belongs to the same connected component of $W_=(x_0,y)$ as $z$, i.e., to $W'_=(x_0,y)$. This shows that $K/x_0\subseteq W(y,x_0)\cup W'_=(x_0,y)$.

Now we prove the converse inclusion $W(y,x_0)\cup W'_=(x_0,y)\subseteq K/x_0$. Since $W(y,x_0)=y/x_0$, we have the inclusion  $W(y,x_0)\subseteq K/x_0$. Now, pick any $u\in  W'_=(x_0,y)$. 
Then $d(u,x_0)=d(u,y)=k$. By (TC), there exists a common neighbor $z$ of $x_0$ and $y$ at distance $k-1$ from $u$. If $z\in K$, then $u\in z/x_0\subseteq K/x_0$ and we are done. Therefore, $z\notin K$.
Since $u,y\in W(y,x_0)\cup W'_=(x_0,y)$, $z\in [u,y]$,  and $W(y,x_0)\cup W'_=(x_0,y)$ is convex, we conclude that $z\in W(y,x_0)\cup W'_=(x_0,y)$.
Since $K\cup \{ x_0\}$ is a maximal clique of $G$ and $z$ is adjacent to $x_0$ and $y$, there exists a vertex $z'\in K$ not adjacent to $z$. But then $z'\in K/x_0\subseteq W(y,x_0)\cup W'_=(x_0,y)$.  Since $x_0\in [z,z']$ and  $x_0\in W(x_0,y)$,
we obtain a contradiction with the convexity of the set $W(y,x_0)\cup W'_=(x_0,y)$. Consequently, $K/x_0=W(y,x_0)\cup W'_=(x_0,y)$ and therefore the shadow $K/x_0$ is convex. By Lemma \ref{extended}, the extended
shadow $x_0//K=(x_0|K)\cup W_=(K\cup \{ x_0\})$ is the complement of $K/x_0$.
Since $K/x_0=W(y,x_0)\cup W'_=(x_0,y)$ and by the result of \cite{Ch_johnson} mentioned above, the complement of $W(y,x_0)\cup W'_=(x_0,y)$ is the halfspace $W(x_0,y)\cup W''_=(x_0,y)$), we conclude that $x_0//K$ is convex.

Now, suppose that $G$ is any 2-connected meshed partial Johnson graph  whose triangle complex is simply connected. Then by Proposition 2.18 of  \cite{CCHO} it follows that any proper gated set of $G$ is a single vertex, thus $G$
is elementary, hence $G$ is prime.  This finishes the proof.
\end{proof}

We believe that Proposition \ref{partial_Johnson} holds for all partial Johnson graphs:

\begin{question} \label{allpartialJohnson}  Show that all partial Johnson graphs are $S_3$.
\end{question}

\begin{remark}
The Petersen graph $P_{10}$ is a partial Johnson graph \cite{Ch_johnson} and is an $S_3$-graph (Example \ref{petersen}), which does not satisfies (TC).
On the other hand, the icosahedron is a meshed $S_3$-graph (Example \ref{petersen}), which is neither a partial Johnson graph nor a partial half-cube.
\end{remark}

Basis graphs of matroids are the most important examples of partial Johnson graphs. A \emph{matroid} $M$ on a finite set $E$ is a collection $\mathcal B$ of subsets of $E$,
called \emph{bases}, satisfying the following exchange property: for all $A,B\in {\mathcal B}$  and $e\in A\setminus B$,  there
exists $f\in B\setminus A$  such that $A\setminus \{ e\} \cup \{ f\}\in {\mathcal B}$. All the bases of a matroid $M=(E, {\mathcal B})$ have the same
cardinality. The \emph{basis graph of a matroid} $M$ is the graph $G(M)$ whose vertices are the bases of $\mathcal B$
and edges are the pairs $A, B$ of bases such that $|A\Delta B| = 2$. From the exchange axiom and since all bases have the same cardinality,
it follows that basis graphs of matroids are partial Johnson graphs. The bases of the \emph{graphic matroid} $M(H)$ of a graph $H=(V,E)$ has $E$ as the ground set and
the spanning trees $T$  of $H$ as bases;  two spanning trees $T,T'$ are  adjacent in the basis graph $M(H)$ of if $T'$ can be obtained from $T$ by removing an edge
$e\in T$ and adding an edge $e'\in T'\setminus T$.
Maurer \cite{Mau} characterized the
basis graphs of matroids as the graphs satisfying the following three conditions: the interval
condition, the positioning condition, and the link condition (which state that the open neighborhood of each vertex is the line graph of a bipartite graph).
Using this characterization it was shown in \cite{Ch_delta} that basis graphs of matroids are meshed. From this and Proposition \ref{partial_Johnson}, we obtain the following corollary:

\begin{corollary} \label{basis-graph} The basis graph $G(M)$ of any matroid $M=(E,{\mathcal B})$ is an $S_3$-meshed graph.
\end{corollary}

It is also possible to derive Corollary \ref{partial_Johnson} from Theorem \ref{S3meshed} and Maurer's result. Indeed, as noted above, basis graphs of matroids are
meshed. On the other hand, the 5 graphs from Figure~\ref{non-S3} cannot occur in basis graphs of matroids: the first two graphs violate the positioning condition and the last three graphs
contain vertices whose neighborhoods are not line graphs of bipartite graphs.

Maurer \cite[Theorem 6.1]{Mau2} proved that the convex subgraphs of the basis graph of a matroid $M$ are in bijection with the minors of $M$. Therefore, the
semispaces and their complements have an interpretation as minors. We illustrate this interpretation in case of graphic matroids of 3-edge connected graphs.

\begin{example} Let $H=(V,E)$ be a  3-edge connected graph and let $M(H)$ be the graphic matroid of $H$. Denote by $G=G(M(H))$ the basis graph of $M(H)$.
For any edge $e_0$ of $H$, the spanning trees of $H$ can be partitioned into two sets: the set $S$ of all spanning trees not containing $e_0$ and the set $\overline{S}$ of all spanning trees containing $e_0$.
With the edge $e_0$ we have two operations on $M(H)$: the \emph{deletion} of $e_0$ is the  matroid $M\setminus e_0=M(H')$ of the graph $H'=(V,E\setminus \{ e_0\})$ obtained from $H$ by removing  $e_0$ (but keeping its ends)
and the \emph{contraction} of $e_0$ is the  matroid $M/e_0=M(H'')$ of the graph $H''$ obtained from $H$ by contracting the edge $e_0$. Then the subgraph $G(S)$ of the basis graph $G$ induced by $S$ is
the basis graph of the deletion $M\setminus e_0$  and the subgraph $G(\overline{S})$ of $G$ induced by $\overline{S}$ is the basis graph of the contraction $M/e_0$.
By \cite[Theorem 6.1]{Mau2}, $S$ and $\overline{S}$ are convex, thus they are complementary halfspaces of the basis graph $G$.

Now, consider any spanning tree $T_0$ of $G$, pick any edge $e_0$ of $T_0$, and define the partition $(S,\overline{S})$ of all spanning trees with respect to $e_0$ as before.
Removing the edge $e_0$ from $T_0$, the vertex-set $V$ of $H$ is partitioned into two connected components $V'$ and $V''$. Let $E(V',V'')$
denote the set of all edges of $H$ with one end in $V'$ and another end in $V''$. Clearly, $E(V',V'')$ is a cut of $H$ containing
the edge $e_0$. Since $G$ is 3-edge connected, $|E(V',V'')|\ge 3$. For each edge $f\in E(V',V'')\setminus \{ e_0\}$,  $T=T_0\setminus \{e_0\}\cup \{ f\}$ is a spanning tree of $H$; denote this set of spanning trees by $K$.
Since any two trees $T,T'\in K$ share the edges of $T_0\setminus \{ e_0\}$ and differ only on the two edges of $E(V',V'')\setminus \{ e_0\}$, $T$ and $T'$ are adjacent in the basis graph $G$. For the same reason,
 each $T\in K$ is adjacent to $T_0$. Consequently, $K':=K\cup \{ T_0\}$ is a clique of $G$. We assert that $K'$ is a maximal clique of $G$, unless $|E(V',V'')|\ge 2$. If not, there exists a spanning tree $T^*$ adjacent to all spanning trees of $K'$.
 If $e_0\notin T^*$, since $T_0$ and $T^*$ are adjacent, then $T_0\setminus \{ e_0\}\subset T^*$ and this implies that $T^*$ is a tree from $K$. Thus $e_0\in T^*$. This implies that  there exists
 $e\in (T_0\setminus \{ e_0\})\in T^*$, $f\in T^*\setminus T_0$, and $T^*=T_0\setminus \{ e\}\cup \{ f\}$. Pick any tree $T\in K$. Then $e_0\notin T$ and $e\in T$. Since $T$ is adjacent to $T^*$,
 necessarily $T=T^*\setminus \{ e_0\} \cup \{ e\}$. This is possible only if $|K|=1$, i.e., the cut $E(V',V'')$ has size 2.  Thus, $K'=K\cup \{ T_0\}$ is a maximal clique of $G$.

By previous results, the shadow $K/T_0$ is a semispace of $G$ and its complement is a halfspace of $G$. Since $E(V',V'')$ is a cut of $H$, any spanning tree $T$ of $H$ contains at least one edge $g$ of $E(V',V'')$.
If $T$ does not contains the edge $e_0$, then in the basis graph $G$, $T$ is closer to the spanning tree $T'=T_0\setminus \{ e_0\}\cup \{ g\}$ than to $T_0$, therefore $T$ belongs to the shadow $T/T_0\subset K/T_0$.
Vice-versa, any spanning tree $T$ belonging to the shadow $T'/T_0$ for $T'=T_0\setminus \{ e_0\}\cup \{ g\}$ and $g\in E(V',V'')\setminus \{ e_0\}$ must contain the edge $g$ and to not contain the edge $e_0$
(because $T'$ and $T_0$ coincide elsewhere except $e_0$ and $g$). Consequently, the semispace $K/T_0$ at $T_0$ and adjacent to $T_0$ coincides with the set $S$ of all spanning trees of $H$ not containing the edge $e_0$. Furthermore,
the complement of the shadow $K/T_0$ in $G$ coincides with the set $\overline{S}$ of all spanning trees passing via $e_0$. As we noticed above, they correspond to the deletion and the contraction of the edge $e_0$ in $H$ (and in $M(H)$).
\end{example}

Basis graphs of even $\Delta$-matroids are meshed and are partial half-cubes. However they are not $S_3$-graphs because the last two graphs from Figure~\ref{non-S3} can be
embedded as  subgraphs of  half-cubes. Furthermore, they do not have convex clique-shadows. Basis graphs of matroids are the 1-skeleta of the matroid polytopes, i.e.,
of the (Euclidean) convex hulls of $(0,1)$-vectors corresponding to the bases of matroids \cite{BoGeWh}. A similar result, was proved in \cite{Ch_delta} for basis graphs of
even $\Delta$-matroids. In particular, the half-cube itself is the 1-skeleton of an Euclidean polytope. Since the half-cubes $\frac{1}{2}H_d, d\ge 4$, are not $S_3$, there exists examples
of 1-skeleta of Euclidean polytopes which are not $S_3$ (and thus are not $S_4$). This answer the remark on page 89 of  \cite{VdV} that there are no examples of Euclidean polytopes, the
graph of which is not a Pach-Peano graph. On the other hand, all five forbidden graphs from Figure~\ref{non-S3} are planar and the fifth graph is also 3-connected. Therefore,
 the graphs of 3-dimensional polytopes (i.e., 3-connected planar graphs) are not in general $S_3$-graphs and thus are not  Pasch-Peano graphs.

%
%
%

\subsection{Partial Hamming graphs}
Since partial Hamming graphs are partial Johnson graphs, the following result
can be viewed as a partial solution to Question \ref{allpartialJohnson}:

\begin{proposition} \label{partial_Hamming} Any partial Hamming graph $G$ is an $S_3$-graph.
\end{proposition}

\begin{proof} We use the following property of cliques  of partial Hamming graphs $G$, which  characterizes partial Hamming graphs (see Theorem 1 of \cite{Ch_Hamming}): any clique $K$  satisfies the
following conditions:
\begin{enumerate}
\item[(1)] for each vertex $v$ of $G$ either $|\Proj_v(K)|=1$  or $\Proj_v(K)=K$;
\item[(2)] each of the sets $W_x(K)=\{ v\in V: \Proj_v(K)=\{x\}\}, x\in K$ and $W_=(K)=\{ v\in V: \Proj_v(K)=K\}$ are halfspaces of $G$.
\end{enumerate}

Let $S$ be a semispace of a partial Hamming graph $G$ and suppose that $x_0\sim S$ is an attaching vertex of $S$. Since $S$ is convex,  $K:=(N(x_0)\cap S)\cup \{ x_0\}$ is a clique of $G$.
By the second condition of the previous characterization of partial Hamming graphs, the set $W_x(K)$ and its complement $V\setminus W_x(K)$ are halfspaces of $G$. By the first condition, we conclude that
$S\subseteq (K\setminus \{ x_0\})/x_0\subseteq V\setminus W_x(K)$. Since $S$ is a semispace at $x_0$, $S= V\setminus W_x(K)$, thus $S$ is a halfspace.
\end{proof}

\subsection{(3,6)-,(4,4)-, and (6,3)-planar graphs}
We believe that large classes of partial half-cubes and of planar graphs are $S_3$. (This is due to the fact that half-cubes and, more generally, graphs isometrically embeddable into $\ell_1$-spaces
have a large amount of halfspaces because the ends of each edge are separated by at least one or two pairs of complementary halfspaces).
For example, planar bridged triangulations are $S_3$ and are partial half-cubes \cite{BaCh_wmg} (they are not partial Johnson graphs because they do not satisfy the link condition).
Now, we present three classes of planar partial half-cubes, which generalize bridged planar triangulations and which are $S_3$.

A plane graph $G=(V,E)$ (i.e., a planar graph with a planar embedding in the plane) is an \emph{$(p,q)$-graph} if each inner face of $G$ has at least $p$ edges and each inner vertex
of $G$ has degree at least $q$. To ensure non-positive curvature, Lyndon \cite{Ly0,Ly,LySc} considered and investigated three subclasses of $(p,q)$-graphs: the (3,6)-graphs, the (4,4)-graphs, and the (6,3)-graphs.
(Planar bridged triangulations  are the (3,6)-graphs in which all faces are triangles.) The $(p,q)$-graphs have their origins in small cancellation theory of groups. Beyond geometric group theory,
the (3,6)-, (4,4)-, and (6,3)-graphs have been investigated in geometry, combinatorics, and algorithms in the papers \cite{BaPe1,ChDrVa,PrSoCh,Zu} and the references cited therein.

It was shown in \cite{ChDrVa} that (3,6)-, (4,4)-, and (6,3)-graphs are partial half-cubes (for (4,4)-graphs this was established before in \cite{PrSoCh} and for bridged triangulations in \cite{BaCh_wmg}). In all cases,  the isometric embedding
is done using alternating cuts (see also \cite{ChDeGr} and \cite{GlMe} for alternating cuts in general planar graphs). 
A {\it cut} $\{ A,B\}$ of $G$ is a partition of the
vertex-set $V$ into two parts, and a {\it convex cut} is a cut in
which the halves  $A$ and $B$ are complementary halfspaces.  Denote by $E(A,B)$ the
set of all edges of $G$ having one end in $A$ and another one in
$B,$ and say that those edges are crossed (or cut)  by $\{ A,B\}.$
The {\it zone} $Z(A,B)$ of the cut $\{ A,B\}$ is the subgraph
induced by the union of all inner faces of $G$ sharing edges with
$E(A,B)$ and call the subgraphs induced by $\partial A= Z(A,B)\cap
A$ and $\partial B=Z(A,B)\cap B$ the {\it borders} of the cut $\{
A,B\}.$ A zone $Z(A,B)$ is called a {\it strip} if  $Z(A,B)$
induces a path in the dual graph of $G$.
Two edges $e'=(u',v')$ and
$e''=(u'',v'')$ on a common inner face $F$ of $G$ are called {\it
opposite} in $F$ if $d_F(u',u'')=d_F(v',v'')$ and equals the
diameter of the cycle $F.$ If $F$ is an even face, then any its
edge has an unique opposite edge, otherwise, if $F$ is an odd
face, then every edge $e\in F$ has two opposite edges $e^+$ and
$e^-$ sharing a common vertex. In the latter case, if $F$ is
oriented clockwise, for $e$ we distinguish the {\it left opposite
edge} $e^+$ and the {\it right opposite edge} $e^-$. If every face
of $Z(A,B)$ is crossed by a cut  $\{ A,B\}$  in two opposite
edges, then we say that $\{ A,B\}$ is an {\it opposite cut} of
$G.$ We say that an opposite cut $\{ A,B\}$ is {\it straight} on
an even face $F\in Z(A,B)$ and that it {\it makes a turn} on an
odd face $F\in Z(A,B).$ The turn is {\it left} or {\it right}
depending which of the pairs $\{ e,e^+\}$ or $\{ e,e^-\}$ it
crosses. An opposite cut $\{ A,B\}$ of a plane graph $G$ is {\it
alternating}  if the turns on it alternate.  
%
The following result summarizes the properties of alternating cuts:
\begin{proposition} \label{properties364463} \cite{ChDrVa} Let $G=(V,E)$ be a (3,6)-, (4,4)-, or (6,3)-graph. 
Then the following holds:
\begin{itemize}
\item[(i)] the border lines $\partial A, \partial B$ of an alternating cut $(A,B)$ are convex paths of $G$.
\item[(ii)] the alternating cuts $(A,B)$ of $G$, their zones  $Z(A,B)$, and the inner faces of $G$  are convex;
\item[(iii)] each zone $Z(A,B)$ is a strip;
\item[(iv)] each edge  of $G$ is cut by two alternating cuts (where each alternating cut whose zone consists of even
faces only is counted twice). Consequently, for any two vertices $u,v$, the number of alternating cuts separating $u$ and $v$ is
equal to $2d(u,v)$.
\end{itemize}
\end{proposition}

Proposition \ref{properties364463}(iv) implies that (3,6)-, (4,4)-, or (6,3)-graphs are partial half-cubes.
The two alternating cuts $\{ A',B'\}$ and $\{ A'',B''\}$ (which are not necessarily distinct) crossing an edge $xy$ of $G$ have the same convexity
properties with respect to the sets $W(x,y), W(y,x)$, and $W_=(x,y)$ as the partial Jonhson graphs (but (3,6)-, (4,4)-, or (6,3)-graphs are not
partial Johnson graphs because they do not satisfy the link condition).
Namely, by removing the edges of $E(A',B')\cup E(A'',B'')$ from $G$ but leaving their end
vertices, we get a graph $G^+$ whose connected components are induced by
the pairwise intersections $A'\cap A'', B'\cap B'', A'\cap B'',$
and $ A''\cap B'.$  It was shown in \cite{ChDrVa} that these
convex sets coincide with $W(x,y), W(y,x)$ and the connected
components of $W_=(x,y)$ (for an illustration, see \cite[Figure 3]{ChDrVa}:

\begin{lemma} \label{zone}  \cite{ChDrVa} $W(x,y)=A'\cap A'', W(y,x)=B'\cap
B'',$ while $W'_=(x,y):=B'\cap A''$ and $W''_=(x,y):=A'\cap B''$
constitute  a partition of $W_=(xy)$ into two (maybe empty) convex subsets.
\end{lemma}

If both connected components $W'_=(x,y)$ and $W''_=(x,y)$ are nonempty, then as was shown in \cite{ChDrVa},
$Z=Z(A',B')\cap Z(A'',B'')$  consists of one or several faces forming a strip, which at the two ends
has two odd faces $F$ and $D$ and all other faces of the strip are even. Denote by $p$ the furthest from the edge $xy$ vertex of $F$ and by $q$ the
furthest from $xy$ vertex of $D$. Denote by $p',p''$ the neighbors of $p$ in $F$ and by $q',q''$ the neighbors of $q$ in $D$ such that $pp',qq' \in E(A',B'), pp'',qq''\in E(A'',B'')$. 

\begin{lemma} \label{W'W''} If $z'\in W'_=(x,y)$ and $z''\in W''_=(x,y)$, then there exists a shortest $(z',z'')$-path of $G$ passing via the vertices $p$ and $q$.
\end{lemma}

\begin{proof}   Let $P$ be any shortest $(z',z'')$-path of $G$. Since  $z'$ and $z''$ belong to
different connected components of the graph $G^+$ and to different components of $W_=(x,y)$, and these  components are convex,
the path $P$ will leave $W'_=(x,y)$, enter one of the sets $W(x,y)$ or $W(y,x)$, say $W(x,y)$, and then leave $W(x,y)$ to enter $W''_=(x,y)$ and finally  reach $z''$.
Denote by $a'$ the first vertex of $P$ in $W(x,y)$ and by $a''$ the last vertex of $P$ in $W(x,y)$. Then necessarily, $a'$ is incident to a vertex $b'$ such that the edge $a'b'$ belongs to
$E(A',B')$ and $a''$ is incident to a vertex $b''$ such that the edge $a''b''$ belongs to $E(A'',B'')$. Since $b',b''$ belong to a shortest $(z',z'')$-path $P$,
if $p$ and $q$ will belong to a shortest $(b',b'')$-path, then $p$ and $q$
will also belong to a shortest $(z',z'')$-path. Therefore it suffices the assertion for $b'$ and $b''$, i.e., we can suppose that $z'=b'$ and $z''=b''$. Consequently, $z'\in \partial B'\cap A''$ and
$z''\in \partial B''\cap A'$.

For each pair of vertices $z'\in \partial B'\cap A''$ and
$z''\in \partial B''\cap A'$, we prove the assertion by induction on $k=d(z',p)+d(p,q)+d(q,z'')$. Let  $P_0$ be the $(z',z'')$-path of length $k$ consisting of the unique shortest path between $z'$ and $p$ (on the respective border line), the shortest path between $p$ and $q$ going via the vertices $p'$ and $q''$, and
the unique shortest path between $q$ and $z''$ (on the respective border). We assert that $P$ and $P_0$ have the same length. Clearly $|P_0|\ge |P|$. Consider any alternating cut $(A,B)$ crossing an edge of $P_0$. If each such cut also crosses
the path $P$, since each edge of $G$ is crossed by two alternating cuts and no alternating cut crosses any shortest path twice, we will deduce that $|P|\ge |P_0|$. Therefore, there exists an alternating cut
$(A,B)$ which crosses $P_0$ but does not crosses $P$. Then clearly $(A,B)$ crosses $P_0$ in two distinct edges $u'v'$ and $u''v''$ with $u',u''\in A$ and $v',v''\in B$. Since the boundary lines of alternating curs are convex, necessarily
$u'v'$ is an edge of the shortest path between $z'$ and $p$ such that $v'$ is closer to $p$ than $u'$. For the same reason, $u''v''$ is an edge of the shortest path between $q$ and $z''$ such that $v''$ is closer to $q$ than $u''$.

If $z'\ne u'$ or $z''\ne u''$, then $d(u',p)+d(p,q)+d(q,z'')<k$ and by induction assumption $p$ and $q$ belong to a shortest $(u',u'')$-path. Thus $v',v''\in [u',u'']$. This is impossible since
the border line $\partial A$ is convex, $u',u''\in \partial A$ and $v',v''\in \partial B$. Consequently, $z'=u'$ and $z''=u''$. 
$Z(A'',B'')\cap Z(A,B)$ containing the edges $z''a''$ and $z''v''$. 
Since $d(v',p)+d(p,q)+d(q,v'')<d(z',p)+d(p,q)+d(q,z'')=k$, by induction assumption,  $p$ and $q$ lie on a shortest $(v',v'')$-path. On the other hand, $v',v''\in \partial B$ and $\partial B$ is a convex path of $G$.
Consequently, the unique shortest path  between $v'$ and $v''$ must pass via $p$ and $q$. But this is impossible since $p$ and $q$ are connected by two disjoint shortest paths, one passing via $p'$ and $q'$ and another via $p''$ and $q''$.
This concludes the proof.
\end{proof}

From previous lemmas, we can easily deduce the following result:

\begin{proposition} (3,6)-, (4,4)-, or (6,3)-graphs are $S_3$.
\end{proposition}

\begin{proof} Let $S$ be a semispace of $G$ having $x_0\sim S$ as an attaching vertex and let $y$ be a neighbor of $x_0$ in $S$. Consider the two alternating
cuts $\{ A',B'\}$ and $\{ A'',B''\}$ crossing the edge $x_0y$ and suppose that $x_0\in A'\cap A''$ and $y\in B'\cap B''$. Pick any vertex $z$ of $S$. Then either $z\in W(y,x_0)\subseteq B'\cap B''$
or $z\in W_=(x_0,y)$. We assert that $S$ cannot contain two vertices $z',z''\in W_=(x_0,y)$ such that $z'\in W'_=(u,v)$ and $z''\in W''_=(u,v)$. Suppose by way of contradiction that such $z'$ and $z''$ exist.
By Lemma \ref{W'W''}, there is a shortest $(z',z'')$-path passing via the vertices $p$ and $p$. But $p$ and $q$ are connected by two shortest paths, one belonging to $A'\cap A''$ and the second belonging to
$B'\cap B''=W(x_0,y)$ (in fact the second path contains the vertex $x_0$). Consequently, $[z',z'']$ intersects $W(x_0,y)$. Since $z',z''\in S$ and $S$ is convex, this contradicts the fact
that  $S\subseteq W(x_0,y)\cup W_=(x_0,y)$.
\end{proof}

\begin{question} Are (3,6)-, (4,4)-, or (6,3)-graphs $S_4$ (i.e., Pasch)?
\end{question}

\subsection{Summary of examples} We conclude by summarizing the known examples of $S_3$-graphs and meshed $S_3$-graphs.
The following graphs are $S_3$:
\begin{itemize}
\item partial cubes,
\item partial Hamming graphs,
\item partial Johnson graphs satisfying (TC),
\item (3,6)-,(4,4)-, and (6,3)-graphs,
\item the Petersen graph and the dodecahedron.
\end{itemize}
Additionally, the following graphs are meshed $S_3$-graphs:
\begin{itemize}
\item hyperoctahedra, complete graphs, the icosahedron, and  the graph $\Gamma$ from Figure \ref{non-convexshadows},
\item basis graphs of matroids,
\item median, quasi-median, and weakly median graphs,
\item the 2-dimensional $\ell_\infty$-grid and any its subgraph contained in the region of ${\mathbb R}^2$ bounded by a simple closed spath of the grid (they are Helly and do not contain the graphs from Figure~\ref{non-S3}).
\end{itemize}
%
%
%
%

\section{Halfspace separation problem} \label{s:halfspacesep}
In this section we consider the halfspace separation problem. This problem is well-known in machine learning for  sets in ${\mathbb R}^d$
\cite{BoGuVa,MiPa} and was recently introduced and studied for general convexity spaces by Seiffart, Horv\'{a}th, and Wrobel \cite{SeHoWr}.

\begin{definition} [Halfspace separation problem] \cite{SeHoWr} Given a pair $(A,B)$ of sets of a convexity space $(X,\C)$, the \emph{halfspace separation problem} asks if $A$ and $B$ are separable by complementary halfspaces $H',H''$ and to find such separating halfspaces, if they exist.
\end{definition}

We refer to  its version when $B$ is a single point $x_0$, the
\emph{halfspace $S_3$-separation problem}. The halfspace separation problem is NP-complete  \cite{SeHoWr} because deciding if a graph has a pair of complementary halfspaces is NP-complete \cite{ArDaDoSz}.
We consider two methods for solving this problem, by enumeration of its halfspaces and by reducing the problem to the separation of adjacent disjoint convex sets $A,B$ which coincide with  their mutual shadows. We apply
these methods to geodesic, gated, and induced path convexities in graphs.

\subsection{Halfspace separation problem in $S_4$- and $S_3$-convexity spaces}
Here we show how to solve the halfspace separation problem in finite $S_4$-convexity spaces, in which the convex hull of sets can be computed efficiently. We check if $\conv(A)$ and $\conv(B)$ are disjoint and return
answer ``yes'' if they are disjoint and answer ``not'' otherwise. To compute the complementary halfspaces $H'$ and $H''$ separating $A$ and $B$ we proceed as follows. Set $A:=\conv(A)$ and $B:=\conv(B)$ and while $A\cup B\ne X$,
for each point $x\notin  A\cup B$ we check if $\conv(A\cup \{ x\})\cap B=\varnothing$ or $\conv(B\cup \{ x\})\cap A=\varnothing$. Since $A,B$ are disjoint and  $\C$ is $S_4$, at least one of the intersections if empty (otherwise the
shadows $A/B$ and $B/A$ intersect and the convex sets $A$ and $B$ cannot be separated, contrary to $S_4$). If $\conv(A\cup \{ x\})\cap B=\varnothing$,
then we set $A:=\conv(A\cup \{ x\})$, otherwise we set $B:=\conv(B\cup \{ x\})$. In both cases, we continue. When $A\cup B=X$, we return $H'=A$ and $H''=B$.
In fact, this is Algorithm 1 of  \cite{SeHoWr} and the proof of Theorem  \ref{S4arity}(1) of \cite{Ch_separa,Ch_S3} in the finite case (where the maximality choice using Zorn's lemma is replaced by the algorithmic loop).
This algorithm can be also adapted to solve the halfspace $S_3$-separation problem in $S_3$-spaces $(X,\C)$. Given a set $A$ and a point $x_0\notin A$, first we check if $x_0\in \conv(A)$. If $x_0\in \conv(A)$, then we return the answer ``not'', otherwise
we return ``yes'' (because $(X,\C)$ is $S_3$). To find the separating halfspaces, we set $A:=\conv(A)$ and while there exists a point $y\in X\setminus (A\cup \{ x_0\})$ such that $x_0\notin \conv(A\cup \{ y\})$, we set
$A=\conv(A\cup \{ y\})$ and continue. Then the final set $A$ is a semispace at $x_0$ containing the initial set $A$, thus $A$ will be also a halfspace by $S_3$.

\subsection{Halfspace separation using the halfspace enumeration} Let $A$ and $B$ be two disjoint sets of a convexity space $(X,\C)$ on $n$ points. 
To solve the halfspace separation problem for $(A,B)$, at preprocessing stage we enumerate the set $\mathcal H$ of all pairs of disjoint halfspaces of $(X,\C)$ and then we test if $\mathcal H$ contains a pair $(H',H'')$ of complementary halfspaces such that
$A\subseteq H'$ and $B\subseteq H''$. Clearly, this method is polynomial if and only if $\mathcal H$ has polynomial size.

Surprisingly, the problem of enumeration of complementary halfspaces of geodesic convexity in graphs was considered
by Glanz and Meyerhenke \cite{GlMe} (with the motivation of minimizing communication costs in parallel computations). They considered convex cuts in graphs (which define complementary halfspaces),
proved that there exists a polynomial number of convex cuts in bipartite graphs and in planar graphs, and show how to compute them in those two cases.
They also provide a (non-algorithmic) characterization of convex cuts in general graphs. The fact that any bipartite graph $G=(V,E)$ may contain at most $|E|$ pairs
of complementary spaces and the way to enumerate the pairs of complementary halfspaces,
follows from the following simple lemma (which can be compared with Proposition 2.1 of \cite{GlMe}):

\begin{lemma} If $G=(V,E)$ is a bipartite graph, $uv$ and edge of $G$, and $H',H''$ are complementary halfspaces with $u\in H'$ and $v\in H''$, then $H'=W(u,v)$ and $H''=W(v,u)$.
\end{lemma}

\begin{proof} By Lemma \ref{shadow-separation}, $W(u,v)=u/v\subseteq H'$ and $W(v,u)=v/u\subseteq H''$. Since $G$ is bipartite, $W(u,v)\cup W(v,u)=V$. Therefore, if $H'$ is not contained in $W(u,v)$, then there exists $x\in H'\cap W(v,u)$.
Since $v\in [x,u]$ and $x,u\in H', v\in H''$, in contradiction with the convexity of $H'$.
\end{proof}

Glanz and Meyerhenke \cite{GlMe} proved that each pair $u,v$ of adjacent vertices of a planar graph $G=(V,E)$ with $n$ vertices are separated by at most $O(n^4)$ complementary halfspaces, thus $G$ contains at most $O(n^5)$ halfspaces.
Furthermore, they showed that
all pairs of complementary halfspaces of $G$ can be enumerated in $O(n^7)$. 
Next, we present a generalization of the first result of \cite{GlMe}, using a different approach and meeting their upper bounds in case of planar graphs.

In the proof of the following result we use the following classical notions. Let $(X,{\mathcal F})$ be a family of sets. A subset $Y$ of $X$ is \emph{shattered}
by $\mathcal F$ if for any $Y'\subseteq Y$ there exists $F\in \mathcal F$ such that $Y\cap C=Y'$ (i.e., the trace of $\C$ on $Y$ is $2^Y$). The
\emph{Vapnik-Chervonenkis dimension} (VC-dimension)~ \cite{VaCh} $\vcd(\mathcal F)$ of $\mathcal F$ is the cardinality of the
largest subset of $X$ shattered by $\mathcal F$.  The VC-dimension   was introduced by~\cite{VaCh}
as a complexity measure of set systems. VC-dimension  is  central  in PAC-learning and  plays an important role in combinatorics,
algorithmics, and discrete geometry. One of the basic facts in VC-theory is the \emph{Sauer lemma} \cite{Sa}, which asserts that a set family $\mathcal F$ of VC-dimension $d$
on a set of $n$ elements contains at most $O(n^d)$ sets. For a graph $G$, the \emph{Hadwiger number} $\eta(G)$ is the largest integer $r$ such that $G$ can be contracted (by successive contraction
of some of its edges) to the complete graph $K_r$ on $r$ vertices.


\begin{theorem}\label{numberofhalfspaces} If $(X,\C)$ is a finite convexity space with $n$ points and Radon number $r$, then $(X,\C)$ contains at most $O(n^r)$ halfspaces. If $\C$ is a convexity
with connected convex sets on a graph $G=(V,E)$ with $n$ vertices and not containing $K_{k+1}$ as a minor, then $\C$ contains at most $O(n^{2k})$ halfspaces. If $G$ is planar, then $\C$ contains
at most $O(n^5)$ halfspaces. Finally, if $G$ is any graph endowed with monophonic convexity, then $G$ contains at most $O(n^{\omega(G)})$ halfspaces.
\end{theorem}

\begin{proof} Let $\cH$ denote the set of all halfspaces of  $(X,\C)$. We assert that $\vcd(\cH)\le r$ by adapting an argument from \cite{Du} for balls in
Euclidean spaces (the same argument works for halfspaces). Suppose by way of contradiction that $\vcd(\cH)>r$. Then $X$ contains a set $Y$ with $r+1$ points that can be shattered by $\cH$. Since $r(\C)=r$, the set $Y$ admits a Radon partition,
i.e., a partition of $Y$ in two sets $Y'$ and $Y''$ such that $\conv(Y')\cap \conv(Y'')\ne \varnothing$. Since $\cH$ shatters $Y$, there exists a halfspace $H'$ such that $H'\cap Y=Y'$. This implies that
$Y''$ belongs to $H''=X\setminus H'$. Since $H''$ is convex, $\conv(Y'')\subseteq H''$, contrary to the assumption that $\conv(Y'')$ contains a point of $\conv(Y')\subseteq H'$.
This shows that  $\vcd(\cH)\le r$ . By Sauer's lemma, $\cH$ contains at most  $O(n^r)$ halfspaces.

Now, let $\C$ be a convexity on a graph $G=(V,E)$ such that each convex set induces a connected subgraph of $G$. Duchet and Meyniel \cite[Th\'eor\`eme 2.6]{DuMe}  proved that
the Radon number $r(\C)$ of $\C$ is at most $2\eta(G)$. Since $G$ cannot be contracted to $K_{n+1}$, $\eta(G)=k$, thus $r(\C)\le 2k$. In case of planar graphs,
Duchet and Meyniel \cite[Proposition 3.4]{DuMe}
proved that $r(\C)\le 5$. For monophonic convexity, Duchet \cite{Du} proved that the Radon number of a graph $G$. is $\omega(G)$. Now, all three assertions follow from the first assertion of the theorem.
\end{proof}

Since the Radon number of Helly graphs is equal to the clique number $\omega(G)$ \cite{BaPe_radon} and  the
Radon number of chordal graphs is $\omega(G)$ except when $\omega(G)=3$ \cite{Ch_triangulated}, in which
case $r(G)$ is 3 or 4, we obtain the following corollaries of Theorem \ref{numberofhalfspaces}:

\begin{corollary}\label{chordal-helly} A Helly graph $G=(V,E)$ with $n$ vertices and clique number $\omega(G)$ contains at most $O(n^{\omega(G)})$ halfspaces.
Any chordal graph $G=(V,E)$ with $n$ vertices and clique number $\omega(G)\ne 3$ (respectively, $\omega(G)=3$) contains at most $O(n^{\omega(G)})$
halfspaces (respectively, $O(n^{4})$ halfspaces).
\end{corollary}

\begin{remark} We believe that the bouds given by Corollary \ref{chordal-helly} are not sharp and that Helly and chordal graphs with  $\omega(G)\ne 3$  contain at most $O(2^{\omega(G)}\poly(n))$ halfspaces
\end{remark}

An algorithm with complexity $O(2^{\omega(G)}\poly(n))$ for enumerating monophonic halfspaces of graphs has been recently proposed by
Bressan, Esposito, and Thiessen \cite{BrEsTh}. Now we present an algorithm for  enumeration of halfspaces of Helly and chordal graphs. This method  works for more general classes of meshed
and weakly modular graphs, in particular for bridged graphs, but in this case its complexity
depends of the Hadwiger number $\eta(G)$. The algorithm is based on the notion of dismantling.
A vertex $u$ \emph{dominates} a vertex $v$ in a graph $G$ if $u\sim v$ and any neighbor $w\ne u$ of $v$  is adjacent to $u$. A \emph{dismantling ordering}
of a finite graph $G=(V,E)$ is a total order $v_1,\ldots,v_n$ of  $V$  such that each $v_i$ is dominated in $G_i=G(\{v_1,\ldots,v_i\}$ by a vertex $v_j$
with $j<i$.  A class of graphs $\mathcal G$ is  \emph{hereditary dismantlable} if any graph $G\in {\mathcal G}$ admits a dismantling order $v_1,\ldots,v_n$ such that
all level graphs $G_i$ belong to $\mathcal G$. It was shown in \cite{BaPe_helly} that Helly graphs are hereditary dismantlable. It was shown in \cite{AnFa} that bridged graphs are
dismantlable (in \cite{Ch_bridged} it was shown that bridged graphs are dismantlable by BFS). Since bridged graphs are hereditary by taking isometric subgraphs, they are
hereditary dismantlable. Finally, weakly bridged graphs are dismantlable by LexBFS \cite{ChOs}. From this and \cite[Theorem A(e)]{ChOs} it follows
that weakly bridged graphs are hereditary dismantlable. The following lemma is the basis of our enumeration algorithm:

\begin{lemma} \label{dominating-halfspce} Let $G=(V,E)$ be a meshed graph and $v$ be a dominated vertex of $G$ such that the subgraph $G'=G(V\setminus \{ v\})$ is also meshed. Then each halfspace of
$G$ containing $v$ either coincides with $\{ v\}$ or has the form $H'\cup \{ v\}$, where $H'$ is a halfspace of $G'$.  Each halfspace of $G$ not containing $v$
either coincides with $V\setminus \{ v\}$ or coincides with $H''$, where $H''$ is a halfspace of $G'$.
\end{lemma}

\begin{proof} Let $u$ be a vertex of $G$ dominating $v$. Pick any halfspace $H$ of $G$ containing $v$. If $H=\{ v\}$ (i.e., all neighbors of $v$ are pairwise adjacent), then
we are in the first case. In this case $V\setminus \{ v\}$ is a halfspace of $G$ not containing $v$. Now suppose that $H=H'\cup \{ v\}$, where  $H'$ is a proper nonempty subset of
$V'=V\setminus \{ v\}$. Let $H''=V'\setminus H'$. We assert that $H'$ and $H''$ are halfspaces of $G'$. Since $u$ dominates $v$, $G'$ is an isometric subgraph of $G$.
Since $H''=V\setminus H$, $H''$ is a halfspace of $G$. Therefore $H''$ is a convex set of $G'$. Now we show that $H'$ is a convex set of $G'$.
If $H'$ induces a connected subgraph of $G'$, since $G'$ is meshed, by Theorem \ref{lc->c} it suffices to prove that $H'$ is locally-convex.
This is the case because $H=H'\cup \{v\}$ is convex, and thus locally-convex.
It remains to show that $H'$ is connected in $G'$. Since $H$ is convex and thus connected in $G$, this is obviously true if $u$ belongs to $H$ and thus to $H'$. Now, suppose that $u$ belongs to $H''$.
Since $H$ is convex in $G$ and $u$ dominates $v$, all neighbors of $v$ in $H'$ are pairwise adjacent. Hence $H'$ is connected, concluding the proof.
\end{proof}

By Lemma \ref{dominating-halfspce}, to enumerate the halfspaces of $G$ it suffices to enumerate the halfspaces of $G'$. For each pair of complementary halfspaces $(H',H'')$ of $G'$, we check
if $(H'\cup \{ v\},H'')$ or $(H',H''\cup \{ v\}$ are complementary halfspaces of $G$ and return the respective pair(s) if this is the case. Additionally, we check if  $(\{ v\}, V\setminus \{v\})$
are complementary halfspaces and return this pair if this is the case. This can be done in time polynomial in the number of halfspaces of $G'$ and the number of vertices of $G$.
Now, let $G$ be a meshed graph admitting a hereditary dismantling order $v_1,\ldots,v_n$. Then we perform the previous algorithm for each pair $G_i$ and $G_{i-1}$ (with $G_i$ as $G$ and $G_{i-1}$ as $G$).
Since $\eta(G_i)\le \eta(G)$ and $\omega(G_i)\le \omega(G)$, we obtain an algorithm with complexity $\poly(n)n^{2\eta(G)}$ for weakly bridged graphs and with complexity $\poly(n)n^{2\omega(G)}$
for chordal graphs and Helly graphs (if Question \ref{helly-radon-bridged} has positive answer, this will lead also to an algorithm with complexity $\poly(n)n^{O(\omega(G))}$ for bridged graphs).
This algorithm is not output polynomial since the intermediate graphs $G_i$ may have more halfspaces that the final graph $G$.

Summarizing the results of this subsection and using the result of \cite{GlMe} for planar graphs, we obtain the following conclusion for halfspace separation problem via halfspace enumeration:

\begin{proposition} \label{enum-halfspaces} The halfspace enumeration and the halfspace separation problems can be solved for geodesic convexity in a graph $G$ with $n$ vertices in the
following classes of graphs:
\begin{enumerate}
\item[(1)] $O(\poly(n))$ time if $G$ is bipartite;
\item[(2)] \cite{GlMe} $O(\poly(n))$ if $G$ is planar;
\item[(3)] $O(\poly(n)n^{\omega(G)})$ if $G$ is chordal or Helly;
\item[(4)] $O(\poly(n)n^{2\eta(G)})$ if $G$ is meshed and admits a hereditary dismantling order.
\end{enumerate}
\end{proposition}

\begin{question} Can the halfspaces of a graph $G$ not containing $K_{k+1}$ as a minor be enumerated in $O(\poly(n)n^{2k})$ time?
Can the halfspaces of a chordal, Helly graph, or basis graph of a matroid  be enumerated in
ourtput polynomial time?
\end{question}

\subsection{Halfspace separation using the three steps method} We present the \emph{three steps method} for a convexity $\C$ on a graph $G=(V,E)$ with $n$ vertices, such that the convex sets of $\C$ are also geodesically convex (this is the case of gated  and monophonic convexities). The first step  applies to all convexity spaces. Let $A$ and $B$ be two disjoint sets and suppose that their convex hulls $A:=\conv(A)$ and $B:=\conv(B)$ are disjoint.
To solve the halfspace separation problem for $(A,B)$, from Lemma \ref{shadow-separation} we know that for any pair of complementary halfspaces $H',H''$ with $A\subseteq H'$ and $B\subseteq H''$, $H'$ contains the convex hull $\conv(A/B)$ of the shadow $A/B$
and $H''$ contains the  convex hull $\conv(B/A)$ of the shadow $B/A$. Then $A,B$ are separable if and only if $\conv(A/B)$ and $\conv(B/A)$ are separable. Consequently, if $\conv(A/B)$ and $\conv(B/A)$ are not disjoint, then $A$ and $B$ are not  separable and the answer is ``not''.
Otherwise, we replace $A$ and $B$ by $\conv(A/B)$ and $\conv(B/A)$, construct their mutual shadows and the convex hulls of those shadows, test if they are disjoint, and reiterate while is possible. If the answer ``not'' was not provided, then we will end up with a pair $(A^*,B^*)$ of disjoint convex sets  such that $A\subseteq A^*, B\subseteq B^*$ and $A^*=A^*/B^*, B^*=B^*/A^*$. Furthermore $A,B$ are separable if and only if $A^*,B^*$ are separable.
We call the pair $(A^*,B^*)$  \emph{shadow-closed}. The shadows and their convex hulls can be constructed using a linear number of calls
to an algorithm for convex hulls. Therefore, the shadow-closed pair  $(A^*,B^*)$ for the input pair $(A,B)$ can be constructed in polynomial time. This step is called the \emph{shadow-closure step}.
If $A^*\cup B^*=V$, then clearly $A^*,B^*$ are complementary halfspaces separating $A$ and $B$ and we return the answer ``yes''.

Now, let $(A^*,B^*)$ be a shadow-closed pair such that the set $V\setminus (A^*\cup B^*)$ is nonempty. A pair  $(A^*,B^*)$ is said \emph{osculating} if $d(A^*,B^*)=1$, i.e., the sets $A^*$ and $B^*$ are disjoint and adjacent.
If $A^*$ and $B^*$ are osculating, then we do nothing. Now, let $d(A^*,B^*)>1$.  Let $u^*\in A^*,v^*\in B^*$ be a closest pair of vertices, i.e., $d(u^*,v^*)=d(A^*,B^*)$. Let $P=(u^*=u_1,u_2,\ldots,u_{k-1},u_k=y^*)$
be any shortest $(u^*,v^*)$-path. If there exist  complementary halfspaces $H',H''$ with $A^*\subseteq H'$ and $B^*\subseteq H''$, then $P$ contains an edge $u_iu_{i+1}$ with $u_i\in H'$ and $u_{i+1}\in H''$
(i.e., the convex cut defined by $H',H''$ cuts the path $P$). Since we are searching for the pair  $(H',H'')$ and we do not know the respective edge of $P$, we simply perform the following procedure
with each edge $u_iu_{i+1}$ of $P$. We set $A^+_i=\conv(A^*\cup \{ u_i\})$ and $B^+_i=\conv(B^*\cup \{ u_{i+1}\})$, $i=1,\ldots,k-1$. To each of the pairs $(A^+_i,B^+_i)$ such that $A^+_i$ and $B^+_i$ are disjoint, we perform
the shadow-closure step. If this step does not return the answer ``not'' for the pair  $(A^+_i,B^+_i)$, then we denote by $(A^{**}_i,B^{**}_i)$ the resulting shadow-closed pair. If no such shadow-closed pair is returned, then $A^*,B^*$ (and thus $A$ and $B$)
are not separable and we return the answer ``not''. Otherwise, the second step of the algorithm will return several shadow-closed osculating pairs $(A^{**}_i,B^{**}_i)$ of convex sets. 
Again, if one such pair $(A^{**}_i,B^{**}_i)$ partitions $V$, then we return the answer ``yes'' and return it as a pair of separating halfspaces.
Similarly to the first step, the second step requires polynomial time.

For each shadow-closed osculating pair $(A^{**},B^{**})$ returned by the second step, the third step have to solve the halfspace separation problem. Since the first two steps run in polynomial
time and the halfspace separation problem in NP-complete \cite{SeHoWr}, with no surprise, the problem is NP-complete for shadow-closed osculating pairs. Nevertheless, we can solve it efficiently in
several cases. Before going to particular cases, we formulate the framework for this specific separation problem. Let $(A^{**},B^{**})$ be a shadow-closed osculating pair and suppose that $V^+=V\setminus (A^{**}\cup B^{**})$ is a nonempty
residue, whose vertices have to be distributed with $A^{**}$ or  $B^{**}$. Pick any edge $uv$ of $G$ with $u\in A^{**}$ and $v\in B^{**}$. Since $(A^{**},B^{**})$ is shadow-closed, for any vertex $x\in V^+$ we must have $d(x,u)=d(x,v)$, otherwise
$x$ belongs to  $v/u\subseteq B^{**}/A^{**}$ or to $u/v\subseteq A^{**}/B^{**}$, which is impossible. In case of geodesic convexity in bipartite graphs, the residue $V^+$ is always empty, thus
the  halfspace separation problem can be solved using only the first two steps. 
Summarizing, the three steps method works in the following way:
\begin{enumerate}
\item[(1)] Given the input pair $(A,B)$, compute the shadow-closed pair $(A^*,B^*)$. Return ``not'' if $A^*$ and $B^*$ are not disjoint.
\item[(2)] Given a shadow-closed pair $(A^*,B^*)$,  pick a shortest path $P$ between a closest pair of $(A^*,B^*)$ and for each edge $u_iu_{i+1}$ of $P$
set  $A^+_i=\conv(A^*\cup \{ u_i\})$ and $B^+_i=\conv(B^*\cup \{ u_{i+1}\})$ and for $(A^+_i,B^+_i)$ execute the step 1 and compute the shadow-closed pair
$(A^{**}_i,B^{**}_i)$. If $A^{**}_i$ and $B^{**}_i$ are not disjoint for all edges of $P$, then return ``not''.
\item[(3)] For each shadow-closed osculating pair $(A^{**}_i,B^{**}_i)$ returned at step 2, solve the halfspace separation problem using a case-oriented algorithm. If ``yes'' is returned for at least one such pair, then
return the answer ``yes'' for the input pair $(A,B)$, otherwise, return ``not''.
\end{enumerate}

\subsubsection{Gated convexity} Let $\ccG$ be the family of all gated sets of $G$. 
As we noted above, each gated set
is geodesically convex. In general, a graph contains much less gated sets than geodesic convex sets, but in the important case of median graphs all geodesic convex sets are gated.  Let $\cg(S)$  be the gated hull of a set $S$;
$\cg(S)$ can be constructed as follows. First, we construct the geodesic convex hull $\conv(S)$ of $S$. Then for each $x\in V\setminus \conv(S)$, compute its imprint $\Imp_x(\conv(S))$ in $\conv(S)$. If $|\Imp_x(\conv(S))|=1$ for any $x\in V\setminus \conv(S)$,
then $\conv(S)$ is gated and we return it. Otherwise, we find  $x\in V\setminus \conv(S)$ such that $\Imp_x(\conv(S))$ contains two distinct vertices $u,v$. Then we find a quasi-median of $x,u,v$. Since $\conv(S)$ is convex and $u,v\in \Imp_x(\conv(S))$, this quasi-median has the form $x_0uv$, where $x_0$ is a furthest from $x$ vertex in  $[x,u]\cap [x,v]$ (where we consider the usual geodesic intervals). The next lemma (which can be viewed as a general
shows that $x_0$ belongs to $\cg(\conv(S))$: 

\begin{lemma} \label{trianglegated}  If $x_0uv$ is a metric triangle and $A$ is a gated set with $u,v\in A$, then $x_0$  belongs to $A$.
\end{lemma}

\begin{proof} Since $A$ is gated, $x_0$ has a gate $x'_0$ in $A$. By definition of the gate, $x'_0\in [x_0,u]\cap [x_0,v]$. Since $x_0uv$ is a metric triangle, necessarily $x_0=x'_0$, i.e., $x_0\in A$.
\end{proof}

By Lemma \ref{trianglegated}, $x_0$ belongs to any gated set containing $u$ and $v$, thus $x_0\in \cg(\conv(S))$. Therefore, set $S:=\conv(\{ x_0\}\cup \conv(S))$ and continue.  Consequently, the gated hull $\cg(S)$ can be constructed in polynomial time.

Now, let $(A,B)$ be a shadow-closed osculating pair of  gated sets of  $G$. We assert that if the residue $V^{+}=V\setminus (A\cup B)$ is non-empty, then the halfspace separation problem for  $(A,B)$ has answer ``not''. Pick  $x\in V^{+}$ and pick any edge $uv$ with $u\in A$ and $v\in B$. As noted above, $d(x,u)=d(x,v)$. Let $x_0$ be the furthest from $x$ vertex belonging to $[x,u]\cap [x,v]$. Then $x_0uv$ is a metric triangle of $G$.
Note that $x_0$ also belongs to the residue  $V^{+}$. Indeed, if this is not the case and say $x_0\in A$, then $x_0\in [x,v]$, yielding $x\in x_0/v\subseteq A/B$, in contradiction with $x\in V^{+}$ and shadow-closeness of $(A,B)$.
Consequently, $x_0\in V^{+}$. Now, suppose by way of contradiction that the halfspace separation problem has a positive answer for  $(A,B)$ and $H',H''$ are complementary gated halfspaces with $A\subseteq H'$ and $B\subseteq H''$. Then $x_0$ belongs to $H'$ or $H''$, say $x_0\in H'$. Since $x_0uv$ is a metric triangle, $x_0,u\in H', v\in H''$, and $H',H''$ are gated, we obtain a contradiction with Lemma \ref{trianglegated}. This proves the following result:

\begin{proposition} \label{halfspacesepgated} If $(A,B)$ is a shadow-closed osculating pair of  gated sets of a graph $G$, then either $A$ and $B$ are complementary gated halfspaces of $G$ or  $A$ and $B$ are not separable. Consequently,
the gated halfspace separation problem can be solved in polynomial time.
\end{proposition}

In fact, the argument used to prove Proposition \ref{halfspacesepgated} can be also used to show that the gated halfspace separation problem can be solved by halfspace  enumeration.

\begin{proposition} If $uv$ is an edge of a graph $G$ and there exists a pair of complementary gated halfspaces $H',H''$ of $G=(V,E)$ with $u\in H'$ and $v\in H''$, then $H'=W(u,v)$ and $H''=W(v,u)$.
Consequently, $G$ contains at most $|E|$ pairs of complementary gated halfspaces, which can be enumerated in polynomial time.
\end{proposition}

\begin{proof} By Lemma \ref{shadow-separation}, $W(u,v)=u/v\subseteq H'$ and $W(v,u)=v/u\subseteq H''$. If there exists $x\in V\setminus (W(u,v)\cup W(v,u))$, then $d(x,u)=d(x,v)$. Let $x_0$ be a furthest from
$x$ vertex in $[x,u]\cap [x,v]$. If $x_0\in H'$, since $x_0uv$ is a metric triangle of $G$ and $x_0,u\in H'$, by Lemma \ref{trianglegated}, $v$ must belong to $H'$, which is impossible because $v\in H''$.
\end{proof}

\subsubsection{Monophonic convexity} 
Let $\cM$ denote the set of all monophonically convex sets of a graph $G$ and let $\cm(S)$ denotes the \emph{monoponic (convex) hull} of $S$. Note that $(V,\cM)$ is an interval convexity space.
Since shortest paths are induced paths,
monophonically convex sets are geodesically convex. In general, a graph contains much less monophonic convex sets than geodesically convex sets. Monophonic convexity was introduced in \cite{FaJa_mo} and further investigated in \cite{Ba_intrinsic,Du_mo}; for complexity issues, see \cite{DouPrSzw}.  The monophonic  hull $\cm(S)$ can be computed in polynomial
time. First, we construct the geodesic convex hull $\conv(S)$ of $S$. Then for each pair of non-adjacent vertices $u,v$ of $\conv(S)$ we compute if $u$ and $v$ can be connected by a path outside $\conv(S)$, i.e., in the subgraph $G_{u,v}$ of $G$ induced
by $(V\setminus \conv(S))\cup \{ u,v\}$. If there exists such a pair $u,v\in \conv(S)$, then in $G_{uv}$ we compute a shortest $(u,v)$-path $P$, set $S:=\conv(S)\cup P$, and continue (i.e., compute the geodesic convex hull of the  new set $S$, find an outside path, etc). Now, if  there is no pair $u,v\in \conv(S)$ which can be connected outside $\conv(S)$, then we return the current $\conv(S)$ as the monophonic hull $\cm(S)$.


%

Let $(A,B)$ be a shadow-closed osculating pair of monophonic convex sets of $G$. 
Let $\partial A$ the set of all  $u\in A$ having a neighbor $v\in B$; $\partial B$ is defined analogously. Set $V^{+}=V\setminus (A\cup B)$ and let
$V^{+}_0$ be the set of all vertices $x_0\in V^{+}$ such that $x_0$ is adjacent to a vertex of $A\cup B$. 

\begin{lemma} \label{boundaries} $\partial A$ and $\partial B$ are cliques of $G$.
\end{lemma}

\begin{proof} If $u,u'\in \partial A$ and $u\nsim u'$, then we assert that $u$ and $u'$ can be connected by a path $P$ whose all intermediate vertices belong to $B$. Indeed, $u$ is adjacent to  $v\in B$ and
$u'$ is adjacent to $v'\in B$. Since $B$ is monophonic convex, $v$ and $v'$ are connected inside $B$ by a path $P'$. Then the path $P'$ together with the edges $uv$ and $v'u'$ constitute the requested path $P$.
Now, since $u\nsim v$, $P$ contains an induced $(u,v)$-path, contrary to the assumption that $A$ is monophonic convex.
\end{proof}

\begin{lemma} \label{x_0} Any $x_0\in V^{+}_0$ is adjacent to all vertices of $\partial A\cup \partial B$. 
\end{lemma}

\begin{proof} 
Suppose that $x_0$ is adjacent to $z\in A\cup B$, say $z\in A$. Pick $u\in \partial A$ and $v\in \partial B$ with $u\sim v$.
We assert that $x_0$ is adjacent to $u$ or to $v$.
Suppose that among all neighbors of $x_0$ in $A$, $z$ is closest to $u$.  Let $P'$ be a shortest path connecting $z$ and $u$. Since $A\in \cM$, $P'\subseteq A$. From the choice of $z$, either $x\sim u$ or
the path $P''$ consisting of $P'$ plus the edge $zx$ is an induced $(x,u)$-path. Hence, $x$ and $v$ are connected by a path $P$ ($P''$ plus the edge $uv$) whose all intermediate vertices
belong to $A$. If $x\nsim v$, from $P$ we can extract an $(x,v)$-path with all intermediate vertices in $A$, yielding $x\in A/v\subseteq A/B$, contrary to the assumption that $(A,B)$ is shadow-closed and $x\in V^+$.
Consequently, $x$ is adjacent to $u$ or $v$. Since $x\notin u/v\cup v/u$, $x$ is also adjacent to the second vertex. Since $u$ and $v$ with $u\sim v$ have been chosen arbitrarily,
$x$ is adjacent to all vertices of $\partial A\cup \partial B$.
%
\end{proof}


\begin{lemma} \label{monophonicquasimedian} If $x\in V^{+}, u\in \partial A, v\in \partial B$ with $u\sim v$, then any quasi-median of the triplet $x,u,v$ has the form $x_0uv$, where $x_0\sim u,v$ and $x_0\in V^+$.
\end{lemma}

\begin{proof} Since $x\in V^+$ and $(A,B)$ is shadow-closed, necessarily $d(x,u)=d(x,v)$. Since $u\sim v$, any quasi-median of $x,u,v$ has the form $x_0uv$ with $d(x_0,u)=d(x_0,v)=k$. First suppose that
$k>1$. Pick any shortest $(x_0,u)$-path $P'$ and any shortest $(x_0,v)$-path $P''$. Since $[x_0,u]\cap [x_0,v]=\{ x_0\}$, $u$ is not adjacent to any vertex of $P''\setminus \{ v\}$ and $v$ is not adjacent
to any vertex of $P'\setminus \{ u\}$. This implies that $P'$ plus the edge $uv$ is an induced $(x_0,v)$-path passing via $u$ and $P''$ plus the edge $vu$ is an induced $(x_0,u)$-path passing via $u$. Consequently,
$x_0\in u/v\cap v/u\subseteq A/B\cap B/A$, contrary to the assumption that $A$ and $B$ are disjoint. Consequently, $x_0$ is adjacent to $u$ and $v$. If $x_0$ belongs to $A$ or to $B$, say $x_0\in A$,
then $x\in x_0/v\subseteq A/B$, contrary to the choice of $x$ from $V^+$.
\end{proof}

For  $x\in V^{+}$, let $S_x$ be the set of all  $x_0\in V^{+}_0$ such that there exist $u\in \partial A,v\in B$ with $u\sim v$ and $x_0\in \cm(x,u)\cap \cm(x,v)$. Since $[x,u]\subseteq \cm(x,u)$ and
$[x,v]\subseteq \cm(x,v)$, Lemma \ref{monophonicquasimedian} implies that each  $S_x$ is nonempty because it contains the quasi-medians of $x,u,v$. Clearly, we can construct all sets $S_x, x\in V^{+}$ in total polynomial time.

%

\begin{lemma} \label{vertexxx} Let $x\in V^+$. If $S_x$ contains two non-adjacent vertices $x_0,x'_0$, then the monophonic halfspace separation problem for $(A,B)$ has answer ``not''.
If $S_x$ is a clique, then for all complementary halfspaces $H',H''$ with $A\subseteq H'$ and $B\subseteq H''$, either $S_x\cup \{ x\}\subseteq H'$ or $S_x\cup \{ x\}\subseteq H''$.
\end{lemma}

\begin{proof} Let $H',H''$ be a pair of complementary monophonic halfspaces with $A\subseteq H'$ and $B\subseteq H''$ and pick any $x_0,x'_0\in S_x$. First suppose that $x_0\nsim x'_0$.
By Lemma \ref{x_0}, $x_0$ and $x'_0$ are adjacent to all vertices of $\partial A\cup \partial B$. Therefore, $x_0$ and $x'_0$ must belong to different halfspaces $H',H''$, say $x_0\in H'$ and $x'_0\in H''$.
Since $H',H''$ are complementary halfspaces, $x$ belong to one of them, say $x\in H'$. But then $x'_0\in [u,x]\cap H''$ for any $u\in \partial A$, a contradiction. Consequently, $S_x$ is a clique.

Now suppose that $S_x$ contains two vertices $x_0,x'_0$ with $x_0\in H'$ and $x'_0\in H''$. Suppose without loss of generality that $x\in H'$.   Let $uv$ be an edge with $u\in \partial A, v\in \partial B$ such that
$x'_0\in \cm(x,u)\cap \cm(x,v)$ (the existence of such an edge follows since  $x'_0\in S_x$).  Consequently, $x'_0\in \cm(u,x)\cap H''\subset H'\cap H''$, which is impossible.  Thus
the clique $S_x$ is entirely contained in one of the halfspaces $H'$ or $H''$. If say $S_x\subset H'$, then since $x'_0\in \cm(x,v)$, $v\in H''$ and $x'_0\in H'$, we conclude that
$x$ belongs to  $H'$.
\end{proof}

Therefore, we can further suppose that all sets $S_x, x\in  V^{+}$  are cliques.

\begin{lemma} \label{separator} For any $x\in V^{+}$, any path from $x$ to a vertex of $A\cup B$ intersect the set $S_x$.
\end{lemma}

\begin{proof} Suppose that $x$ can be connected to a vertex $y$ of $A$ by a path avoiding $S_x$. Then $x\in V^+\setminus V^+_0$. Since $A\in {\mathfrak M}$, $y$
can be connected by a path inside $A$ with any $u\in \partial A$.
Concatenating those two paths, we get a $(x,u)$-path avoiding $S_x$. Consequently, $x$ can be connected to any vertex of $\partial A\cup \partial B$
by a path avoiding $S_x$, from which we can extract an induced such path. Now, pick $p\in \partial A\cup \partial B$, say $p\in \partial A$,
for which such a connecting induced path $P$ avoiding $S_x$
has minimal length. Let $q$ be a neighbor of $p$ in $\partial B$. Since $x\in V^{+}$, we have $p\notin \cm(x,q)$ (otherwise, $x\in p/q\subset A/B$, contrary with the assumption that $(A,B)$ is shadow-closed).
Thus, $P$ plus the edge $qp$ is not an induced $(x,q)$-path. Let $z\ne p$ be a neighbor of $q$ in $P$. If $z$ is not a neighbor of $p$ in $P$, then we get a contradiction with the
minimality choice of  $p$ and $P$. Consequently,  $z\sim p$ and $z$ is the unique vertex of $P\setminus \{ p\}$ adjacent to $q$. But this implies that $z$ belongs
to $S_x$ because $P$ and the path constituted by the edge $qz$ and the subpath of $P$ between $z$ and $x$ are both induced paths, yielding $z\in \cm(x,p)\cap \cm(x,q)$.
This contradicts that $P\cap S_x=\varnothing$.
\end{proof}

Summarizing, we are lead to the following. The set $V^{+}_0$ is covered by the sets $S_x, x\in V^{+}$, each $S_x$ is a non-empty clique of $G$ and separates $x$ from each of the sets $A$ and $B$.
Each vertex $x_0\in V^{+}_0$ is adjacent to all vertices of
$\partial A\cup \partial B$. 
Finally, if $H',H''$ are complementary monophonic halfspaces with $A\subseteq H'$ and $B\subseteq H''$, 
then each $S_x\cup \{ x\}, x\in V^{+}$ belong either to $H'$ or to $H''$
and if two vertices $x_0,x'_0\in V^{+}_0$, are not adjacent,
then they belong to different halfspaces $H',H''$.

To distribute the vertices of the residue $V^{+}$ to the halfspaces $H'$ and $H''$ or to decide that such a distribution does  not exist, for each $x\in V^{+}$ we define
a binary variable $a_x$; we used a similar method in \cite{ChSe}).
The set of variables $a_x, x\in V^{+}$ satisfies the following conditions:
\begin{enumerate}

\item[(1)] $a_x=a_{x_0}$ for any $x\in  V^{+}$  and $x_0\in S_x$;
\item[(2)] $a_{x_0}\ne a_{y_0}$ if $x_0,y_0\in V^{+}_0$ and $x_0$ and $y_0$ are not adjacent.
\end{enumerate}
We define a 2-SAT formula $\Phi$ by replacing every constraint of the form
$a=b$ by two clauses $(a \vee \overline{b})$ and $(\overline{a}\vee b)$ and every constraint of the form $a\ne b$ by
two clauses $(a\vee b)$ and $(\overline{a}\vee \overline{b})$.  $\Phi$ can be solved in linear time using \cite{AsPlTa}. From the previous discussion we conclude that if $(H',H'')$ is a solution
for the halfspace separation problem, then the sets $H'\cap V^{+}$ and $H''\cap V^{+}$ satisfy the constraints (1) and (2), thus
setting $\alpha(a_x)=0$ if $x\in H'\cap V^{+}$ and $\alpha(a_x)=1$ if $x\in H''\cap V^{+}$ (or, vice-versa, $\alpha(a_x)=1$ if $x\in H'\cap V^{+}$ and $\alpha(a_x)=0$ if $x\in H''\cap V^{+}$),
we will get a satisfying assignment $\alpha$ of $\Phi$. Therefore, if the 2-SAT formula $\Phi$ is not satisfiable, then the halfspace separation problem for $(A,B)$  has answer ``not''.
Conversely, suppose that $\Phi$ has a satisfying assignment $\alpha$. Set $A^+:=\{ x\in V^{+}: \alpha(a_x)=0\}$ and $B^+:=\{ x\in V^{+}: \alpha(a_x)=0\}$ and let $H'=A\cup A^+, H''=B\cup B^+$.

\begin{lemma} $H'$ and $H''$ are complementary monophonic halfspaces of $G$ separating $A$ and $B$.
\end{lemma}

\begin{proof} By condition (1), $x\in V^{+}$ belongs to $H'$ (respectively, to $H''$) if and only of $S_x$ belongs to $H'$ (respectively, to $H''$). By condition (2), the set
$V^{+}_0$ is partitioned into two cliques $V'=V^{+}_0\cap H'$ and $V''=V^{+}_0\cap H''$. Suppose by way of contradiction that $H'$ is not monophonically convex. Then $H'$ contains two non-adjacent vertices
$x,y$ which can be connected outside $H'$ by an induced path $P$. We can suppose without loss of generality that all inner vertices of $P$ belong to $H''$. Since $H'=A\cup A^+$ and $A$ is monophonically convex,
either $x,y\in A^+$ or $x\in A^+$ and $y\in A$.

First suppose that $x\in A^+$ and $y\in A$. By Lemma \ref{separator}, $P$ intersect $S_x$ in a vertex $x_0$. By condition (1), $x_0$ belongs to $H'$, contrary to our assumption that all
vertices of $P\setminus \{ x,y\}$ belong to $H''$.

Now suppose that $x,y\in A^+$. Let $z$ be the neighbor of $x$ in $P$ and $z'$ be a neighbor of $y$ in $P$.  We assert that $z,z'$ belongs to $V^{+}_0\cup B$. Let $x_0\in S_x$ such that
$x_0uv$ is a quasi-median of $x,u,v$, where $u\in \partial A, v\in \partial B$, and $u\sim v$ (if $x\in V^+_0$, then $x=x_0$). By Lemma \ref{monophonicquasimedian},
$x_0\sim u,v$, thus $x_0$ belongs to $S_x$.   
Let $Q$ be a shortest path from $x$ to $x_0$ and let $x'$ be any vertex of $Q$. Since $x_0\in [x',u]\cap [x',v]$, $x_0uv$ is also a quasi-median of $x',u,v$, thus $x_0\in S_{x'}$.
Applying condition (1) for each of the pairs $x,x_0$ and $x',x_0$, we conclude that each vertex $x'$ of the path $Q$ belongs to $H'$ and thus $z$ is not a vertex of $Q$.
If $z$ does not belong to  $V^{+}_0\cup B$, then $z$ is not adjacent to $u$ and $v$. Then from the path consisting of the edge $zx$, the path $Q$, and the edge $x_0u$ we extract an induced
$(z,u)$-path $Q'$ passing via $x_0$. Analogously, from the path consisting of the edge $zx$, the path $Q$, and the edge $x_0v$ we can extract an induced $(z,v)$-path passing via $x_0$.
Consequently, $x_0\in \cm(z,u)\cap \cm(z,v)$, showing that $x_0\in S_z$. By condition (1), we get $a_z=a_{x_0}=a_x$, thus $z$ belongs to $H'$, contrary to our assumption. Thus
$z\in V^{+}_0\cup B$. Analogously, one can show that $z'\in V^{+}_0\cup B$.

We assert that $z,z'\in V^{+}_0\cup B$ implies $x,y\in V^+_0$. If $z\in B$, then $x\in V^+_0$ by the definition of $V^+_0$. 
If $z\in V^+_0$,
then $z\sim u,v$. If $x\in V^+\setminus V^+_0$, then $x\nsim u,v$, thus  $z\in [x,u]\cap [x,v]$, and $zuv$ is a quasi-median of $x,u,v$. Consequently, $z\in S_x$, contrary to
the assumption that $z\in H''$. Therefore, $x\in V^+_0$. Analogously, if $z'\in V^+_0\cup B$, then $y\in V^+_0$. Consequently, $x,y\in V^+_0$. Since $x,y\in H'\cap V^+_0=V'$ and $V'$ is a clique, we get $P=(x,y)$.
This concludes the proof.
\end{proof}

Putting altogether, we obtain the following result:

\begin{theorem} \label{monophonic} Given two disjoint sets $A,B$ of a graph $G$ it can be decided in polynomial time if $A$ and $B$ can be separated by complementary monophonic halfspaces.
\end{theorem}

\begin{remark} Theorem \ref{monophonic} was very recently obtained in \cite{ElNouVi} and \cite{BrEsTh} with approaches which at some points intersect ours. However our method and proofs are different 
(and independent of those two papers) and our approach is shorter.
\end{remark}


\end{document}